\documentclass[11pt,reqno,a4paper,dvipsnames]{amsart}
\usepackage[utf8]{inputenc}
\usepackage[dvipsnames]{xcolor}

\title[
Carleson operators on doubling metric measure spaces]{
Carleson operators on doubling metric measure spaces}
\author[Becker]{Lars Becker}
\address{Mathematical Institute,
	University of Bonn,
	Endenicher Allee 60, 53115, Bonn,
	Germany}
\email{becker@math.uni-bonn.de}

\author[van Doorn]{Floris van Doorn}
\address{Mathematical Institute,
	University of Bonn,
	Endenicher Allee 60, 53115, Bonn,
	Germany}
\email{vdoorn@math.uni-bonn.de}

\author[Jamneshan]{Asgar Jamneshan}
 \address{Mathematical Institute,
	University of Bonn,
	Endenicher Allee 60, 53115, Bonn,
	Germany}
\email{ajamnesh@math.uni-bonn.de}

\author[Srivastava]{Rajula Srivastava}
\address{Mathematical Institute,
	University of Bonn,
	Endenicher Allee 60, 53115, Bonn,
	Germany, and
 \newline Max Planck Institute for Mathematics, Vivatsgasse 7,
53111, Bonn,
Germany.}
\email{rajulas@math.uni-bonn.de}

\author[Thiele]{Christoph Thiele}
 \address{Mathematical Institute,
	University of Bonn,
	Endenicher Allee 60, 53115, Bonn,
	Germany}
\email{thiele@math.uni-bonn.de}

\date{\today}

\usepackage{mathtools}
\usepackage{amssymb}
\usepackage{amsthm}
\usepackage{amsmath}
\usepackage{graphicx}
\usepackage{xcolor}
\usepackage{enumerate} 
\usepackage[hidelinks, colorlinks=false]{hyperref}
\usepackage{cleveref}

\usepackage[style=trad-alpha]{biblatex}
\bibliography{main.bib}

\theoremstyle{plain}
\newtheorem{theorem}{Theorem}
\newtheorem{lemma}[theorem]{Lemma}
\newtheorem{proposition}[theorem]{Proposition}

\theoremstyle{definition}

\numberwithin{theorem}{section}
\numberwithin{equation}{section}

\newcommand{\R}{\mathbb{R}}
\newcommand{\C}{\mathbb{C}}

\DeclareMathOperator{\dens}{\operatorname{dens}}

\DeclareMathOperator{\Lip}{\operatorname{Lip}}

\def \fp {\mathfrak p}
\def \fP {\mathfrak P}
\def \fu {\mathfrak u}
\def \fU {\mathfrak U}

\def \fq {\mathfrak q}

\def\fT{\mathfrak T}
\def\fL{\mathfrak L}
\def\fC{\mathfrak C}
\def\pc{\mathrm{c}}
\def\ps{\mathrm{s}}

\def \fc{\Omega}

\def\scI{\mathcal{I}}
\def \tQ{{Q}}
\def\mfa{\vartheta}
\def\mfb{\theta}
\def\Mf{\Theta}
\def\fcc{\mathcal{Q}}

\ExplSyntaxOn
\NewDocumentCommand{\uses}{m}
  {\clist_map_inline:nn{#1}{\vphantom{\ref{##1}}}%
  \ignorespaces}
\ExplSyntaxOff

\ExplSyntaxOn
\ExplSyntaxOff
\newcommand{\lean}[1]{}
\newcommand{\leanok}{}
\newcommand{\proves}[1]{}
\begin{document}

\begin{abstract}
Doubling metric measure spaces provide a natural framework for singular integral operators.
In contrast, the study of maximally modulated singular integral operators, the so-called Carleson operators, has largely been limited to Euclidean space with modulation functions such as polynomials
defined by algebraic means.
We present a general axiomatic approach to modulation functions on doubling metric measure spaces and
prove $L^p$ bounds for the corresponding
Carleson operators in \Cref{metric-space-Carleson}
and \Cref{linearised-metric-Carleson}.
This generalizes classical and modern
results on Carleson operators.
In addition to the proofs presented here, our main results  have been computer verified using the language Lean and the library mathlib,
as documented in the sibling communication \cite{becker2025blueprintformalizationcarlesonstheorem}.

\end{abstract}

\maketitle

\tableofcontents


\section{Introduction} 

Doubling metric measure spaces and the more general spaces of homogeneous type of Coifman and Weiss \cite{MR0499948} are a natural environment to define singular integrals.
Thanks to the work of Macias and Segovia \cite{MaciasSegovia}, one can pass back and forth between doubling metric measure spaces and spaces of homogeneous type.
We refer to the textbook \cite{stein-book} for an account of these spaces.

A singular integral operator $S$ acts on suitable functions $f$ on the doubling metric measure space $X$ and satisfies
\begin{equation*}
    \int Sf(x) g(x)\, \mathrm{d}\mu(x):=
\int\int  K(x,y) f(y) g(x) \mathrm{d}\mu(y)\mathrm{d}\mu(x)
\end{equation*}
for functions $f$ and $g$ with disjoint supports,
where the kernel $K$ satisfies Calder\'on-Zygmund estimates
such as
\eqref{eqkernel-size}, \eqref{eqkernel-y-smooth}. Estimates
\eqref{eqkernel-size}, \eqref{eqkernel-y-smooth}, are
in terms of the metric distance $\rho$ and the measure distance, that is the measure $\mu(B)$ of the smallest ball $B$ containing two points, and have a natural scaling behaviour.

Much of the theory of singular integrals
consists of conditional results in the following sense. Conditioned on boundedness of a singular integral operator
in the Hilbert space  $L^2(X)$,
the theory provides other estimates such as $L^p$ bounds, weighted bounds, and sparse bounds, for the singular integral as well as for related operators
such as maximal or nontangential maximal truncations.
The initial bound often arises externally from Hilbert space techniques in the setting of the given singular integral. More abstractly, one has
$S(1)$ and $S(b)$
theorems that obtain the general $L^2(X)$ bound from testing
the $L^2(X)$ estimate on subsets
consisting of so-called test functions. In the setting of spaces of homogeneous type, \cite{christ1990b}
provides a general $S(b)$ theorem.

A variant of singular integral
operators  beyond the basic variety are the maximally modulated singular integrals $T$, also called Carleson operators. For disjointly supported $f$ and $g$, they satisfy
\begin{equation*}
\int Tf(x) g(x)\, \mathrm{d}\mu=\int\sup_{\mfa\in \Mf}\left|\int  K(x,y) e^{i\mfa(y)}f(y) g(x) \mathrm{d}\mu(y)\right|\,\mathrm{d}\mu(x)
\end{equation*}
for a set $\Mf$ of modulation functions. The classical  Carleson operator
is defined with the Hilbert kernel $K(x,y)=(x-y)^{-1}$
on the real line and the set $\Mf$ of linear functions. Estimates for the classical Carleson operator
are provided by the Carleson-Hunt theorem and are closely related to the celebrated theorem of Carleson \cite{carleson} on almost everywhere convergence of Fourier series of functions in $L^2$.
Besides the translation and dilation symmetry
 of the Hilbert transform that is typical
for singular integral theory, the classical Carleson operator
exhibits an additional modulation symmetry. This modulation symmetry mandates
techniques  that are often referred to as time-frequency analysis and go beyond the core singular integral techniques.
Further examples of Carleson operators in the literature include the Stein-Wainger \cite{stein-wainger} operator
with polynomial modulations without linear term, thereby avoiding time-frequency analysis, Lie's quadratic Carleson operator \cite{lie-quadratic} with quadratic
polynomials including linear terms, and Lie's general polynomial Carleson operator \cite{lie-polynomial}
and its further generalization \cite{zk-polynomial}. Mnatsakanyan \cite{mnatsakanyan} has encountered a Carleson operator with non-polynomial albeit still algebraically explicit modulations.

The purpose of the present paper is to generalize all these results by defining classes $\Mf$ of modulation functions axiomatically on doubling metric measure spaces and proving bounds for the corresponding Carleson operators.
The axioms are in terms of a family of metrics  on $\Mf$
parameterized by balls in $X$. The metric related to a ball $B$ controls the relative oscillation of two modulation functions
on $B$
as in \eqref{osccontrol}, and in many instances it can be chosen equal to the left-hand side
of \eqref{osccontrol}.
The axioms demand  a number of doubling properties
of the metrics
very much in the spirit of doubling metric measure spaces, namely
\eqref{firstdb},  \eqref{seconddb}, and \eqref{thirddb}.
 One additional axiom \eqref{eq-vdc-cond} that we call the cancellative axiom is of somewhat more technical nature
 and replaces techniques of partial integration in the Euclidean setting.

 Our main Theorem \ref{metric-space-Carleson} is then of conditional nature roughly in
 the sense described above.
 It takes as assumption an $L^2(X)$ estimate on an operator related to the singular intgral, the
 maximally truncated
 non-tangential operator $T_*$ defined in
 \eqref{def-tang-unm-op}.
Bounds for this operator follow for standard Calder\'on-Zygmund kernels that are regular in both variables from bounds
in
$L^2(X)$ of the singular integral itself.
We pose bounds for this stronger operator as assumption,
because we only demand regularity of the
kernel $K$ in one of two variables.
We call such $K$  one-sided kernels. They
are natural for Carleson operators in view of the maximal construction in the other variable.
The non-tangential construction replaces regularity
in the other variable.
We also formulate a second theorem, Theorem \ref{linearised-metric-Carleson},
concluding boundedness for the  so-called linearized Carleson operator with the supremum over modulations replaced by a choice function $\tQ$.
Having such a choice function fixed,
one can demand a weaker
hypothesis on
an operator
\eqref{def-lin-star-op}
adapted to the specific choice function $\tQ$.
This more technical but stronger version is necessary in some applications such as a possible deduction of the Walsh Carleson theorem \cite{MR217510}, where the kernel $K$ is closely linked to the linearizing function $\tQ$ and one only has the weaker hypothesis.

We proceed to introduce the formal setup for our main theorems. The setup is intentionally worded completely parallel to the  sibling communication \cite{becker2025blueprintformalizationcarlesonstheorem}, see also the end
of the introduction for further discussion of the relationship between the sibling communications.

We carry a multi purpose parameter, a natural number
\begin{equation*}
    a\ge 4
\end{equation*} in our notation.
As $a$ gets larger, both the
hypotheses and the conclusions of the main theorems will become weaker.

A doubling metric measure space $(X,\rho,\mu, a)$ is a complete
and locally compact metric space $(X,\rho)$
equipped with a non-zero locally finite Borel measure $\mu$ that satisfies the doubling condition that for all $x\in X$ and all $R>0$ we have
\begin{equation}\label{doublingx}
    \mu(B(x,2R))\le 2^a\mu(B(x,R))\,,
\end{equation}
where we have denoted by $B(x,R)$ the open ball of radius $R$ centred at $x$:
\begin{equation*}
 B(x,R):=\{y\in X: \rho(x,y)<R\}. \end{equation*}

A collection $\Mf$ of real valued continuous functions on the doubling metric measure space $(X,\rho,\mu,a)$ is called compatible, if there is a point $o\in X$ where all the functions are equal to $0$ and there exists for each ball $B \subset X$ a metric $d_B$ on $\Mf$ such that the following five properties \eqref{osccontrol}, \eqref{firstdb}, \eqref{monotonedb}, \eqref{seconddb}, and \eqref{thirddb} are satisfied. For every ball $B \subset X$
\begin{equation}\label{osccontrol}
    \sup_{x,y\in B}|\mfa(x)-{\mfa(y)}- \mfb(x)+{\mfb(y)}| \le d_{B}(\mfa,\mfb)\,.
\end{equation}
For any two balls $B_1=B(x_1,R)$, $B_2= B(x_2,2R)$ in $X$ with $x_1\in B_2$ and any $\mfa,\mfb\in \Mf$,
\begin{equation}\label{firstdb}
    d_{B_2}(\mfa,\mfb)\le 2^a d_{B_1}(\mfa,\mfb) .
\end{equation}
For any two balls $B_1, B_2$ in $X$ with $B_1 \subset B_2$ and any $\mfa, \mfb \in \Mf$
\begin{equation}\label{monotonedb}
    d_{B_1}(\mfa,\mfb) \le d_{B_2}(\mfa, \mfb)
\end{equation}
and for any two balls
$B_1=B(x_1,R)$, $B_2= B(x_2,2^aR)$
with $B_1\subset B_2$, and $\mfa,\mfb\in \Mf$,
\begin{equation}\label{seconddb}
    2d_{B_1}(\mfa,\mfb)
\le d_{B_2}(\mfa,\mfb) .
\end{equation}
For every ball $B$ in $X$ and every $d_B$-ball $\tilde B$ of radius $2R$ in $\Mf$, there is a collection $\mathcal{B}$ of
at most $2^a$ many $d_B$-balls of radius $R$ covering $\tilde B$, that is,
\begin{equation}\label{thirddb}
    \tilde B\subset \bigcup \mathcal{B}.
\end{equation}

Further, a compatible collection $\Mf$ is called cancellative, if
for any ball $B$ in $X$ of radius $R$, any Lipschitz function $\varphi: X\to \C$
supported on $B$, and any $\mfa,\mfb\in \Mf$ we have
\begin{equation}
    \label{eq-vdc-cond}
    \left|\int_B e(\mfa(x)-{\mfb(x)}) \varphi(x) \mathrm{d}\mu(x)\right|\le 2^a \mu(B)\|\varphi\|_{\Lip(B)}
(1+d_B(\mfa,\mfb))^{-\frac{1}{a}},
\end{equation}
where $\|\cdot\|_{\Lip(B)}$ denotes the inhomogeneous Lipschitz norm on $B$:
$$
    \|\varphi\|_{\Lip(B)} = \sup_{x \in B} |\varphi(x)| + R \sup_{x,y \in B, x \neq y} \frac{|\varphi(x) - \varphi(y)|}{\rho(x,y)}\,.
$$

A one-sided Calder\'on--Zygmund kernel $K$ on the doubling metric measure space $(X, \rho, \mu, a)$ is a measurable function
\begin{equation*}
  K:X\times X\to \mathbb{C}
\end{equation*}
such that for all $x,y',y\in X$ with $x\neq y$, we have
\begin{equation}\label{eqkernel-size}
    |K(x,y)| \leq \frac{2^{a^3}}{V(x,y)}
\end{equation}
and if $2\rho(y,y') \leq \rho(x,y)$, then
\begin{equation}
  \label{eqkernel-y-smooth}
  |K(x,y) - K(x,y')| \leq \left(\frac{\rho(y,y')}{\rho(x,y)}\right)^{\frac{1}{a}}\frac{2^{a^3}}{V(x,y)},
\end{equation}
where \[V(x,y):=\mu(B(x,\rho(x,y))).\]
Define the maximally truncated non-tangential singular integral $T_{*}$ associated with $K$ by
\begin{equation}
    \label{def-tang-unm-op}
    T_{*}f(x):=\sup_{R_1 < R_2} \sup_{\rho(x,x')<R_1} \left|\int_{R_1< \rho(x',y) < R_2} K(x',y) f(y) \, \mathrm{d}\mu(y) \right|\,.
\end{equation}
We define the generalized Carleson operator $T$ by
\begin{equation}
    \label{def-main-op}
    Tf(x):=\sup_{\mfa\in\Mf} \sup_{0 < R_1 < R_2}\left| \int_{R_1 < \rho(x,y) < R_2} K(x,y) f(y) e(\mfa(y)) \, \mathrm{d}\mu(y) \right|\, ,
\end{equation}
where $e(r)=e^{ir}$.

Our main result is the following restricted weak type estimate for $T$ in the range $1<q\le 2$, which by interpolation techniques recovers $L^q$ estimates for the open range
$1<q<2$.
\begin{theorem}[metric space Carleson]
\label{metric-space-Carleson}
\leanok
\lean{metric_carleson}
    For all integers $a \ge 4$ and real numbers $1<q\le 2$
    the following holds.
    Let $(X,\rho,\mu,a)$ be a doubling metric measure space. Let $\Mf$ be a
    cancellative compatible collection of functions and let $K$ be a one-sided Calder\'on--Zygmund kernel on $(X,\rho,\mu,a)$. Assume that for every bounded measurable function $g$ on $X$ supported on a set of finite measure we have
    \begin{equation}\label{nontanbound}
        \|T_{*}g\|_{2} \leq 2^{a^3} \|g\|_2\,,
    \end{equation}
    where $T_{*}$ is defined in
\eqref{def-tang-unm-op}.
    Then for all Borel sets $F$ and $G$ in $X$ and
    all Borel functions $f:X\to \C$ with
    $|f|\le \mathbf{1}_F$, we have, with $T$ defined in \eqref{def-main-op},
    \begin{equation}
        \label{resweak}
        \left|\int_{G} T f \, \mathrm{d}\mu\right| \leq \frac{2^{443a^3}}{(q-1)^6} \mu(G)^{1-\frac{1}{q}} \mu(F)^{\frac{1}{q}}\, .
    \end{equation}
\end{theorem}

For a Borel function $\tQ:X\to \Mf$, and $\mfa \in \Mf$ and $x\in X$ define
\begin{equation*}
    R_{\tQ}(\mfa,x)=\sup\{r>0:d_{B(x,r)}(\mfa, \tQ(x))<1\}
\end{equation*}
and define further
\begin{equation}
    \label{def-lin-star-op}
    T_{\tQ}^\mfa f(x):=\sup_{R_1<R_2} \ \sup_{\rho(x,x')<R_1}
    \left|\int_{R_1< \rho(x',y) < \min\{R_2,R_{\tQ}(\mfa,x')\}} K(x',y) f(y) \, \mathrm{d}\mu(y) \right|\,.
\end{equation}
Define further the linearized generalized Carleson operator $T_\tQ$ by
\begin{equation}
    \label{def-lin-main-op}
    T_\tQ f(x):= \sup_{0 < R_1 < R_2}\left| \int_{R_1 < \rho(x,y) < R_2} K(x,y) f(y) e(\tQ(x)(y)) \, \mathrm{d}\mu(y) \right|\,,
\end{equation}
where again $e(r)=e^{ir}$.

\begin{theorem}[linearised metric Carleson]
\label{linearised-metric-Carleson}
\leanok
\lean{linearized_metric_carleson}
    For all integers $a \ge 4$ and real numbers $1<q\le 2$
    the following holds.
    Let $(X,\rho,\mu,a)$ be a doubling metric measure space. Let $\Mf$ be a
    cancellative compatible collection of functions.
    Let $\tQ:X\to \Mf$ be a Borel function with finite range.
    Let $K$ be a one-sided Calder\'on--Zygmund kernel on $(X,\rho,\mu,a)$. Assume that for every $\mfa\in \Mf$ and every bounded measurable function $g$ on $X$ supported on a set of finite measure we have
    \begin{equation}\label{linnontanbound}
        \|T_{\tQ}^\mfa g\|_{2} \leq 2^{a^3} \|g\|_2\,,
    \end{equation}
    where $T_{\tQ}^\mfa$ is defined in \eqref{def-lin-star-op}.
    Then for all bounded Borel sets $F$ and $G$ in $X$ and
    all Borel functions $f:X\to \C$ with
    $|f|\le \mathbf{1}_F$, we have, with $T_\tQ$ defined in \eqref{def-lin-main-op},
    \begin{equation}
        \label{linresweak}
        \left|\int_{G} T_\tQ f \, \mathrm{d}\mu\right| \leq \frac{2^{443a^3}}{(q-1)^6} \mu(G)^{1-\frac{1}{q}} \mu(F)^{\frac{1}{q}}\, .
    \end{equation}
\end{theorem}

There is extensive literature on Carleson operators. We mention selected  results with particular relevance to the approach in this paper and with focus on the variety
of classes of modulation functions.
The
time frequency analysis required
to estimate many Carleson operators goes back to Carleson's seminal paper  \cite{carleson} on convergence and growth of Fourier series. Shortly after, this was extended to the dyadic case with Walsh modulations \cite{MR217510}.
The first actual $L^p$ bounds for the classical Carleson operator were observed by Hunt
\cite{MR238019}.
Somewhat dual approaches to Carleson operators were given in
\cite{fefferman}, \cite{lacey-thiele}.
In particular, the approach by Fefferman \cite{fefferman} has found applications in \cite{lie-polynomial} and \cite{zk-polynomial},
which were inspirational for the present paper. The possibility of more general classes of modulation
functions was elucidated
in the  work of \cite{mnatsakanyan}, who proves a Carleson-type theorem for the Malmquist-Takenaka series, which leads to modulation functions related to Blaschke products. A generalization of \eqref{def-main-op} from the previously mentioned Euclidean setting into the anisotropic setting that was suggested in \cite{zk-polynomial} is included in our theory.

Recent interest in polynomial Carleson operators was sparked by extensions of the theory of Stein and Wainger \cite{stein-wainger}.
Among these are maximally modulated Radon transforms
initiated by the work of Pierce and Yung \cite{MR3945729},  discrete analogs
of the results of Stein and Wainger as in
\cite{MR4422211},
\cite{MR4630911}, and
monomial modulation functions with fractional powers avoiding resonances in
\cite{MR4149067}.

The quest to extend the $TT^*$ methods of Stein and Wainger towards Carleson operators which
need time frequency analysis led
to various threads.
These include estimates involving restricted suprema, \cite{MR3708001}
and
results for a simplified singular kernel motivated by Radon transforms
 \cite{ramos}.
In
\cite{MR4797542}, \cite{gaitan2024nonzerozerocurvaturetransition}, \cite{hsu2024carlesonradontransformthenonresonant},
time frequency analysis came in through a hybrid construction between the bilinear Hilbert transform and Carleson operators, albeit with non-resonant frequencies.
A basic result with resonant modulations was established in \cite{becker2023degree}.

This paper has a sibling communication containing a computer verification
of Theorems \ref{metric-space-Carleson} and \ref{linearised-metric-Carleson} using the language Lean and the library mathlib.
More precisely, the ultimate purpose of the sibling communication is to provide a computer verification of part of the classical Carleson theorem. It does so as an application of Theorem \ref{metric-space-Carleson}, which it therefore also proceeds to verify. The axiomatic setup of our theorems is very suitable for computer verification, hence the route to the classical theorem via the modern generalization is natural.

The sibling communication has seventeen authors,
twelve authors in addition to the present ones contributing substantially to the coding in Lean. The coding effort took thirteen months.
Over a dozen further experts are acknowledged in the sibling communication for additional contributions.
Only recently such speed of formalization has become possible.
A mathematical statement coded in Lean is correct with certainty as soon as a proof coded in Lean compiles properly. To verify the mathematical statement written in English, the remaining human task is to verify that the translation of the statement into Lean is correct, which amounts to a tiny fraction of work compared with the classical task of checking correctness of a proof.
For easy comparison, the statements of Theorems \ref{metric-space-Carleson} and \ref{linearised-metric-Carleson} together with their assumptions have been formulated here with much detail and in parallel within the
corresponding part of the sibling communication.

The present paper  announces the new \Cref{metric-space-Carleson} and \Cref{linearised-metric-Carleson} and presents a forty page English proof with a level of detail that is common in present research mathematics. This text largely preceeded the formalization and is absent in the sibling communication, but in our
opinion is the central piece of communication in the present culture of mathematics.
The auxiliary statements along with their headings in this paper
match statements in the so-called blueprint, a much longer and more detailed proof text attached to the sibling communication
that was guiding  Lean experts in the distributed effort to produce
the necessary Lean code.


\noindent \textit{Acknowledgement.}
L.B., F.v.D., R.S., and C.T. were funded by the Deutsche Forschungsgemeinschaft (DFG, German Research Foundation) under Germany's Excellence Strategy -- EXC-2047/1 -- 390685813.
L.B. , R.S., and C.T. were also supported by SFB 1060.
A.J. was funded by the Deutsche Forschungsgemeinschaft (DFG, German Research Foundation) - 547294463.

\section{Proof of Metric Space Carleson, overview}
\label{overviewsection}

We first note the quick argument that  \Cref{linearised-metric-Carleson} implies
\Cref{metric-space-Carleson}. By the monotone convergence theorem, we may restrict the supremum in \eqref{def-main-op} to finitely many values of $\vartheta$. Thus assuming $\Mf$ finite, we choose $Q(x)$ to be a maximizer of the supremum in \eqref{def-main-op}, so that the left hand side of \eqref{resweak} equals the left hand side of \eqref{linresweak}.
Observing that assumption \eqref{linnontanbound}  follows from assumption \eqref{nontanbound} as
the operator defined in \eqref{def-tang-unm-op} is larger than the operator \eqref{def-lin-star-op},
\Cref{metric-space-Carleson} follows from \Cref{linearised-metric-Carleson}.

Our proof of \Cref{linearised-metric-Carleson} is a refinement of \cite{zk-polynomial}, which itself is in the tradition of \cite{fefferman} and \cite{lie-polynomial}. The Carleson operator is broken up into
pieces \eqref{definetp}, parameterized by so-called tiles, which are
localized in both metric spaces $X$ and $\Mf$ and have nice dyadic properties thanks to a grid structure constructed in \Cref{grid-existence-tile-structure}.
The bulk of these pieces is regrouped into
collections called trees, which come with a geometric parameter $n$ and two density parameters. In a tree, the modulation parameter $\mfa$ can be assumed constant so that the hypothesis \eqref{linnontanbound} can be applied. It is important to collect almost orthogonal trees together into forests,
and the resulting forest bound is formulated in
\Cref{forest-operator}.

Outside the bulk, we end up with error terms either collected into antichains
or thrown into negligible exceptional sets.
Antichains are almost orthogonal collections of individual tiles and estimated in
\Cref{antichain-operator} without reference to the assumption \eqref{linnontanbound}.

The pair $(X,\Mf)$ has enough of the geometric properties of the Euclidean phase plane to apply adaptions of the tricks in the Euclidean setting to obtain an efficient decomposition into
forests and antichains formulated in
\Cref{prop-decomposition},
so that one obtains the desired bounds when summing the bounds from
Propositions \ref{forest-operator}
and \ref{antichain-operator}
over the various parameter values.

Each of the above mentioned propositions as well as
the short auxiliary \Cref{Holder-van-der-Corput}
is proved in its own section in this paper.
These sections can be read independently of each other, as all necessary definitions and statements used across the sections are formulated in the rest of the present section. The strong modularity, at the expense of a more technical overview section, is an accomplishment of the present paper
and important for the distributed Lean coding effort.



\subsection{Choosing parameters and preliminary reductions}

We prove Theorem \ref{linearised-metric-Carleson}.
Let $a, q$ be given as in \Cref{linearised-metric-Carleson}. Let a doubling metric measure space $(X,\rho,\mu, a)$, a cancellative compatible collection $\Mf$ of functions on $X$, and a point $o\in X$ with $\mfa(o)=0$ for all $\mfa\in \Mf$ be given. Let further a Borel measurable function $\tQ:X\to \Mf$ with finite range, a one-sided Calder\'on--Zygmund kernel $K$ on $X$ so that for every $\mfa\in \Mf$ the operator $T_{\tQ}^{\mfa}$ defined in \eqref{def-lin-star-op}
satisfies \eqref{linnontanbound}, Borel sets $F, G$ with finite measure, and a measurable function $f$ with $f \le \mathbf{1}_F$ be given.

We choose parameters
\begin{equation}\label{defineD}
    D:= 2^{100 a^2}\, ,
\end{equation}
\begin{equation*}
    \kappa:= 2^{-10a}\,,
\end{equation*}
and
\begin{equation}
    \label{defineZ}
    Z := 2^{12a}\,.
\end{equation}

Let
 $\psi:\R \to \R$ be the unique compactly supported, piece-wise linear, continuous function with corners at $\frac 1{4D}$, $\frac 1{2D}$, $\frac 14$, and $\frac 12$ that satisfies
 \begin{equation*}
    \sum_{s\in \mathbb{Z}} \psi(D^{-s}x)=1
\end{equation*}
for all $x>0$. This function vanishes outside $[\frac1{4D},\frac 12]$, is constant one on
$[\frac1{2D},\frac 14]$, and is Lipschitz
with constant $4D$.
For $s\in\mathbb{Z}$, we define
\begin{equation*}
    K_s(x,y):=K(x,y)\psi(D^{-s}\rho(x,y))\,,
\end{equation*}
so that for each $x, y \in X$ with $x\neq y$ we have
$$K(x,y)=\sum_{s\in\mathbb{Z}}K_s(x,y).$$
As a consequence of \eqref{eqkernel-size} and \eqref{eqkernel-y-smooth}, the functions $K_s$ satisfy
\begin{equation}
   \label{eq-Ks-size}
    |K_s(x,y)|\le \frac{2^{102 a^3}}{\mu(B(x, D^{s}))}\,
\end{equation}
and
\begin{equation}
    \label{eq-Ks-smooth}
    |K_s(x,y)-K_s(x, y')|\le \frac{2^{127a^3}}{\mu(B(x, D^{s}))}
    \left(\frac{ \rho(y,y')}{D^s}\right)^{\frac 1a}\,.
\end{equation}
Furthermore, if $K_s(x,y)\neq 0$, then
\begin{equation}\label{supp-Ks}
  \frac{1}{4} D^{s-1} \leq \rho(x,y) \leq \frac{1}{2} D^s\,.
\end{equation}
Up to an error that is controlled by the Hardy-Littlewood maximal function, the left hand side of \eqref{linresweak} is dominated by
\[
    \int_G \sup_{\sigma_1 \le \sigma_2} \left|\sum_{s={\sigma_1}}^{{\sigma_2}} \int K_s(x,y) f(y) e(\tQ(x)(y)) \, \mathrm{d}\mu(y)\right|  \mathrm{d}\mu(x).
\]
By monotone convergence as $S \to \infty$, we may and do restrict the supremum to finitely many values $-S \le \sigma_1 ,\sigma_2 \le S$.
Define $\sigma_1, \sigma_2: X \to [-S, S]$ to be two measurable functions selecting a maximizing pair for the supremum. By Fatou's lemma, we may further assume that $F, G$ are bounded sets. We may and do increase $S$ so that $F$ and  $G$ are contained in $B(o, \frac{1}{4}D^S)$. Finally, it is enough to prove that there exists $G' \subset G$ with $2\mu(G')\le \mu(G)$ such that
\begin{equation*}
    \int_{G\setminus G'} \left|\sum_{s={\sigma_1}(x)}^{{\sigma_2}(x)} \int K_s(x,y) f(y) e(\tQ(x)(y)) \, \mathrm{d}\mu(y) \right| \mathrm{d}\mu(x)
\end{equation*}
\begin{equation}
    \label{eq-linearized}
    \le \frac{2^{442a^3}}{(q-1)^5} \mu(G)^{1-\frac{1}{q}}
     \mu(F)^{\frac 1 q}\,.
\end{equation}
Indeed, applying this iteratively to $G_0 = G$, $G_1 = G_0'$, $G_2 = G_1'$ and so forth and summing the resulting geometric series yields \eqref{resweak} and \eqref{linresweak}.

With these reductions done, \Cref{linearised-metric-Carleson} will follow once we prove \eqref{eq-linearized}.


\subsection{The dyadic model operator}
We now define notions of grids on $X$ and on $\Theta$, which we will use to define a dyadic model operator for the left hand side of \eqref{eq-linearized}.


A grid structure $(\mathcal{D}, c, s)$ on $X$ consists of a finite collection $\mathcal{D}$ of pairs $(I, k)$ of Borel
sets in $X$ and integers $k \in [-S, S]$, the projection
\[
s\colon \mathcal{D}\to [-S, S], \qquad (I, k) \mapsto k
\]
to the second component which is assumed to be surjective and
called scale function, and a function $c:\mathcal{D}\to X$
called center function such that the five properties
\eqref{coverdyadic}, \eqref{dyadicproperty}, \eqref{subsetmaxcube},
\eqref{eq-vol-sp-cube}, and \eqref{eq-small-boundary} below hold. We call the elements of $\mathcal{D}$ dyadic cubes. By abuse of notation, we will usually write just $I$ for the cube $(I,k)$, and we will write $I \subset J$ to mean that for two cubes $(I,k), (J, l) \in \mathcal{D}$ we have $I \subset J$ and $k \le l$.

For each dyadic cube $I$ and each $-S\le k<s(I)$ we have
\begin{equation}\label{coverdyadic}
    I\subset \bigcup_{J\in \mathcal {D}: s(J)=k}J\, .
\end{equation}
Any two non-disjoint dyadic cubes $I,J$ with $s(I)\le s(J)$ satisfy
\begin{equation}\label{dyadicproperty}
    I\subset J.
\end{equation}
There exists a $I_0 \in \mathcal{D}$ with $s(I_0) = S$ and $c(I_0) = o$
and for all cubes $J \in \mathcal{D}$, we have
\begin{equation}\label{subsetmaxcube}
    J \subset I_0\,.
\end{equation}
For any dyadic cube $I$,
\begin{equation}
    \label{eq-vol-sp-cube}
    c(I)\in B(c(I), \frac{1}{4} D^{s(I)}) \subset I \subset B(c(I), 4 D^{s(I)})\,.
\end{equation}
For any dyadic cube $I$ and any $t$ with $tD^{s(I)} \ge D^{-S}$, recalling $\kappa$ from \eqref{defineD},
\begin{equation}
    \label{eq-small-boundary}
    \mu(\{x \in I \ : \ \rho(x, X \setminus I) \leq t D^{s(I)}\}) \le 2 t^\kappa \mu(I)\,.
\end{equation}

A tile structure $(\fP,\scI,\fc,\fcc,\pc,\ps)$
for a given grid structure $(\mathcal{D}, c, s)$
is a finite set $\fP$ of elements called tiles with five maps
\begin{align*}
    \scI&\colon \fP\to {\mathcal{D}}\\
    \fc&\colon \fP\to \mathcal{P}(\Mf) \\
    \fcc &\colon \fP\to \tQ(X)\\
    \pc &\colon \fP\to X\\
    \ps &\colon \fP\to \mathbb{Z}
\end{align*}
with $\scI$ surjective and $\mathcal{P}(\Mf)$ denoting the power set of $\Mf$ such that the six properties \eqref{injective}, \eqref{eq-dis-freq-cover}, \eqref{eq-freq-dyadic},
\eqref{eq-freq-comp-ball}, \eqref{tilecenter}, and
\eqref{tilescale} below hold.
For each dyadic cube $I$, the restriction of the map $\Omega$ to the set
\begin{equation}\label{injective}
    \fP(I)=\{\fp: \scI(\fp) =I\}
\end{equation}
is injective
and we have the disjoint covering property ($\dot{\cup}$ denotes a disjoint union)
\begin{equation}\label{eq-dis-freq-cover}
    \tQ(X)\subset \dot{\bigcup}_{\fp\in \fP(I)}\fc(\fp).
\end{equation}
For any tiles $\fp,\fq$ with $\scI(\fp)\subset \scI(\fq)$ and $\fc(\fp) \cap \fc(\fq) \neq \emptyset$ we have
\begin{equation} \label{eq-freq-dyadic}
    \fc(\fq)\subset \fc(\fp) .
\end{equation}
For each tile $\fp$,
\begin{equation}\label{eq-freq-comp-ball}
    \fcc(\fp)\in B_{\fp}(\fcc(\fp), 0.2) \subset \fc(\fp) \subset B_{\fp}(\fcc(\fp),1)\,,
\end{equation}
where
\begin{equation*}
    B_{\fp} (\mfa, R) := \{\mfb \in \Mf \, : \, d_{\fp}(\mfa, \mfb) < R\,\} ,
\end{equation*}
and
\begin{equation}\label{defdp}
    d_{\fp} := d_{B(\pc(\fp),\frac 14 D^{\ps(\fp)})}\, .
\end{equation}
We have for each tile $\fp$
\begin{equation}\label{tilecenter}
    \pc(\fp)=c(\scI(\fp)),
\end{equation}
\begin{equation}\label{tilescale}
    \ps(\fp)=s(\scI(\fp)).
\end{equation}

\begin{proposition}[grid existence and tile structure]
    \label{grid-existence-tile-structure}
    There exists a grid structure $(\mathcal{D}, c,s)$. For a given grid structure $(\mathcal{D}, c,s)$, there exists a tile structure
        $(\fP,\scI,\fc,\fcc,\pc,\ps)$.
\end{proposition}

We prove \Cref{grid-existence-tile-structure} in \Cref{christsection}.

We fix a grid structure $(\mathcal{D}, c, s)$ and a tile structure $(\fP,\scI,\fc,\fcc,\pc,\ps)$. We define for $\fp\in \fP$
\begin{equation}\label{defineep}
    E(\fp)=\{x\in \scI(\fp): \tQ(x)\in \fc(\fp) , {\sigma_1}(x)\le \ps(\fp)\le {\sigma_2}(x)\}
\end{equation}
and
\begin{equation}\label{definetp}
    T_{\fp} f(x)= \mathbf{1}_{E(\fp)}(x) \int K_{\ps(\fp)}(x,y) f(y) e(\tQ(x)(y)-\tQ(x)(x))\, \mathrm{d}\mu(y).
\end{equation}
We also introduce, for any collection $\fC \subset \fP$ of tiles, the notation
\[
    T_\fC = \sum_{\fp \in \fC} T_\fp.
\]
The covering properties \eqref{coverdyadic}, \eqref{subsetmaxcube} and \eqref{eq-dis-freq-cover} imply that we can rewrite the left hand side of \eqref{eq-linearized} as
\begin{equation}
    \label{eq-sum-tiles}
  \int_{G \setminus G'} \left|  T_{\fP} f \right| \, \mathrm{d}\mu.
\end{equation}

\subsection{Partition of the set of tiles into forests and antichains}

To estimate \eqref{eq-sum-tiles}, we decompose the collection of tiles $\fP$ into certain subcollections called antichains and forests, and subsequently apply estimates for the operators associated to the subcollections. We proceed by giving relevant definitions and stating the estimates and the decomposition.

We define the relation
\begin{equation}\label{straightorder}
    \fp\le \fp'
\end{equation}
 on $\fP\times \fP$ meaning
$\scI(\fp)\subset \scI(\fp')$ and
$\Omega(\fp')\subset \Omega(\fp)$.
We further define for $\lambda,\lambda' >0$
the relation
\begin{equation}\label{wiggleorder}
    \lambda \fp \lesssim \lambda' \fp'
\end{equation}
on $\fP\times \fP$ meaning
$\scI(\fp)\subset \scI(\fp')$ and
\begin{equation*}
    B_{\fp'}(\fcc(\fp'),\lambda') \subset B_{\fp}(\fcc(\fp),\lambda)\, .
\end{equation*}
Occasionally, one of $\lambda$ or $\lambda'$ is equal to $1$, at which point we omit it in the notation \eqref{wiggleorder}.

We define for a tile $\fp$ and $\lambda>0$
\begin{equation}\label{definee1}
    E_1(\fp):=\{x\in \scI(\fp)\cap G: \tQ(x)\in \fc(\fp)\}\, ,
\end{equation}
\begin{equation}\label{definee2}
    E_2(\lambda, \fp):=\{x\in \scI(\fp)\cap G: \tQ(x)\in B_{\fp}(\fcc(\fp), \lambda)\}\, .
\end{equation}
Given a subset $\fP'$ of $\fP$, we define
$\fP(\fP')$ to be the set of
all $\fp \in \fP$ such that there exist $\fp' \in \fP'$ with $\scI(\fp)\subset \scI(\fp')$. Define the densities
\begin{equation}\label{definedens1}
    {\dens}_1(\fP') := \sup_{\fp'\in \fP'}\sup_{\lambda \geq 2} \lambda^{-a} \sup_{\fp \in \fP(\fP'), \lambda \fp' \lesssim \lambda \fp}
    \frac{\mu({E}_2(\lambda, \fp))}{\mu(\scI(\fp))}\, ,
\end{equation}
\begin{equation}\label{definedens2}
    {\dens}_2(\fP') := \sup_{\fp'\in \fP'}
    \sup_{r\ge 4D^{\ps(\fp)}}
    \frac{\mu(F\cap B(\pc(\fp),r))}{\mu(B(\pc(\fp),r))}\, .
\end{equation}

An antichain is a subset $\mathfrak{A}$
of $\fP$ such that for any distinct $\fp,\fp'\in \mathfrak{A}$ we do not have $\fp\le \fp'$.
The following estimate for operators associated to antichains is proved in \Cref{antichainboundary}.

\begin{proposition}[antichain operator]
\label{antichain-operator}
\leanok
\lean{antichain_operator}

\uses{dens2-antichain,dens1-antichain}
For any antichain $\mathfrak{A} $ and for all $f:X\to \C$ with $|f|\le \mathbf{1}_F$ and all $g:X\to\C$ with $|g| \le \mathbf{1}_G$
\begin{equation} \label{eq-antiprop}
    \Big|\int \overline{g} T_{\mathfrak{A}} f\, \mathrm{d}\mu\Big|
    \le \frac{2^{117a^3}}{q-1} \dens_1(\mathfrak{A})^{\frac {q-1}{8a^4}}\dens_2(\mathfrak{A})^{\frac 1{q}-\frac 12} \|f\|_2 \|g\|_2\, .
\end{equation}
\end{proposition}

Let $n\ge 0$.
An $n$-forest is a pair $(\fU, \mathfrak{T})$
where $\fU$ is a subset of $\fP$
and $\mathfrak{T}$ is a map assigning to
each $\fu\in \fU$ a nonempty set $\fT (\fu)\subset \fP$ called tree
such that the following properties
\eqref{forest1}, \eqref{forest2},
\eqref{forest3},
\eqref{forest4},
\eqref{forest5}, and
\eqref{forest6}
hold.

For each $\fu\in \fU$ and each $\fp\in \fT(\fu)$
we have $\scI(\fp) \ne \scI(\fu)$ and
\begin{equation}\label{forest1}
    4\fp\lesssim \fu.
\end{equation}
For each $\fu \in \fU$ and each $\fp,\fp''\in \fT(\fu)$ and $\fp'\in \fP$
we have
\begin{equation}\label{forest2}
    \fp, \fp'' \in \mathfrak{T}(\fu), \fp \leq \fp' \leq \fp'' \implies \fp' \in \mathfrak{T}(\fu).
\end{equation}
We have
\begin{equation}\label{forest3}
   \|\sum_{\fu\in \fU} \mathbf{1}_{\scI(\fu)}\|_\infty \leq 2^n\,.
\end{equation}
We have for every $\fu\in \fU$
\begin{equation}\label{forest4}
    \dens_1(\fT(\fu))\le 2^{4a + 1-n}\, .
\end{equation}
We have for $\fu, \fu'\in \fU$ with $\fu\neq \fu'$ and $\fp\in \fT(\fu')$ with $\scI(\fp)\subset \scI(\fu)$ that, recalling $Z$ from \eqref{defineZ},
\begin{equation}\label{forest5}
    d_{\fp}(\fcc(\fp), \fcc(\fu))>2^{Z(n+1)}\, .
\end{equation}
We have for every $\fu\in \fU$ and $\fp\in \fT(\fu)$ that
\begin{equation}\label{forest6}
    B(\pc(\fp), 8D^{\ps(\fp)})\subset \scI(\fu)\, .
\end{equation}
When discussing the decomposition of the set of tiles,
we shall with slight abuse of language identify the forest $(\fU,\fT)$ with the set
$\bigcup_{\fu\in \fU}\fT(\fu)$ of tiles.

The following estimate for operators associated to forests is proved in \Cref{treesection}.
\begin{proposition}[forest operator]
\label{forest-operator}
\leanok
\lean{forest_operator}
For any $n\ge 0$ and any $n$-forest $(\fU,\fT)$ we have for all $f: X \to \mathbb{C}$ with $|f| \le \mathbf{1}_F$ and all $g:X\to\C$ with $|g| \le \mathbf{1}_G$
$$
    \Big| \int \overline{g} \sum_{\fu\in \fU} T_{\fT(\fu)} f \, \mathrm{d}\mu\Big|
    \le
    2^{440a^3}2^{-\frac{q-1}{q} n} \dens_2\left(\bigcup_{\fu\in \fU}\fT(\fu)\right)^{\frac{1}{q}-\frac{1}{2}} \|f\|_2 \|g\|_2 \,.
$$
\end{proposition}

The next proposition provides an efficient decomposition into antichains and forests.


\begin{proposition}
\label{prop-decomposition}
There exists a Borel set $G'$ with $2\mu(G') \leq \mu(G)$ such that the following holds. The set $\fP'$ of all tiles $\fp \in \fP$ with $\scI(\fp) \not \subset G'$ can be decomposed as a disjoint union
\begin{equation*}
   \fP' = \bigcup_{n \ge 0} \bigcup_{j= 0}^{12(n+2)^2} \bigcup_{\fu \in \fU_{n,j}} \mathfrak{T}_{n,j}(\fu) \cup \bigcup_{n \ge 0} \bigcup_{j=0}^{Z(n+2)^3} \mathfrak{A}_{n,j}
\end{equation*}
where each $(\fU_{n,j}, \fT_{n,j})$ is an $n$-forest with
\begin{equation*}
    \dens_2\Big(\bigcup_{\fu \in \fU_{n,j}} \fT_{n,j}(\fu)\Big) \le 2^{2a+5} \frac{\mu(F)}{\mu(G)}
\end{equation*}
and each $\mathfrak{A}_{n,j}$ is an antichain with
\begin{equation*}
    \dens_1(\mathfrak{A}_{n,j}) \le 2^{4a+1 - n}
\end{equation*}
and
\begin{equation*}
    \dens_2(\mathfrak{A}_{n,j}) \le 2^{2a+5} \frac{\mu(F)}{\mu(G)}.
\end{equation*}
\end{proposition}

The proof of \Cref{prop-decomposition} is done in \Cref{proptopropprop}.

Proposition
 \ref{prop-decomposition} gives a set $G'$ and a decomposition of the set of relevant tiles. We apply the triangle inequality to
\eqref{eq-sum-tiles}
along this decomposition and estimate the pieces by Propositions  \ref{antichain-operator} and \ref{forest-operator}. The resulting sum adds to less than the  bound \eqref{eq-linearized}, which
completes the proof of \Cref{linearised-metric-Carleson}.

\subsection{The cancellation condition with Hölder regularity}
To bridge the gap between the H\"older regularity condition \eqref{eqkernel-y-smooth} in \Cref{linearised-metric-Carleson}  and the Lipschitz regularity of the testing functions in the  cancellative condition \eqref{eq-vdc-cond},  we follow \cite{zk-polynomial} to formulate a variant of \eqref{eq-vdc-cond}
in the following proposition proved in \Cref{liphoel}.

Define
\begin{equation*}
    \tau:=\frac 1a\, .
\end{equation*}
Define for any open ball $B$ of radius $R$ in $X$ the $L^\infty$-normalized $\tau$-H\"older norm by
\begin{equation*}
    \|\varphi\|_{C^\tau(B)} = \sup_{x \in B} |\varphi(x)| + R^\tau \sup_{x,y \in B, x \neq y} \frac{|\varphi(x) - \varphi(y)|}{\rho(x,y)^\tau}\,.
\end{equation*}

\begin{proposition}[Holder van der Corput]
    \label{Holder-van-der-Corput}
    \leanok
    \lean{holder_van_der_corput}
    \uses{Lipschitz-Holder-approximation}
     Let $z\in X$ and $R>0$ and set $B=B(z,R)$.
     Let $\varphi: X \to \mathbb{C}$ be
     supported on $B$ and satisfy $\|{\varphi}\|_{C^\tau(2B)}<\infty$.
     Let $\mfa, \mfb \in \Mf$. Then
    \begin{equation*}
        \Big|\int e(\mfa(x)-{\mfb(x)})\varphi(x) \mathrm{d}\mu\Big|\le
         2^{7a} \mu(B) \|{\varphi}\|_{C^\tau(2B)}
       (1 + d_{B}(\mfa,\mfb))^{-\frac{1}{2a^2+a^3}}
    \,.
    \end{equation*}
\end{proposition}

\subsection{Monotonicity of cube metrics}
\label{global-auxiliary-lemmas}
We close this section by recording a technical strengthening of the monotonicity \eqref{monotonedb} of the metrics $d_B$.




\begin{lemma}[monotone cube metrics]
    \label{monotone-cube-metrics}
    \leanok
    \lean{Grid.dist_mono, Grid.dist_strictMono}
    Let $(\mathcal{D}, c, s)$ be a grid structure. Denote for cubes $I \in \mathcal{D}$
    \begin{equation}\label{defineIsupo}
         I^\circ := B(c(I), \frac{1}{4} D^{s(I)})\,.
    \end{equation}
    Let $I, J \in \mathcal{D}$ with $I \subset J$.
    Then for all $\mfa, \mfb \in\Mf$ we have
    $$
        d_{I^\circ}(\mfa, \mfb) \le d_{J^\circ}(\mfa, \mfb)\,,
    $$
    and if $I \ne J$ then we have
    $$
        d_{I^\circ}(\mfa, \mfb) \le 2^{-95a} d_{J^\circ}(\mfa, \mfb)\,.
    $$
\end{lemma}

\begin{proof}
    \leanok
    If $s(I) \ge s(J)$ then \eqref{dyadicproperty} and the assumption $I\subset J$ imply $I = J$.

    If $s(J) \ge s(I)+1$, then from the monotonicity \eqref{monotonedb}, \eqref{defineD} and \eqref{seconddb}
    \begin{equation}
    \label{eq-dIJ-est}
        d_{I^\circ}(\mfa, \mfb) \le d_{B(c(I), 4 D^{s(I)})}(\mfa, \mfb) \le 2^{-100a} d_{B(c(I), 4D^{s(J)})}(\mfa, \mfb)\,.
    \end{equation}
    Using \eqref{eq-vol-sp-cube}, together with the inclusion $I \subset J$, we obtain
    $$
        B(c(I), 4 D^{s(J)}) \subset B(c(J), 8 D^{s(J)})\,.
    $$
    Using this together with the monotonicity property \eqref{monotonedb} and \eqref{firstdb} in \eqref{eq-dIJ-est}
    \begin{align*}
        d_{I^\circ}(\mfa, \mfb) \le 2^{-100a} d_{B(c(J), 8D^{s(J)})}(\mfa, \mfb) \le 2^{-95a} d_{B(c(J), \frac{1}{4}D^{s(J)})}(\mfa, \mfb)\,.
    \end{align*}
    This proves the second inequality claimed in the lemma, from which the first follows since $a \ge 4$ and hence $2^{-95a} \le 1$.
\end{proof}

\section{Existence of a tile structure}
\label{christsection}

Here we prove Proposition \ref{grid-existence-tile-structure}.
 Existence of a grid structure  was proved in the generality needed here by Christ \cite[\S 3]{christ1990b}. We record this as:

\begin{lemma}[grid existence]
    \label{grid-existence}
    \leanok
    \lean{grid_existence}
    There exists a grid structure $(\mathcal{D}, c,s)$.
\end{lemma}
The next lemma follows \cite[Lemma 2.12]{zk-polynomial} to construct a tile structure.

\begin{lemma}[tile structure]
    \label{tile-structure}
    \leanok
    \lean{tile_existence}
        For a given grid structure $(\mathcal{D}, c,s)$, there exists a tile structure
        $(\fP,\scI,\fc,\fcc,\pc,\ps)$.
\end{lemma}

\begin{proof}
Choose a grid structure $(\mathcal{D}, c, s)$.
For each $I \in \mathcal{D}$, fix a set $\mathcal{Z} = \mathcal{Z}(I)$ of maximal cardinality such that
\begin{equation*}
    \mathcal{Z} \subset \tQ(X)
\end{equation*}
and such that for any $\mfa, \mfb \in \mathcal{Z}$ with $\mfa\ne \mfb$ we have, recalling definition \eqref{defineIsupo},
\begin{equation*}
    B_{I^\circ}(\mfa, 0.3) \cap B_{I^\circ}(\mfb, 0.3) \cap \tQ(X) = \emptyset\,.
\end{equation*}
Note that this is clearly possible, as $\tQ(X)$ is finite. Maximality implies that
for each $I \in \mathcal{D}$, we have
\begin{equation*}
    \tQ(X) \subset \bigcup_{z \in \mathcal{Z}(I)} B_{I^\circ}(z, 0.7)\,.
\end{equation*}
We define the set of tiles
$$
    \fP = \{(I, z) \ : \ I \in \mathcal{D}, z \in \mathcal{Z}(I)\}\,,
$$
and set
$$\scI((I, z)) = I, \qquad\qquad \fcc((I, z)) = z.$$ We further set $$\ps(\fp) = s(\scI(\fp)),\qquad \qquad \pc(\fp) = c(\scI(\fp)).$$ Then \eqref{tilecenter} and \eqref{tilescale} hold by definition. It remains to construct the map $\Omega$, and verify properties \eqref{eq-dis-freq-cover}, \eqref{eq-freq-dyadic} and
\eqref{eq-freq-comp-ball}.

We first construct an auxiliary map $\Omega_1$. For each $I \in \mathcal{D}$, we enumerate
$$
    \mathcal{Z}(I) = \{z_1, \dotsc, z_M\}\,.
$$
We define {$\Omega_1:\fP \mapsto \mathcal{P}(\Mf) $ as below}. Set
$$
    \Omega_1((I, z_1)) = B_{I^\circ}(z_1, 0.7) \setminus \bigcup_{z \in \mathcal{Z}(I)\setminus \{z_1\}} B_{I^\circ}(z, 0.3)
$$
and then define iteratively
\begin{equation*}
    \Omega_1((I, z_k)) = B_{I^\circ}(z_k, 0.7) \setminus \bigcup_{z \in \mathcal{Z}(I) \setminus \{z_k\}} B_{I^\circ}(z, 0.3) \setminus \bigcup_{i=1}^{k-1} \Omega_1((I, z_i))\,.
\end{equation*}
Then the sets $\Omega_1(\fp)$, $\fp \in \fP(I)$ are clearly pairwise disjoint, and an induction argument on $k$ shows that their union still contains $Q(X)$ and that
\begin{equation}
     \label{eq-omega1-incl}
     B_{\fp}(\fcc(\fp), 0.3) \subset \Omega_1(\fp) \subset B_{\fp}(\fcc(\fp), 0.7)\,.
\end{equation}
Now we are ready to define the function $\Omega$. We define for all $\fp \in \fP(I_0)$
\begin{equation}
    \label{eq-max-omega}
    \fc(\fp) = \Omega_1(\fp)\,.
\end{equation}
For all other cubes $I \in \mathcal{D}$, there exists a unique parent cube $J \supset I$ with $s(J) = s(I) + 1$, and we may assume that $\Omega(\fq)$ is already defined for $\fq \in \fP(J)$. Then we set for $\fp \in \fP(I)$
\begin{equation}
    \label{eq-it-omega}
    \fc(\fp) = \bigcup_{z \in \mathcal{Z}(J) \cap \Omega_1(\fp)} \Omega((J, z)) \cup B_{\fp}(\fcc(\fp),0.2)
    \, .
\end{equation}

We now verify that $(\fP,\scI,\fc,\fcc,\pc,\ps)$ forms a tile structure.

We first verify \eqref{eq-freq-comp-ball}. If $I =I_0$, then \eqref{eq-freq-comp-ball} holds for all $\fp \in \fP(I)$ by \eqref{eq-omega1-incl} and \eqref{eq-max-omega}. Else, we may assume by induction that \eqref{eq-freq-comp-ball} holds for the parent cube $J$ of $I$. Suppose that $\mfa \in \Omega(\fp)$. By \eqref{eq-it-omega}, $\mfa \in B_\fp(\fcc(\fp), 0.2)$ or there exists $z \in \mathcal{Z}(J) \cap \Omega_1(\fp)$ with $\mfa \in \Omega(J,z)$, so by \eqref{eq-omega1-incl}
$$
    d_{I^\circ}(\fcc(\fp),\mfa) \le d_{I^\circ}(\fcc(\fp), z) + d_{I^\circ}(z, \mfa) \le 0.7 + d_{I^\circ}(z, \mfa)\,.
$$
By \Cref{monotone-cube-metrics} and the induction hypothesis, this is estimated by
$$
    \le 0.7 + 2^{-95a} d_{J^\circ}(z,\mfa) \le 0.7 + 2^{-95a}\cdot 1 < 1\,.
$$
This shows the second inclusion in \eqref{eq-freq-comp-ball}, the first holds by definition.

The disjoint covering property \eqref{eq-dis-freq-cover} holds for $I = I_0$, because it holds for $\Omega_1$. The covering part of \eqref{eq-dis-freq-cover} then clearly follows by downward induction for all other cubes $I$ from the definition \eqref{eq-it-omega}. For disjointedness, suppose that $\mfa \in B_\fp(\fcc(\fp), 0.2)$ and $z \in \mathcal{Z}(J)$ is a point with $\mfa \in \Omega((J,z))$. Then
\[
    d_{I^\circ}(\fcc(\fp), z) \le d_{I^\circ}(\fcc(\fp), \mfa) + d_{I^\circ}(\mfa, z) \le 0.2 + 2^{-95a} d_{J^\circ}(\mfa, z) < 0.3.
\]
Thus, by \eqref{eq-omega1-incl}, $z \in \mathcal{Z}(J) \cap \Omega_1(\fp)$. By downward induction, it follows that the sets $\Omega(\fp), \fp \in \fP(I)$ are pairwise disjoint.

We turn to the grid condition \eqref{eq-freq-dyadic}. We first prove it when $\ps(\fq) = \ps(\fp) + 1$. Let $I = \scI(\fp)$. Since $I \subset \scI(\fq)$ and $\ps(\fq) = \ps(\fp) + 1$, we have $\scI(\fq) = J$, where $J$ is the parent cube in the definition \eqref{eq-it-omega}, in particular $\Omega(\fq) \subset \Omega(\fp)$ as needed. If $\ps(\fq) > \ps(\fp) + 1$, the grid condition \eqref{eq-freq-dyadic} follows by induction on the difference $\ps(\fq) - \ps(\fp)$. Indeed, pick $\fp'$ with $\ps(\fp') = \ps(\fq) - 1$ and $\scI(\fp) \subset \scI(\fp')$ and $\Omega(\fq) \subset \Omega(\fp')$. This exists by the covering property \eqref{eq-dis-freq-cover} and the part of \eqref{eq-freq-dyadic} that we already proved. Then $\Omega(\fp') \cap \Omega(\fp) \ne \emptyset$, so by induction $\Omega(\fp') \subset \Omega(\fp)$.
\end{proof}

\section{Organization of the set of tiles}
\label{proptopropprop}

In this section, we prove \Cref{prop-decomposition}.
The overall proof happens in \Cref{subsectilesorg},
where we define the exceptional set $G'$,
the size of which is estimated in
\Cref{subsetexcset},
and decompose the set $\fP$ of tiles
outside the exceptional set into
forests and antichains.
 \Cref{subsec-lessim-aux} contains  auxiliary lemmas needed for the verification of the forest- and antichain properties, which happens in \Cref{subsecforest} and \Cref{subsecantichain}.

\subsection{Organization of the tiles}\label{subsectilesorg}

In the following definitions, $k, n$, and
$j$ will be nonnegative integers.

We start by sorting tiles $\fp$ based on the size of the intersection of the spatial cube $\scI(\fp)$ and its parents with $G$. Define
$\mathcal{C}(G,k)$ to be the set of $I\in \mathcal{D}$
such that there exists a $J\in \mathcal{D}$ with $I\subset J$
and
\begin{equation}\label{muhj1}
   {\mu(G \cap J)} > 2^{-k-1}{\mu(J)}\, ,
\end{equation}
but there does not exist a $J\in \mathcal{D}$ with $I\subset J$ and
\begin{equation}\label{muhj2}
   {\mu(G \cap J)} > 2^{-k}{\mu(J)}\,.
\end{equation}
Let
\begin{equation*}
    \fP(k)=\{\fp\in \fP \ : \ \scI(\fp)\in \mathcal{C}(G,k)\, .
\end{equation*}
We need a notion of density `restricted to $\fP(k)$'. Define for $\fP'\subset \fP(k)$
\begin{equation}
    \label{eq-densdef}
   \dens_k' (\fP'):= \sup_{\fp'\in \fP'}\sup_{\lambda \geq 2} \lambda^{-a} \sup_{\fp \in \fP(k): \lambda \fp' \lesssim \lambda \fp}
    \frac{\mu({E}_2(\lambda, \fp))}{\mu(\scI(\fp))}\,.
\end{equation}
Sorting tiles further by density we define
\begin{equation}
    \label{def-cnk}
    \fC(k,n):=\{\fp\in \fP(k) \ : \
    2^{4a}2^{-n}< \dens_k'(\{\fp\}) \le
    2^{4a}2^{-n+1}\}\,.
\end{equation}
Our goal is to split $\fC(k,n)$ into a controlled number of $n$-forests. The main difficulty is to achieve separation of the involved trees. This is done following a trick of Fefferman \cite{fefferman}.
Define $ {\mathfrak{M}}(k,n)$ to be the set of maximal with respect to $\le$ tiles $\fp \in \fP(k)$ such that
 \begin{equation}\label{defMkn}
    \mu({E_1}(\fp)) > 2^{-n} \mu(\scI(\fp))\,.
 \end{equation}
 We
define for $\fp \in \fC(k,n)$
\begin{equation}\label{defbfp}
     \mathfrak{B}(\fp) := \{ \mathfrak{m} \in \mathfrak{M}(k,n) \ : \ 100 \fp \lesssim \mathfrak{m}\}
\end{equation}
and
\begin{equation*}
       \fC_1(k,n,j) := \{\fp \in \fC(k,n) \ : \ 2^{j} \leq |\mathfrak{B}(\fp)| < 2^{j+1}\}
\end{equation*}
and
\begin{equation*}
       \fL_0(k,n) := \{\fp \in \fC(k,n) \ : \ |\mathfrak{B}(\fp)| <1\}\,.
\end{equation*}

To increase separation of the trees, we remove the $Z(n+1) + 1$ minimal layers of tiles in each $\fC_1(k,n,j)$. Each minimal layer clearly forms an antichain, so this results in at most $Z(n+1)+1$ antichains $\fL_1(k,n,j,l)$, where $ 0 \le l \le Z(n+1)$, and a collection of leftover tiles
\begin{equation}
    \label{eq-C2-def}
    \fC_2(k,n,j):= \fC_1(k,n,j)\setminus \bigcup_{0\le l \le Z(n+1)}
\fL_1(k,n,j,l)\, .
\end{equation}

The remaining tile organization will be relative to
prospective tree tops, which we define now.
Define
\begin{equation}\label{defunkj}
     \fU_1(k,n,j)
\end{equation}
to be the set of all
$\fu \in \fC_1(k,n,j)$ such that
for all $\fp \in \fC_1(k,n,j)$
with $\scI(\fu)$ strictly contained in
$\scI(\fp)$ we have $B_{\fu}(\fcc(\fu), 100) \cap B_{\fp}(\fcc(\fp),100) = \emptyset$.

We first remove the pairs that are outside the immediate reach of any of the prospective tree tops.
Define
\begin{equation*}
\fL_2(k,n,j)
\end{equation*}
to be the set of all $\fp\in \fC_2(k,n,j)$ such that there
does not exist
$\fu\in \fU_1(k,n,j)$
with $\scI(\fp)\neq \scI(\fu)$ and $2\fp\lesssim \fu$.
Define
\begin{equation}
\label{eq-C3-def}
\fC_3(k,n,j):=\fC_2(k,n,j)
  \setminus \fL_2(k,n,j)\, .
\end{equation}

We next remove the $Z(n+1)+1$ maximal layers in $\fC_3(k,n,j)$, resulting in antichains $\fL_3(k,n,j,l)$, where $0 \le l \le Z(n+1)$, and the remainder
\begin{equation}
\label{eq-C4-def}
\fC_4(k,n,j):=\fC_3(k,n,j)
  \setminus \bigcup_{0 \le l \le Z(n+1)} \fL_3(k,n,j,l)\,.
\end{equation}

Finally, we remove tiles close to the boundary of the prospective tree tops, this is to ensure that \eqref{forest6} holds. Define
\begin{equation}
    \label{eq-L-def}
    \mathcal{L}(\fu)
\end{equation}
to be the set of all $I \in \mathcal{D}$ with $I \subset \scI(\fu)$ and $s(I) = \ps(\fu) - Z(n+1) - 1$ and
\begin{equation*}
    B(c(I), 8 D^{s(I)})\not \subset \scI(\fu)\, .
\end{equation*}
Define
\begin{equation*}
    \fL_4(k,n,j)
\end{equation*}
to be the set of all $\fp\in \fC_4(k,n,j)$ such that there exists
$\fu\in \fU_1(k,n,j)$
with $\scI(\fp) \subset \bigcup \mathcal{L}(\fu)$, and define
\begin{equation*}
\fC_5(k,n,j):=\fC_4(k,n,j)
  \setminus \fL_4(k,n,j)\, .
\end{equation*}

We define three exceptional sets.
The first exceptional set $G_1$ takes into account the ratio of the measures of $F$ and $G$.
Define $\fP_{F,G}$ to be the set of all $\fp\in \fP$
with
\begin{equation*}
    \dens_2(\{\fp\})> 2^{2a+5}\frac{\mu(F)}{\mu(G)}\,.
\end{equation*}
Define
\begin{equation*}
    G_1:=\bigcup_{\fp\in \fP_{F,G} }\scI(\fp)\, .
\end{equation*}
For an integer $\lambda\ge 0$, define $A(\lambda,k,n)$ to be the set
of all $x\in X$ such that
\begin{equation}
    \label{eq-Aoverlap-def}
    \sum_{\fp \in \mathfrak{M}(k,n)}\mathbf{1}_{\scI(\fp)}(x)>\lambda 2^{n+1}
\end{equation}
and define
\begin{equation*}
    G_2:=
\bigcup_{k\ge 0}\bigcup_{k\le n}
A(2n+6,k,n)\, .
\end{equation*}
Define
    \begin{equation}\label{defineg3}
        G_3 :=
        \bigcup_{k\ge 0}\, \bigcup_{n \geq k}\,
        \bigcup_{0\le j\le 2n+3}
        \bigcup_{\fp \in \fL_4 (k,n,j)}
        \scI(\fp)\, .
     \end{equation}
Define $G'=G_1\cup G_2 \cup G_3$. The following bound of the measure of $G'$ will be proved in \Cref{subsetexcset}.
\begin{lemma}[exceptional set]
    \label{exceptional-set}
    \leanok
    \lean{exceptional_set}
    We have
    \begin{equation*}
        \mu(G')\le 2^{-1}\mu(G)\, .
    \end{equation*}
\end{lemma}

In \Cref{subsecforest}, we prove the following lemma. Recall that we defined
\[
    \fP' = \{\fp \in \fP \ : \ \scI(\fp) \not\subset G'\}.
\]

\begin{lemma}[forest union]
    \label{forest-union}
    \leanok
    \lean{forest_union}
    Let $n \ge k$ and $0 \le j \le 2n+3$. Then the set $\fC_5(k,n,j) \cap \fP'$ is the union of at most $4n + 12$ many $n$-forests.
\end{lemma}

In \Cref{subsecantichain}, we prove the following lemma.
\begin{lemma}[forest complement]
    \label{forest-complement}
    There exists a decomposition of the set of tiles not contained in any forest into a disjoint union of antichains
    \begin{equation}
    \label{eq-forest-complement}
        \fP' \setminus \bigcup_{k \ge 0} \bigcup_{n \ge k} \bigcup_{0\le j \le 2n+3} \fC_5(k,n,j) = \bigcup_{n \ge 0} \bigcup_{j = 0}^{Z(n+2)^3} \mathfrak{A}_{n,j}\, ,
    \end{equation}
    where each $\mathfrak{A}_{n,j}$ is an antichain with
    \begin{equation}
        \label{eq-antichain1dens}
        \dens_1(\mathfrak{A}_{n,j}) \le 2^{4a+1 - n}\, .
    \end{equation}
\end{lemma}

Note that by the definition of $G_1$, for all $\fC \subset \fP'$
\[
    \dens_2(\fC) \le 2^{2a + 5} \frac{\mu(F)}{\mu(G)}\, .
\]
\Cref{exceptional-set}, \Cref{forest-union} and \Cref{forest-complement}
then prove  Proposition \ref{prop-decomposition}.

\subsection{Exceptional set estimates}
\label{subsetexcset}

We prove  bounds for $G_1$, $G_2$ and $G_3$
in \eqref{eqG1}, \eqref{eqG2} and \eqref{eqG3} below. Summing the bounds proves \Cref{exceptional-set}.

The set $G_1$ is contained in the set $\{M\mathbf{1}_{F} > 2^{2a + 5} \mu(F) / \mu(G)\}$, so the weak $L^1$ bound for the Hardy-Littlewood maximal function implies
\begin{equation}
    \label{eqG1}
    \mu(G_1)\le 2^{-5}\mu(G)\,.
\end{equation}

We turn to the set $G_2$.

\begin{lemma}[John Nirenberg]
\label{John-Nirenberg}
\leanok
\lean{john_nirenberg}
    For all integers $k,n,\lambda\ge 0$, we have
    \begin{equation*}
        \mu(A(\lambda,k,n)) \le 2^{k+1-\lambda}\mu(G)\,.
    \end{equation*}
\end{lemma}

\begin{proof}
    By pairwise disjointness of the sets $E_1(\fp)$, $\fp \in \mathfrak{M}(n,k)$ and \eqref{defMkn}, we have for all grid cubes $J$ the Carleson packing condition
    \[
        \sum_{\fp \in \mathfrak{M}(n,k): \scI(\fp) \subset J} \mu(\scI(\fp)) \le 2^n \sum_{\fp \in \mathfrak{M}(n,k): \scI(\fp) \subset J} \mu(E_1(\fp)) \le 2^n \mu(J)\,.
    \]
    By the John-Nirenberg inequality, it follows that for all grid cubes $J$
    \[
        \mu\Big(\Big\{x \in J \ : \ \sum_{\fp \in \mathfrak{M}(n,k): \scI(\fp) \subset J} \mathbf{1}_{\scI(\fp)}(x) > \lambda 2^{n+1}\Big\}\Big) \le 2^{-\lambda} \mu(J).
    \]
    Denote the set of maximal cubes $J$ with $\mu(G \cap J) > 2^{-k-1} \mu(J)$ by $\mathcal{M}^*(k)$.
    Each cube $\scI(\fp)$ with $\fp \in \mathfrak{M}(n,k)$ is contained in a cube $J \in \mathcal{M}^*(k)$. Hence, summing the last display over all cubes $J \in \mathcal{M}^*(k)$ gives
    \[
        \mu(A(\lambda, k,n)) = \sum_{J \in \mathcal{M}^*(k)} \mu(A(\lambda, k,n)\cap J) \le 2^{-\lambda} \sum_{J \in \mathcal{M}^*(k)} \mu(J).
    \]
    Using pairwise disjointness of the cubes in $\mathcal{M}^*(k)$ we conclude
    \[
         \le 2^{k+1-\lambda} \sum_{J \in \mathcal{M}^*(k)} \mu(J \cap G) \le 2^{k+1-\lambda} \mu(G). \qedhere
    \]
\end{proof}
Using Lemma \ref{John-Nirenberg} and summing twice a geometric series yields
\begin{equation}
    \label{eqG2}
    \mu(G_2) \le \sum_{0\le k}\sum_{k\le n}
\mu(A(2n+6,k,n))\le 2^{-2}\mu(G).
\end{equation}

We turn to the set $G_3$.

\begin{lemma}[top tiles]
\label{top-tiles}
\leanok
\lean{top_tiles}
 \uses{John-Nirenberg}
    We have
    \begin{equation}\label{eq-musum}
        \sum_{\mathfrak{m} \in \mathfrak{M}(k,n)} \mu(\scI(\mathfrak{m}))\le 2^{n+k+3}\mu(G).
    \end{equation}
\end{lemma}
\begin{proof}
\leanok
We have for the left-hand side of \eqref{eq-musum}
\begin{equation*}
    \int \sum_{\mathfrak{m} \in \mathfrak{M}(k,n)} \mathbf{1}_{\scI(\mathfrak{m})}(x) \, \mathrm{d}\mu(x) \le
2^{n+1} \sum_{\lambda=0}^{|\mathfrak{M}|}\mu(A(\lambda, k,n))\,.
\end{equation*}
Now the claimed estimate follows from \Cref{John-Nirenberg}.
\end{proof}

\begin{lemma}[tree count]
\label{tree-count}
\leanok
\lean{tree_count}
Let $k,n,j\ge 0$. We have for every $x\in X$
\begin{equation*}
    \sum_{\fu\in \fU_1(k,n,j)} \mathbf{1}_{\scI(\fu)}(x)
    \le 2^{-j}
    2^{9a} \sum_{\mathfrak{m}\in \mathfrak{M}(k,n)}
     \mathbf{1}_{\scI(\mathfrak{m})}(x)\, .
\end{equation*}
\end{lemma}

\begin{proof}
\leanok
Let $x\in X$. For each
$\fu\in \fU_1(k,n,j)$ with $x\in \scI(\fu)$, as $\fu \in \fC_1(k,n,j)$,
there are at least $2^{j}$ elements $\mathfrak{m}\in \mathfrak{M}(k,n)$
with $100\fu \lesssim \mathfrak{m}$.  Hence
\begin{equation}\label{ubymsum}
     \mathbf{1}_{\scI(\fu)}(x)
    \le 2^{-j}\sum_{\mathfrak{m} \in \mathfrak{M}(k,n): 100\fu\lesssim \mathfrak{m}} \mathbf{1}_{\scI(\mathfrak{m})}(x)\, .
\end{equation}
For each $\mathfrak{m}\in \mathfrak{M}(k,n)$
with $x\in \scI(\mathfrak{m})$,
let $\fU(\mathfrak{m})$ be the set of
$\fu\in \fU_1(k,n,j)$ with $x\in \scI(\fu)$
and $100\fu \lesssim \mathfrak{m}$.
Summing \eqref{ubymsum} over $\fu$ gives
\begin{equation*}
     \sum_{\fu\in \fU_1(k,n,j)} \mathbf{1}_{\scI(\fu)}(x)
    \le 2^{-j}\sum_{\mathfrak{m} \in \mathfrak{M}(k,n)}
    \sum_{\fu \in \fU(\mathfrak{m})} \mathbf{1}_{\scI(\mathfrak{m})}(x)\, .
\end{equation*}
It remains to show that $\lvert \fU(\mathfrak{m})\rvert \le 2^{9a}$.
Let $\fu \in \fU(\mathfrak{m})$, then by definition
\begin{equation*}
     d_{\fu}(\fcc(\fu),\fcc(\mathfrak{m}))\le 100\, .
\end{equation*}
If $\fu'$ is a further element in $\fU(\mathfrak{m})$ with $\fu\neq \fu'$, then
\begin{equation*}
    \fcc(\mathfrak{m})
    \in B_{\fu}(\fcc(\fu),100)\cap B_{\fu'}(\fcc(\fu'),100)\,.
\end{equation*}
By the definition of $\fU_1(k,n,j)$, none of $\scI(\fu)$, $\scI(\fu')$ is strictly contained in the other. As both contain $x$, we have $\scI(\fu)=\scI(\fu')$ and in particular $d_{\fu}=d_{\fu'}$.

The geometric doubling condition \eqref{thirddb} for $d_{\fu}$ implies that every collection of disjoint balls $B_\fu(\mfa, 0.2)$ with $\mfa \in B_\fu(\fcc(\mathfrak{m}), 100)$ has size at most $2^{9a}$. Applying this to the balls with centers $\fcc(\fu'), \fu' \in \fU(\mathfrak{m})$, which are disjoint by \eqref{eq-dis-freq-cover}, completes the proof.
\end{proof}

\begin{lemma}[boundary exception]
\label{boundary-exception}
\leanok
\lean{boundary_exception}

Let $\mathcal{L}(\fu)$ be as defined in \eqref{eq-L-def}. For each $\fu\in \fU_1(k,n,l)$
\begin{equation*}
\mu\Big(\bigcup_{I\in \mathcal{L}(\fu)} I\Big)
\le 2 \cdot D^{-\kappa Z(n+1)}
        \mu(\scI(\mathfrak{u})).
\end{equation*}
\end{lemma}
\begin{proof}
\leanok
Let $\fu\in \fU_1(k,n,l)$ and $I \in \mathcal{L}(\fu)$. By definition of $\mathcal{L}(\fu)$
\begin{equation*}
    I \subset X(\fu):=\{x \in \scI(\fu) \, : \, \rho(x, X \setminus \scI(\fu)) \leq 12 D^{\ps(\fu) - Z(n+1)-1}\}\,.
\end{equation*}
 By the small boundary property \eqref{eq-small-boundary} and $D \ge 12$
   $$
        \mu(X(\fu)) \le
        2\cdot(12 D^{-Z(n+1)-1})^\kappa
        \mu(\scI(\mathfrak{u})) \le 2 \cdot D^{-\kappa Z(n+1)}
        \mu(\scI(\mathfrak{u})).
    $$
\end{proof}

\begin{lemma}[third exception]
    \label{third-exception}
    \leanok
    \lean{third_exception}
\uses{tree-count,boundary-exception, top-tiles}
       We have
\begin{equation}
    \label{eqG3}
    \mu(G_3)\le 2^{-4} \mu(G)\, .
\end{equation}
\end{lemma}
\begin{proof}
\leanok
As each $\fp\in \fL_4(k,n,j)$
is contained in $\cup\mathcal{L}(\fu)$ for some
$\fu\in \fU_1(k,n,l)$, we have
\begin{equation*}
\mu\Big(\bigcup_{\fp \in \fL_4 (k,n,j)}\scI(\fp)\Big)
\le \sum_{\fu\in \fU_1(k,n,j)}
\mu\Big(\bigcup_{I \in \mathcal{L} (\fu)}I\Big)\, .
\end{equation*}
Using \Cref{boundary-exception}, \Cref{tree-count} and finally \Cref{top-tiles}, we estimate this further
 by
\begin{equation*}
    \le 2\sum_{\fu\in \fU_1(k,n,j)}
    D^{-\kappa Z(n+1)}
        \mu(\scI(\mathfrak{u}))
    \le 2^{9a+1-j} \sum_{\mathfrak{m}\in \mathfrak{M}(k,n)}
     D^{-\kappa Z(n+1)}
    \mu(\scI(\mathfrak{m}))
\end{equation*}
  \begin{equation*}
     \le
2^{9a + 1-j} D^{-\kappa Z(n+1)}
     2^{n+k+3}\mu(G)\, .
\end{equation*}
Now we estimate $G_3$ defined in \eqref{defineg3} by
\begin{align*}
    \mu(G_3)&\le \sum_{k\ge 0}\, \sum_{n \geq k}\,
    \sum_{0\le j\le 2n+3}
    \mu\Big(\bigcup_{\fp \in \fL_4 (k,n,j)}
    \scI(\fp)\Big)\\
    &\le \sum_{k\ge 0}\, \sum_{n \geq k}\,
    \sum_{0\le j\le 2n+3}
    2^{9a + 4 + n + k -j} D^{-\kappa Z(n+1)}\mu(G)
\end{align*}
Summing the series, using the definitions of $D, Z$ and $\kappa$, proves the lemma.
\end{proof}

\subsection{Auxiliary lemmas}
\label{subsec-lessim-aux}
Before proving \Cref{forest-union} and \Cref{forest-complement}, we collect some useful properties of $\lesssim$.

\begin{lemma}[wiggle order 2]
    \label{wiggle-order-2}
    \leanok
    \lean{smul_C2_1_2}
    \uses{monotone-cube-metrics}
    Let $n, m \ge 1$ and $k > 0$.
    If $\fp, \fp' \in \fP$ with $\scI(\fp) \ne \scI(\fp')$ and
    $n \fp \lesssim k \fp'$
    then
    \begin{equation}
        \label{eq-wiggle2}
        (n + 2^{-95 a} m) \fp \lesssim m\fp'\,.
    \end{equation}
\end{lemma}

\begin{proof}
    \leanok
    The assumption implies that $\scI(\fp) \subsetneq \scI(\fp')$. Let $\mfa \in B_{\fp'}(\fcc(\fp'), m)$. Then we have by the triangle inequality
    $$
        d_{\fp}(\fcc(\fp), \mfa) \le d_{\fp}(\fcc(\fp), \fcc(\fp')) + d_{\fp}(\fcc(\fp'), \mfa)
    $$
    The first summand is bounded by $n$ since
    $$
        \fcc(\fp') \in B_{\fp'}(\fcc(\fp'), k) \subset B_{\fp}(\fcc(\fp), n).
    $$
    Using \Cref{monotone-cube-metrics} for the second summand shows
    $$
        d_{\fp}(\fcc(\fp), \mfa) \le n + 2^{-95a} d_{\fp'}(\fcc(\fp'), \mfa) < n + 2^{-95a} m\,.
    $$
    Combined with $\scI(\fp) \subset \scI(\fp')$, this yields \eqref{eq-wiggle2}.
\end{proof}

\begin{lemma}[wiggle order 3]
\label{wiggle-order-3}
\leanok
\lean{wiggle_order_11_10, wiggle_order_100, wiggle_order_500}
    The following implications hold for all $\fq, \fq' \in \fP$:
    \begin{equation}
        \label{eq-sc1}
        \fq \le \fq' \ \text{and} \ \lambda \ge 1.1 \implies \lambda \fq \lesssim \lambda \fq'\,,
    \end{equation}
    \begin{equation}
        \label{eq-sc2}
        10\fq \lesssim \fq' \ \text{and} \ \scI(\fq) \ne \scI(\fq') \implies 100 \fq \lesssim 100 \fq'\,,
    \end{equation}
    \begin{equation}
        \label{eq-sc3}
        2\fq \lesssim \fq' \ \text{and} \ \scI(\fq) \ne \scI(\fq') \implies 4 \fq \lesssim 500 \fq'\,.
    \end{equation}
\end{lemma}

\begin{proof}
    \leanok
    Claims \eqref{eq-sc2} and \eqref{eq-sc3} are consequences of \Cref{wiggle-order-2} and $a \ge 4$.
    For \eqref{eq-sc1}, if $\scI(\fq) = \scI(\fq')$ then
    $\fq = \fq'$ by \eqref{eq-dis-freq-cover} and the order definition \eqref{straightorder}.
    If $\scI(\fq) \ne \scI(\fq')$, then from \eqref{eq-freq-comp-ball} and the order definitions \eqref{straightorder} and \eqref{wiggleorder} it follows that
    $\fq \lesssim 0.2\fq'$, and \eqref{eq-sc1} follows from
    \Cref{wiggle-order-2}.
\end{proof}

We call a collection $\mathfrak{A}$ of tiles convex if
\begin{equation*}
    \fp \le \fp' \le \fp'' \ \text{and} \ \fp, \fp'' \in \mathfrak{A} \implies \fp' \in \mathfrak{A}\,.
\end{equation*}
With the help of \Cref{wiggle-order-2} and \Cref{wiggle-order-3}, it is easy to verify that each of the collections $\fP(k),  \fC(n,k)$ and $\fC_s(n,k,j)$ for $s = 1,2,3,4, 5$ are convex.

We close this subsection with an estimate for densities.

\begin{lemma}[dens compare]
    \label{dens-compare}
    \leanok
    \lean{dens1_le_dens'}
    For every $k, n\ge 0$ and every $\mathfrak{A} \subset \fC(k,n)$:
    \begin{equation*}
        \dens_1(\mathfrak{A}) \le 2^{4a}2^{-n+1}\,.
    \end{equation*}
\end{lemma}
\begin{proof}
\leanok
We first show that for all sets $\mathfrak{A} \subset \fP(k)$ it holds
\begin{equation}
    \label{eq-dens-comp}
    \dens_1(\mathfrak{A})\le \dens_k'(\mathfrak{A})\,.
\end{equation}
It suffices to show that for all $\fp'\in \mathfrak{A}$
and $\lambda\ge 2$ and $\fp\in \fP(\mathfrak{A})$ with $\lambda \fp' \lesssim \lambda \fp$ we have
\begin{equation*}
    \frac{\mu({E}_2(\lambda, \fp))}{\mu(\scI(\fp))}
    \le \sup_{\fp'' \in \fP(k): \lambda \fp' \lesssim \lambda \fp''}
    \frac{\mu({E}_2(\lambda, \fp''))}{\mu(\scI(\fp''))}.
\end{equation*}
    But if $\fp \in \fP(\mathfrak{A})$ then it satisfies \eqref{muhj1}
    and \eqref{muhj2}, so $\fp \in \fP(k)$. Thus we can simply take $\fp'' = \fp$ and the inequality along with \eqref{eq-dens-comp} follows.

    Combining \eqref{eq-dens-comp} with definitions \eqref{eq-densdef} and \eqref{def-cnk} we have that
    \[
        \dens_1(\fC(n,k)) \le \dens_k'(\fC(k,n))= \sup_{\fp \in \fC(k,n)} \dens_k'(\{\fp\}) \le 2^{4a}2^{-n+1}.
    \]
    The lemma follows since $\dens_1$ is increasing with respect to set inclusion.
\end{proof}

\subsection{Verification of the forest properties}
\label{subsecforest}
We prove \Cref{forest-union}.
Fix $k,n,j\ge 0$.
Define
$$
    \fC_6(k,n,j)
$$
to be the set of all tiles $\fp \in \fC_5(k,n,j)$ such that $\scI(\fp) \not\subset G'$. Since $\fC_5(k,n,j)$ is convex, so is $\fC_6(k,n,j)$.
The following chain of lemmas
establishes that the set $\fC_6(k,n,j)$ can be written as a union of a small number of $n$-forests.

For $\fu\in \fU_1(k,n,j)$, define
\begin{equation*}
    \mathfrak{T}_1(\fu):= \{\fp \in \fC_1(k,n,j) \ : \scI(\fp)\neq \scI(\fu), \ 2\fp \lesssim \fu\}\,.
\end{equation*}
Define
\begin{equation*}
    \fU_2(k,n,j) := \{ \fu \in \fU_1(k,n,j) \, : \, \mathfrak{T}_1(\fu) \cap \fC_6(k,n,j) \ne \emptyset\}\,.
\end{equation*}
Define a relation $\sim$ on $\fU_2(k,n,j)$
by setting $\fu\sim \fu'$
for $\fu,\fu'\in \fU_2(k,n,j)$
if $\fu=\fu'$ or there exists $\fp$ in $\mathfrak{T}_1(\fu)$
with $10 \fp\lesssim \fu'$.

\begin{lemma}[relation geometry]
    \label{relation-geometry}
    \leanok
    \lean{URel.eq, URel.not_disjoint}
    \uses{wiggle-order-3}
    If $\fu \sim \fu'$, then $\scI(u) = \scI(u')$ and
    \begin{equation*}
        B_{\fu}(\fcc(\fu), 100) \cap B_{\fu'}(\fcc(\fu'), 100) \neq \emptyset\, .
    \end{equation*}
\end{lemma}

\begin{proof}
    \leanok
    Let $\fu, \fu' \in \fU_2(k,n,j)$ with $\fu \sim \fu'$ and $\fu \ne \fu'$. Then there exists $\fp \in \fC_1(k,n,j)$ such that $\scI(\fp) \ne \scI(\fu)$ and $2 \fp \lesssim \fu$ and $10 \fp \lesssim \fu'$.
    Using \Cref{wiggle-order-3}, we deduce that
    \begin{equation}
        \label{eq-Fefferman-trick0}
        100 \fp\lesssim 100 \fu\,, \qquad 100 \fp \lesssim 100\fu'\,.
    \end{equation}
    Now suppose that $B_{\fu}(\fcc(\fu), 100)$ and $ B_{\fu'}(\fcc(\fu'), 100)$ are disjoint. Then $\mathfrak{B}(\fu)$ and $ \mathfrak{B}(\fu')$ are disjoint, but also $\mathfrak{B}(\fu) \subset \mathfrak{B}(\fp)$ and $\mathfrak{B}(\fu') \subset \mathfrak{B}(\fp)$, by \eqref{defbfp}, \eqref{wiggleorder} and \eqref{eq-Fefferman-trick0}.
    Hence,
    $$
        |\mathfrak{B}(\fp)| \geq |\mathfrak{B}(\fu)| + |\mathfrak{B}(\fu')| \geq 2^{j} + 2^j = 2^{j+1}\,,
    $$
    which contradicts $\fp \in \fC_1(k,n,j)$. Therefore we must have
    \begin{equation}
    \label{eq100cap}
        B_{\fu}(\fcc(\fu), 100) \cap B_{\fu'}(\fcc(\fu'), 100) \ne \emptyset\, .
    \end{equation}
    Since $\scI(\fp) \subset \scI(\fu)$ and $\scI(\fp) \subset \scI(\fu')$, the cubes $\scI(\fu)$ and $\scI(\fu')$ are nested.
    Combining this with \eqref{eq100cap} and definition \eqref{defunkj} of $\fU_1(k,n,j)$, we conclude that $\scI(\fu) = \scI(\fu')$.
\end{proof}

\begin{lemma}[equivalence relation]
\label{equivalence-relation}
\uses{relation-geometry}
\leanok
\lean{equivalenceOn_urel}
For each $k,n,j$, the relation $\sim$ on
$\fU_2(k,n,j)$ is an equivalence relation.
\end{lemma}

\begin{proof}
    \leanok
    Reflexivity holds by definition.
    For transitivity, pick pairwise distinct
    $\fu, \fu', \fu'' \in \fU_1(k,n,j)$
    with $\fu \sim \fu'$, $\fu' \sim \fu''$.
    By \Cref{relation-geometry}, it follows that $\scI(\fu) =\scI(\fu') = \scI(\fu'')$, that there exists
    \begin{equation*}
        \mfa \in B_{\fu}(\fcc(\fu), 100) \cap B_{\fu'}(\fcc(\fu'), 100)
    \end{equation*}
    and that there exists
    \begin{equation*}
        \mfb \in B_{\fu'}(\fcc(\fu'), 100) \cap B_{\fu''}(\fcc(\fu''), 100)\, .
    \end{equation*}
    We now estimate for $q \in B_{\fu''}(\fcc(\fu''), 1)$
    \begin{align*}
        d_{\fu}(\fcc(\fu), q) &\le d_{\fu}(\fcc(\fu), \mfa) + d_{\fu}(\mfa, \fcc(\fu'))\\
        &\quad+ d_{\fu}(\fcc(\fu'), \mfb) + d_{\fu}(\mfb, \fcc(\fu'')) +
        d_{\fu}(\fcc(\fu''), q)\,.
    \end{align*}
    Using \eqref{defdp} and the fact that $\scI(\fu) = \scI(\fu') = \scI(\fu'')$ this equals
    \begin{multline*}
        d_{\fu}(\fcc(\fu), \mfa) + d_{\fu'}(\mfa, \fcc(\fu'))+ d_{\fu'}(\fcc(\fu'), \mfb) + d_{\fu''}(\mfb, \fcc(\fu'')) +
        d_{\fu''}(\fcc(\fu''), q)\\
        < 100 + 100 + 100 + 100 + 1 < 500\,.
    \end{multline*}
    Since $\fu \sim \fu'$, there exists some $\fp \in \mathfrak{T}_1(\fu)$ with $10\fp\lesssim \fu'$.
    By \eqref{eq-sc3}, $\fp \in \mathfrak{T}_1(\fu)$ implies
    $4 \fp \lesssim 500 \fu$, from which it follows that $d_{\fp}(\fcc(\fp), q) < 4 < 10$. We have shown that $B_{\fu''}(\fcc(\fu''), 1) \subset B_{\fp}(\fcc(\fp), 10)$, combining this with $\scI(\fu'') = \scI(\fu)$ gives $\fu \sim \fu''$.

    For symmetry suppose that $\fu \sim \fu'$, so $\scI(\fu) = \scI(\fu')$ and there exists $\mfa \in B_{\fu}(\fcc(\fu), 100) \cap B_{\fu'}(\fcc(\fu'), 100)$. We may assume that $\fu \ne \fu'$. There exists $\fp \in \mathfrak{T}_1(\fu')$, which then satisfies $2\fp\lesssim \fu'$ and $\scI(\fp) \neq \scI(\fu')$. By \eqref{eq-sc3}
    \begin{equation}
        \label{eq-rel1}
        4\fp \lesssim 500 \fu'\,.
    \end{equation}
    If $q \in B_{\fu}(\fcc(\fu),1)$ then we have, using that $\scI(\fu) = \scI(\fu')$ and hence $d_\fu = d_{\fu'}$:
    \begin{align*}
        d_{\fu'}(\fcc(\fu'), q) &\le  d_{\fu'}(\fcc(\fu'), \mfa) + d_{\fu}(\mfa, \fcc(\fu)) + d_{\fu}(\fcc(\fu), q) < 500\,.
    \end{align*}
    Combined with \eqref{eq-rel1} we obtain $B_{\fu}(\fcc(\fu), 1) \subset B_{\fp}(\fcc(\fp), 4)$, so
    \begin{equation}
        \label{eqcheck1}
        10\fp \lesssim 4\fp \lesssim \fu
    \end{equation}
    and consequently $\fu' \sim \fu$.
\end{proof}

Choose a set $\fU_3(k,n,j)$ of representatives for the equivalence
classes of $\sim$ in $\fU_2(k,n,j)$.
Define for each $\fu\in \fU_3(k,n,j)$
\begin{equation}\label{definesv}
\fT_2(\fu):=
   \bigcup_{\fu\sim \fu'}\mathfrak{T}_1(\fu')\cap \fC_6(k,n,j)\, .
\end{equation}
It is straightforward to check that each $\fT_2(\fu)$ is convex, meaning \eqref{forest2} holds. By construction
\begin{equation*}
    \fC_6(k,n,j)=\bigcup_{\fu\in \fU_3(k,n,j)}\mathfrak{T}_2(\fu)\, .
\end{equation*}
We now check that $(\fC_6(k,n,j), \fT_2)$ satisfies also the forest properties \eqref{forest1}, \eqref{forest4}, \eqref{forest5} and \eqref{forest6}.

\begin{lemma}[forest geometry]
    \label{forest-geometry}
    \uses{relation-geometry}
    \leanok
    \lean{forest_geometry}
    For each $\fu\in \fU_3(k,n,j)$,
    the set $\mathfrak{T}_2(\fu)$
    satisfies \eqref{forest1}.
\end{lemma}
\begin{proof}
    \leanok
    Let $\fp \in \mathfrak{T}_2(\fu)$. By \eqref{definesv}, there exists $\fu' \sim \fu$ with $\fp \in \mathfrak{T}_1(\fu')$.
    The proof of \eqref{eqcheck1} shows that $4\fp \lesssim \fu$. Also $\scI(\fp) \ne \scI(\fu')$ and by \Cref{relation-geometry} $\scI(\fu) = \scI(\fu')$, so $\scI(\fp) \ne \scI(\fu)$.
\end{proof}

\begin{lemma}[forest separation]
    \label{forest-separation}
    \uses{monotone-cube-metrics}
    \leanok
    \lean{forest_separation}
    For each $\fu,\fu'\in \fU_3(k,n,j)$ with $\fu\neq \fu'$ and each $\fp \in \fT_2(\fu)$
    with $\scI(\fp)\subset \scI(\fu')$ we have
    \begin{equation*}
    d_{\fp}(\fcc(\fp), \fcc(\fu')) > 2^{Z(n+1)}\,.
    \end{equation*}
\end{lemma}

\begin{proof}
    \leanok
    By the definition \eqref{eq-C2-def} of $\fC_2(k,n,j)$, there exists a tile $\fp' \in \fC_1(k,n,j)$ with $\fp' \le \fp$ and $\ps(\fp') \le \ps(\fp)- Z(n+1)$.
    By \Cref{monotone-cube-metrics}, we have
    \begin{equation}
        \label{eq big}
        d_{\fp}(\fcc(\fp), \fcc(\fu')) \ge 2^{95a Z(n+1)} d_{\fp'}(\fcc(\fp), \fcc(\fu'))\,.
    \end{equation}
    Since $\fp \in \fT_2(\fu)$, there exists $\mathfrak{v} \sim \fu$ with $2\fp' \lesssim 2\fp \lesssim \mathfrak{v}$ and $\scI(\fp') \ne \scI(\mathfrak{v})$. Since $\fu, \fu'$ are not equivalent under $\sim$, neither are $\mathfrak{v}$ and $\fu'$, thus $10\fp' \not\lesssim \fu'$. This implies that there exists $q \in B_{\fu'}(\fcc(\fu'), 1) \setminus B_{\fp'}(\fcc(\fp'), 10)$.

    From $\fp' \le \fp$ and $\scI(\fp') \subset \scI(\fp) \subset \scI(\fu')$ and \Cref{monotone-cube-metrics} it then follows that
    \begin{align*}
        d_{\fp'}(\fcc(\fp), \fcc(\fu'))
        &\ge -d_{\fp'}(\fcc(\fp), \fcc(\fp')) + d_{\fp'}(\fcc(\fp'), q) - d_{\fu'}(q, \fcc(\fu'))\\
        &> -1 + 10 - 1 = 8\,.
    \end{align*}
    Combining this with \eqref{eq big} completes the proof.
\end{proof}

\begin{lemma}[forest inner]
    \label{forest-inner}
    \uses{relation-geometry}
    \leanok
    \lean{forest_inner}
    For each $\fu\in \fU_3(k,n,j)$
    and each $\fp \in \mathfrak{T}_2(\fu)$
    we have
    \begin{equation*}
        B(\pc(\fp), 8 D^{\ps(\fp)}) \subset \scI(\fu).
    \end{equation*}
\end{lemma}

\begin{proof}
    \leanok
    Let $\fp \in \mathfrak{T}_2(\fu)$, so in particular $\fp \in \fC_4(k,n,j)$. By the definition \eqref{eq-C4-def} of $\fC_4(k,n,j)$, there exists a tile $\fq \in \fC_3(n,k,j)$ with $\fp \le \fq$ and
    $\ps(\fp) \le \ps(\fq) - Z(n+1)$.
    By the definition \eqref{eq-C3-def} of $\fC_3(n,k,j)$, there exists $\fu'' \in \fU_1(k,n,j)$ with $2\fq \lesssim \fu''$ and $\ps(\fq) < \ps(\fu'')$.
    Then we have in particular that $10 \fp \lesssim \fu''$, so, using transitivity of $\sim$, we have $\fu \sim \fu''$.
    \Cref{relation-geometry} shows that $\scI(\fu'') = \scI(\fu)$, hence $\ps(\fq) < \ps(\fu)$ and $\ps(\fp) \le \ps(\fq) - Z(n+1) \le \ps(\fu) - Z(n+1) - 1$.

    Let $I \in \mathcal{D}$ be the cube with $s(I) = \ps(\fu) - Z(n+1) - 1$ and $I \subset \scI(\fu)$ and $\scI(\fp) \subset I$.
    Since $\fp \in \fC_5(k,n,j)$, we have that $I \notin \mathcal{L}(\fu)$, so $B(c(I), 8D^{s(I)}) \subset \scI(\fu)$.
    The same then holds for the subcube $\scI(\fp) \subset I$.
\end{proof}

It remains to prove the final forest property \eqref{forest3}. This is accomplished by the following lemma, which implies that $\fC_6(k,n,j)$ is the union of at most $4n+12$ sub-collections satisfying \eqref{forest3}, which are then $n$-forests.

\begin{lemma}[forest stacking]
    \label{forest-stacking}
    \leanok
    \lean{forest_stacking}
    It holds for $k\le n$ that
    \begin{equation*}
        \sum_{\fu \in \fU_3(k,n,j)} \mathbf{1}_{\scI(\fu)} \le (4n+12)2^{n}\,.
    \end{equation*}
\end{lemma}

\begin{proof}
    \leanok
    Suppose to the contrary that a point $x$ is contained in more than $(4n + 12)2^n$ cubes $\scI(\fu)$ with $\fu \in \fU_3(k,n,j)$.
    Since $\fU_3(k,n,j) \subset \fC_1(k,n,j)$, for each such $\fu$,
    there exists $\mathfrak{m} \in \mathfrak{M}(k,n)$ such that $100\fu \lesssim \mathfrak{m}$.
    We fix such an $\mathfrak{m}(\fu) := \mathfrak{m}$ for each $\fu$, and claim that the map $\fu \mapsto\mathfrak{m}(\fu)$ is injective.
    Indeed, assume for $\fu\neq \fu'$ there is $\mathfrak{m} \in \mathfrak{M}(k,n)$ such that
    $100\fu \lesssim \mathfrak{m}$ and $100\fu' \lesssim \mathfrak{m}$. By \eqref{dyadicproperty},
    either $\scI(\fu) \subset \scI(\fu')$ or $\scI(\fu') \subset \scI(\fu)$.
    By \eqref{defunkj}, the balls $B_{\fu}(\fcc(\fu),100)$ and $B_{\fu'}(\fcc(\fu'), 100)$ are disjoint.
    This contradicts $\Omega(\mathfrak{m})$ being contained in both sets by \eqref{eq-freq-comp-ball}.
    Thus $x$ is contained in more than $(4n + 12)2^n$ cubes $\scI(\mathfrak{m})$, $\mathfrak{m} \in \mathfrak{M}(k,n)$.
    Consequently, we have by \eqref{eq-Aoverlap-def} that $x \in A(2n + 6, k,n) \subset G_2$.
    Let $\scI(\fu)$ be an inclusion minimal cube among the $\scI(\fu'), \fu' \in \fU_3(k,n,j)$ with $x \in \scI(\fu)$.
    It satisfies
    $$
        \scI(\fu) \subset \{y \ : \ \sum_{\fu \in \fU_3(k,n,j)} \mathbf{1}_{\scI(\fu)}(y) > 1 + (4n+12)2^{n}\} \subset G_2\,,
    $$
    thus $\mathfrak{T}_1(\fu) \cap \fC_6(k,n,j) = \emptyset$.
    This contradicts $\fu \in \fU_2(k,n,j)$.
\end{proof}

\subsection{Verification of the antichain property}
\label{subsecantichain}

We prove \Cref{forest-complement}. We first claim that the set on the left side of \eqref{eq-forest-complement} equals
\begin{equation}\label{eq-l-decomposition}
     \fP' \cap \Big[\bigcup_{k \ge 0} \bigcup_{n \ge k} \big[\fL_0(k,n) \cup \bigcup_{0 \le j \le 2n+3} \fL_2(k,n,j)\big] \end{equation}
\begin{equation*}
    \cup \bigcup_{k \ge 0} \bigcup_{n \ge k}\bigcup_{0 \le j \le 2n+3} \bigcup_{0 \le l \le Z(n+1)} \big[\fL_1(k,n,j,l) \cup \fL_3(k,n,j,l)\big]\Big]\,.
\end{equation*}
To see the claim, let $\fp \in \fP(k) \cap \fP'$. By \eqref{muhj2},
\[
    \mu(E_2(\lambda, \fp')) \le 2^{-k} \mu(\scI(\fp'))
\]
whenever $\fp' \in \fP(k)$ with $\scI(\fp) \subset \scI(\fp')$, so $\dens_k'(\{\fp\}) \le 2^{-k}$. Hence there exists $n\ge k$ with $\fp \in \fC(k,n)$. Also, since $\scI(\fp) \not \subset G'$, there exist at most $1 + (4n + 12)2^n < 2^{2n+4}$ tiles $\mathfrak{m} \in \mathfrak{M}(k,n)$ with $\fp \le \mathfrak{m}$. It follows that $\fp \in \fL_0(k,n)$ or $\fp \in \fC_1(k,n,j)$ for some $1 \le j \le 2n + 3$.
The claim now follows as
the construction of the collections $\fC_5(k,n,j)$ consists of removing from $\fC_1(k,n,j)$ the collections $\mathcal{L}_1, \mathcal{L}_2, \mathcal{L}_3$ and $\mathcal{L}_4$
with paremeters as claimed in \eqref{eq-l-decomposition}, where we have omitted $\mathcal{L}_4$ because all tiles $\fp \in \fL_4(k,n,j)$ satisfy $\scI(\fp) \subset G'$ and are therefore not contained in $\fP'$.

\Cref{dens-compare} implies the density estimate \eqref{eq-antichain1dens} for all terms in the decomposition.
Moreover, $\fL_1$ and $\fL_3$ where constructed as minimal or maximal sets of tiles, thus they are all antichains. It  remains to verify that $\fL_0$ and $\fL_2$ can also be decomposed into a small number of antichains.

\begin{lemma}[L0 antichain]
\label{L0-antichain}
\uses{monotone-cube-metrics}
\leanok
\lean{iUnion_L0', pairwiseDisjoint_L0', antichain_L0'}
    For all $k \ge 0$ and $n \ge k$, the set $\fL_0(k,n)$ is the union of at most $n$ antichains.
\end{lemma}

\begin{proof}
    \leanok
    It suffices to show that $\fL_0(k,n)$ contains no chain of length $n + 1$. Suppose that we had such a chain $\fp_0 \le \fp_1 \le \dotsb \le \fp_{n}$ of pairwise distinct $\fp_i$. Since $\fp_n \in \fC(k,n)$, we have that $\dens_k'(\{\fp_n\}) > 2^{-n}$, which by definition means that there exists $\fp' \in \fP(k)$ and $\lambda \ge 2$ with $\lambda \fp_n \le \lambda \fp'$ and
    \begin{equation}
        \label{eq-p'}
        \frac{\mu(E_2(\lambda, \fp'))}{\mu(\scI(\fp'))} > \lambda^{a} 2^{4a} 2^{-n}\,.
    \end{equation}
    Let $\mathfrak{O}$ be the set of all $\fp'' \in \fP(k)$ such that we have $ \scI(\fp'') = \scI(\fp')$ and $B_{\fp'}(\fcc(\fp'), \lambda) \cap \Omega(\fp'') \neq \emptyset$.
    From the geometric doubling property \eqref{thirddb} and the fact \eqref{eq-freq-comp-ball} that the balls $B_{\fp'}(\fcc(\fp''), 0.2)$, $\fp'' \in \mathfrak{O}$ are disjoint, it follows that
    \begin{equation*}
        |\mathfrak{O}| \le 2^{4a}\lambda^a\,.
    \end{equation*}
    By the definitions \eqref{definee1} and \eqref{definee2} we have $E_2(\lambda, \fp') \subset \bigcup_{\fp'' \in \mathfrak{O}} E_1(\fp'')$, thus
    $$
        \sum_{\fp'' \in \mathfrak{O}} \frac{\mu(E_1(\fp''))}{\mu(\scI(\fp''))} > 2^{4a}\lambda^a 2^{-n}\,.
    $$
    Hence there exists a tile $\fp'' \in \mathfrak{O}$ with
    \begin{equation}
    \label{eq-mexists}
        \mu(E_1(\fp'')) \ge 2^{-n} \mu(\scI(\fp''))\,.
    \end{equation}
    From \eqref{eq-p'}, the inclusion $E_2(\lambda, \fp') \subset \scI(\fp')$ and $a\ge 1$ we obtain
    $$
        2^n \geq 2^{4a} \lambda^{a} \geq \lambda\,.
    $$
    Let $\mathfrak{m} \in \mathfrak{M}(k,n)$ with $\fp'' \leq \mathfrak{m}$, it exists by \eqref{eq-mexists}.
    From \Cref{monotone-cube-metrics} and $a \ge 1$, it holds for all $\mfa \in B_{\mathfrak{m}}(\fcc(\mathfrak{m}), 1)$ that
    \begin{align*}
        d_{\fp_0}(\fcc(\fp_0), \mfa)
        &\leq d_{\fp_0}(\fcc(\fp_0), \fcc(\fp_{n})) + d_{\fp_0}(\fcc(\fp_{n}), \fcc(\fp')) + d_{\fp_0}(\fcc(\fp'), \fcc(\fp''))\\
        &\quad+ d_{\fp_0}(\fcc(\fp''), \fcc(\mathfrak{m})) +
        d_{\fp_0}(\fcc(\mathfrak{m}), \mfa)\\
        &\leq 1 + 2^{-95an} (d_{\fp_{n}}(\fcc(\fp_n), \fcc(\fp')) + d_{\fp'}(\fcc(\fp'), \fcc(\fp''))\\
        &\quad+ d_{\fp''}(\fcc(\fp''), \fcc(\mathfrak{m})) +
        d_{\mathfrak{m}}(\fcc(\mathfrak{m}), \mfa))\\
        &\leq 1 + 2^{-95an}(\lambda + (\lambda + 1) + 1 + 1) \leq 100\,.
    \end{align*}
    This implies that $100\fp_0 \lesssim \mathfrak{m}$, a contradiction to $\fp_0 \in \fL_0(k,n)$.
\end{proof}

\begin{lemma}[L2 antichain]
\label{L2-antichain}
\uses{monotone-cube-metrics}
\leanok
\lean{antichain_L2}
    Each of the sets $\fL_2(k,n,j)$ is an antichain.
\end{lemma}

\begin{proof}
    \leanok
    Suppose that there are $\fp_0, \fp_1 \in \fL_2(k,n,j)$ with $\fp_0 \ne \fp_1$ and $\fp_0 \le \fp_1$. By \Cref{wiggle-order-2}, it follows that $2\fp_0 \lesssim 200\fp_1$. Since $\fC_1(k,n,j)$ is finite, there exists a chain $2\fp_0 \lesssim 200 \fp_1 \lesssim \dotsb \lesssim 200 \fp_l$ of distinct $\fp_i \in \fC_1(k,n,j)$ of maximal length $l$. Since $2\fp_0 \lesssim 200 \fp_l \lesssim \fp_l$ and $\fp_0 \in \fL_2(k,n,j)$, we have $\fp_l \not\in \fU_1(k,n,j)$. So, by definition of $\fU_1(k,n,j)$, there exists $\fp_{l+1} \in \fC_1(k,n,j)$ with $\scI(\fp_l) \subsetneq \scI(\fp_{l+1}) $ and $B_{\fp_l}(\fcc(\fp_l), 100) \cap B_{\fp_{l+1}}(\fcc(\fp_{l+1}), 100) \ne \emptyset$. Using \Cref{monotone-cube-metrics}, it is straightforward to deduce $200 \fp_l \lesssim 200\fp_{l+1}$. This contradicts maximality of $l$.
\end{proof}

\section{Proof of the Antichain Operator Proposition}

\label{antichainboundary}

Let an antichain $\mathfrak{A}$ and functions $f$, $g$ as in \Cref{antichain-operator} be given.
We prove in \Cref{sec-TT*-T*T} two inequalities \eqref{eqttt9} and \eqref{eqttt3}, each involving one of the two densities.
The claimed estimate \eqref{eq-antiprop} follows as the product of the $(2-q)$-th power of \eqref{eqttt9} and the $(q-1)$-st power of \eqref{eqttt3}.

The proof of \eqref{eqttt9} will need a careful estimate formulated in
\Cref{tile-correlation} of
the $TT^*$ correlation between two tile operators.
\Cref{tile-correlation} will be proven in
Subsection \ref{sec-tile-operator}.

The summation of the contributions of these individual correlations will require a
geometric \Cref{antichain-tile-count} counting the relevant tile pairs.
\Cref{antichain-tile-count} will be proven in Subsection
\ref{subsec-geolem}.

\subsection{The density arguments}\label{sec-TT*-T*T}

By the definition of $E(\fp)$ and the ordering $\le$ on tiles, the sets $E(\fp)$, $\fp \in \mathfrak{A}$ are pairwise disjoint. This will be used repeatedly below. Set
\begin{equation*}
    \tilde{q}=\frac {2q}{1+q}\,.
\end{equation*}
Since $1< q\le 2$, we have $1<\tilde{q}<q\le 2$.
\begin{lemma}[dens2 antichain]
\label{dens2-antichain}
\leanok
\lean{Dens2Antichain}
We have that
\begin{equation}\label{eqttt9}
  \left|\int \overline{g} T_{\mathfrak{A}} f\, \mathrm{d}\mu\right|\le
  2^{111a^3}({q}-1)^{-1} \dens_2(\mathfrak{A})^{\frac 1{\tilde{q}}-\frac 12} \|f\|_2\|g\|_2\, .
\end{equation}
\end{lemma}
\begin{proof}
Let $\mathcal{B}$ be the collection of balls
\begin{equation*}
    \{B(\pc(\fp), 8D^{\ps(\fp)}) \ : \ \fp \in \mathfrak{A}\}.
\end{equation*}
From disjointness of the $E(\fp), \fp \in \mathfrak{A}$, the triangle inequality and the kernel support \eqref{supp-Ks} and upper bound \eqref{eq-Ks-size} it follows that for every $x$
\begin{equation}\label{hlmbound}
    |T_{\mathfrak{A}} f(x)|\le 2^{107 a^3} \sup_{x \in B \in \mathcal{B}} \frac{1}{\mu(B)} \int_B |f| \, \mathrm{d}\mu \, .
\end{equation}
We have $f=\mathbf{1}_Ff$. Using H\"older's inequality, that $1 < \tilde q \le 2$ and the definition of $\dens_2$, we obtain for
each $x\in B'$ and each $B'\in \mathcal{B}$
\begin{equation*}
    \frac 1{\mu(B')}\int_{B'} |f(y)|\, \mathrm{d}\mu(y)
    \le \left(M (|f|^{\frac {2{\tilde{q}}}{3{\tilde{q}}-2}})(x)\right)^{\frac 32-\frac 1{\tilde{q}}}
\dens_2(\mathfrak{A})^{\frac 1{\tilde{q}}-\frac 12}\, .
\end{equation*}
Combining this with \eqref{hlmbound} and boundedness of the Hardy-Littlewood maximal function completes the proof.
\end{proof}

\begin{lemma}[dens1 antichain]\label{dens1-antichain}
\uses{tile-correlation,antichain-tile-count}
Set $p:=4a^4$. Then
    \begin{equation}\label{eqttt3}
  \left|\int \overline{g} T_{\mathfrak{A}} f\, \mathrm{d}\mu\right|\le
   2^{117a^3}\dens_1(\mathfrak{A})^{\frac 1{2p}} \|f\|_2\|g\|_2\,.
\end{equation}
\end{lemma}

\begin{proof}
 We have by expanding the square
\begin{equation}
    \label{eqtts2}
    \int \Big|\sum_{\fp\in \mathfrak{A}}T^*_{\fp}g(y)\Big|^2\, \mathrm{d}\mu(y)\le2 \sum_{\fp\in \mathfrak{A}} \sum_{\fp'\in \mathfrak{A}: \ps(\fp')\le \ps(\fp)}
    \Big|\int T^*_{\fp}g(y)\overline{T^*_{\fp'}g(y)}\, \mathrm{d}\mu(y)\Big|\,.
\end{equation}
Define for $\fp\in \fP$ the ball
\begin{equation*}
    B(\fp):=B(\pc(\fp), 14D^{\ps(\fp)})
\end{equation*}
and the collection of tiles interacting with $\fp$
\begin{equation*}
    \mathfrak{A}(\fp):=\{\fp'\in\mathfrak{A}: \ps(\fp')\leq \ps(\fp) \land \scI(\fp') \subset B(\fp)\}.
\end{equation*}
Using \Cref{tile-correlation} and the doubling property \eqref{doublingx}, we estimate \eqref{eqtts2} by
\begin{equation}\label{eqtts3}
     \le 2^{232a^3+6a+1} \sum_{\fp\in \mathfrak{A}}
    \int_{E(\fp)}|g|(y) h(\fp)\, \mathrm{d}\mu(y)
\end{equation}
with $h(\fp)$ defined as
\begin{equation*}
    \frac 1{\mu(B(\fp))}\int \sum_{\fp'\in \mathfrak{A}(\fp)}
    {(1+d_{\fp'}(\fcc(\fp'), \fcc(\fp)))^{-1/(2a^2+a^3)}}(\mathbf{1}_{E(\fp')}|g|)(y')\, \mathrm{d}\mu(y')\,.
\end{equation*}
Note that $p\geq 4$ since $a>4$.
By H\"older, using $|g|\le \mathbf{1}_G$ and $E(\fp')\subset B(\fp)$,
\begin{equation*}
    h(\fp) \le \frac{\|g\mathbf{1}_{B(\fp)}\|_{p'}}{\mu(B(\fp))}
    \Big\|\sum_{\fp'\in\mathfrak{A}(\fp)}(1+d_{\fp'}(\fcc(\fp), \fcc(\fp')))^{-1/(2a^2+a^3)}\mathbf{1}_{E(\fp')}\mathbf{1}_G\Big\|_{p}\, .
\end{equation*}
We apply \Cref{antichain-tile-count} to estimate this by
\begin{equation*}
    2^{5a}
    \dens_1(\mathfrak{A})^{\frac 1p}\frac{\|g\mathbf{1}_{B(\fp)}\|_{p'}}{\mu(B(\fp))^{\frac{1}{p'}}} \le 2^{5a}
    \dens_1(\mathfrak{A})^{\frac 1p} (M|g|^{p'})^{\frac{1}{p'}}\, ,
\end{equation*}
where we used that the $I({\fp'})$ with $\fp'\in \mathfrak{A}(\fp)$ are contained in $B(\fp)$.
Hence we may estimate \eqref{eqtts3} by
\begin{equation*}
    \le 2^{232a^3+11a + 1} { \dens_1(\mathfrak{A})^{\frac 1p}}\sum_{\fp\in \mathfrak{A}}\int_{E(\fp)}|g|(y)(M|g|^{p'})^{\frac{1}{p'}} \, \mathrm{d}\mu(y)\,.
\end{equation*}
Since the sets $E(\fp)$ are pairwise disjoint, the lemma now follows from boundedness of the Hardy-Littlewood maximal function.
\end{proof}
The following $TT^*$ estimate will be proved in \Cref{sec-tile-operator}.
\begin{lemma}[tile correlation]
    \label{tile-correlation}
    \leanok
    \lean{Tile.correlation_le, Tile.correlation_zero_of_ne_subset}
    Let $\fp, \fp'\in \fP$ with
    $\ps({\fp'})\leq \ps({\fp})$.
    Then
    \begin{equation*}
        \left|\int T^*_{\fp'}g\overline{T^*_{\fp}g}\right|
    \end{equation*}
    \begin{equation}
    \label{eq-basic-TT*-est}
        \le 2^{232a^3}\frac{(1+d_{\fp'}(\fcc(\fp'), \fcc(\fp)))^{-1/(2a^2+a^3)}}{\mu(\scI(\fp))}\int_{E(\fp')}|g|\int_{E(\fp)}|g|\,.
    \end{equation}
    Moreover, the term \eqref{eq-basic-TT*-est} vanishes unless
    \begin{equation}
        \label{eq-tt*sup}
        \scI(\fp') \subset B(\pc(\fp), 14D^{\ps(\fp)})\, .
    \end{equation}
\end{lemma}

The following lemma will be proved in \Cref{subsec-geolem}.
\begin{lemma}[antichain tile count]
    \label{antichain-tile-count}
\leanok
\lean{Antichain.tile_count}
    \uses{global-antichain-density}
    Set $p:=4a^4$ and let $p'$ be the dual exponent of $p$, that is $1/p+1/p'=1$.
    For every $\mfa\in\Mf$ and every subset $\mathfrak{A}'$ of $\mathfrak{A}$ we have
    \begin{equation}
        \label{eq-antichain-Lp}
        \Big\|\sum_{\fp\in\mathfrak{A}'}(1+d_{\fp}(\fcc(\fp), \mfa))^{-1/(2a^2+a^3)}\mathbf{1}_{E(\fp)}\mathbf{1}_G\Big\|_{p}
    \end{equation}
    \begin{equation*}
        \le
        2^{5a}\dens_1(\mathfrak{A})^{\frac 1p}\mu\left(\cup_{\fp\in\mathfrak{A}'}I({\fp})\right)^{\frac 1p}\, .
    \end{equation*}
\end{lemma}

\subsection{The tile correlation lemma}\label{sec-tile-operator}

We start with a geometric estimate for two tiles.
\begin{lemma}\label{tile-uncertainty}
\leanok
\lean{Tile.uncertainty}
    Let $\fp_1, \fp_2\in \fP$ with
    $B(\pc(\fp_1),5D^{\ps(\fp_1)}) \cap B(\pc(\fp_2),5D^{\ps(\fp_2)}) \ne \emptyset$ and
$\ps({\fp_1})\leq \ps({\fp_2})$. For each $x_1\in E(\fp_1)$ and
$x_2\in E(\fp_2)$ we have
\begin{equation*}
  1+d_{\fp_1}(\fcc(\fp_1), \fcc(\fp_2))\le
    2^{8a}(1 + d_{B(x_1, D^{\ps(\fp_1)})}(\tQ(x_1),\tQ(x_2)))\, .
\end{equation*}
\end{lemma}
\begin{proof}
\leanok
Let $i\in \{1,2\}$.
By definition \eqref{defineep} of $E$,
we have
\begin{equation}\label{dponetwo}
    d_{\fp_i}(\tQ(x_i),\fcc(\fp_i)) < 1\, .
\end{equation}
By the triangle inequality and \eqref{eq-vol-sp-cube} we have $\scI(\fp_1) \subset B(\pc(\fp_2),14D^{\ps(\fp_2)})$.
Thus, using again \eqref{eq-vol-sp-cube} and the doubling property \eqref{firstdb},
\begin{equation}\label{tgeo0.5}
    d_{\fp_1}(\tQ(x_2), \fcc(\fp_2)) \le 2^{6a} d_{\fp_2}(\tQ(x_2), \fcc(\fp_2)) \le 2^{6a}\,.
\end{equation}
By the triangle inequality, we obtain from \eqref{dponetwo} and
\eqref{tgeo0.5}
\begin{equation}\label{tgeo1}
     1+d_{\fp_1}(\fcc(\fp_1), \fcc(\fp_2))\le 2 + 2^{6a} +d_{\fp_1}(\tQ(x_1), \tQ(x_2))\, .
\end{equation}
As $x_1\in \scI(\fp_1)$ we have by \eqref{eq-vol-sp-cube} and the triangle inequality
\begin{equation*}
    \scI(\fp_1)\subset B(x_1,8D^{\ps(\fp_1)})\, .
\end{equation*}
Applying monotonicity \eqref{monotonedb} of the metrics $d_B$ and the doubling property \eqref{firstdb} in \eqref{tgeo1} completes the proof.
\end{proof}

Now we prove \Cref{tile-correlation}.
\begin{proof}[Proof of \Cref{tile-correlation}]
    \proves{tile-correlation}
The support of $K_s$, see \eqref{supp-Ks}, and the triangle inequality imply that the left-hand side of \eqref{eq-basic-TT*-est} vanishes unless
$B(\pc(\fp_1),5D^{\ps(\fp_1)}) \cap B(\pc(\fp_2),5D^{\ps(\fp_2)}) \neq \emptyset$. We assume this for the remainder of the proof. Then \eqref{eq-tt*sup} follows from the triangle inequality and
the squeezing property \eqref{eq-vol-sp-cube}.

To prove \eqref{eq-basic-TT*-est} we expand the left-hand side and apply Fubini and the triangle inequality to bound it from above by
\begin{equation*}
    \int_{E(\fp_1)} \int_{E(\fp_2)} {\bf I}(x_1, x_2) |g(x_1)||g(x_2)|\, \mathrm{d}\mu(x_1)\mathrm{d}\mu(x_2)\,
\end{equation*}
with
\begin{equation*}
    {\bf I}(x_1, x_2):=
    \left|\int
    e(-\tQ(x_1)(y)+\tQ(x_2)(y))\varphi_{x_1,x_2}(y)
    \mathrm{d}\mu(y) \right|
\end{equation*}
and
\begin{equation*}
 \varphi_{x_1, x_2}(y) := \overline{K_{s_1}(x_1, y)}
 K_{s_2}(x_2, y) \, .
\end{equation*}
Note that by \eqref{supp-Ks} the function $\varphi$ is supported in $B(x_1, D^{s_1})$ and by \eqref{eq-Ks-size} and \eqref{eq-Ks-smooth} we have with $\tau = 1/a$
\begin{equation*}
  \|\varphi_{x_1, x_2}\|_{C^\tau(B(x_1, D^{s_1}))}\le
\frac{2^{231 a^3}}{\mu(B(x_1, D^{s_1}))\mu(B(x_2, D^{s_2}))}
      \, .
\end{equation*}
We can therefore estimate ${\bf I}(x_1, x_2)$ for fixed $x_1\in E(\fp_1)$ and
$x_2\in E(\fp_2)$ with the van-der-Corput type estimate from
\Cref{Holder-van-der-Corput} on the ball $B' := B(x_1, D^{\ps(\fp_1)})$. We obtain
\begin{equation*}
 {\bf I}(x_1, x_2) \le  \frac{2^{231a^3+8a}}
 {\mu(B(x_2, D^{\ps(\fp_2)}))}
       (1 + d_{B'}(\tQ(x_1),\tQ(x_2)))^{-1/(2a^2+a^3)}\,.
\end{equation*}
Using \Cref{tile-uncertainty} and $a\ge 1$, we estimate this by
\begin{equation}\label{eqa2}
 \le \frac{2^{231a^3 + 8a + 1}}
 {\mu(B(x_2, D^{\ps(\fp_2)}))}
       (1+d_{\fp_1}(\fcc(\fp_1), \fcc(\fp_2)))^{-1/(2a^2+a^3)}\,.
\end{equation}
As $x_2\in E(\fp_2)$ we have $\rho(x_2,\pc(\fp_2))\le 4D^{\ps(\fp_2)}$, thus by the squeezing property \eqref{eq-vol-sp-cube} and the triangle inequality
$\scI(\fp_2)\subset B(x_2,8D^{\ps(\fp_2)})$. Applying the doubling property \eqref{doublingx} in the first factor of \eqref{eqa2}, we obtain \eqref{eq-basic-TT*-est}.
\end{proof}

\subsection{The tile count lemma}
\label{subsec-geolem}

We start with some auxiliary lemmas.
\begin{lemma}[tile reach]\label{tile-reach}
\leanok
\lean{Antichain.tile_reach}
\uses{monotone-cube-metrics}
Let $\mfa\in \Mf$ and $N\ge0$ be an integer.
Let $\fp, \fp'\in \fP$ with
\begin{equation}\label{eqassumedismfa}
    d_{\fp}(\fcc(\fp), \mfa))\le 2^N\qquad\text{and}\qquad
    d_{\fp'}(\fcc(\fp'), \mfa))\le 2^N\, .
\end{equation}
Assume $\scI(\fp)\subset \scI(\fp')$ and $\ps(\fp)<\ps(\fp')$.
Then
\begin{equation}\label{lp'lp''}2^{N+2}\fp\lesssim 2^{N+2} \fp'\, .
\end{equation}
\end{lemma}

\begin{proof}
\leanok
By the monotonicity of  \Cref{monotone-cube-metrics}, we have
\begin{equation*}
     d_{\fp}(\fcc(\fp'),\mfa)
     \le d_{\fp'}(\fcc(\fp'),\mfa)
     \le 2^{N} \, .
\end{equation*}
Together with \eqref{eqassumedismfa} and the triangle inequality, we obtain
\begin{equation}\label{eqdistqpqp}
    d_{\fp}(\fcc(\fp'),\fcc(\fp))\le 2^{N+1} \, .
\end{equation}
To show \eqref{lp'lp''}, pick some
\begin{equation*}
    \mfa'\in B_{\fp'}(\fcc(\fp'),2^{N+2})\, .
\end{equation*}
By the doubling property \eqref{firstdb}, applied five times, we have
\begin{equation}\label{ageo1} d_{B(\pc(\fp'),8D^{\ps(\fp')})}(\fcc(\fp'),\mfa') < 2^{5a+N+2}\, .
\end{equation}
With the assumption, $\scI(\fp)\subset \scI(\fp')$,
the squeezing property \eqref{eq-vol-sp-cube}, and the triangle inequality, we have
\begin{equation*}
 B(\pc(\fp), 4D^{\ps(\fp')})
 \subseteq
B(\pc(\fp'),8D^{\ps(\fp')})\, .
\end{equation*}
Together with \eqref{ageo1} and monotonicity \eqref{monotonedb} of $d$ this implies
\begin{equation*}
    d_{B(\pc(\fp),4D^{\ps(\fp')})}(\fcc(\fp'),\mfa') < 2^{5a+N+2}\, .
\end{equation*}
Using $\ps(\fp)<\ps(\fp')$, $D=2^{100a^2}$,  $a\ge 4$, and the doubling property \eqref{seconddb} gives
\begin{equation*}
    d_{\fp}(\fcc(\fp'),\mfa') < d_{B(\pc(\fp),2^{2-5a^2-2a}D^{\ps(\fp')})}(\fcc(\fp'),\mfa') < 2^{N}\, .
\end{equation*}
Finally, we obtain with the triangle inequality and \eqref{eqdistqpqp},
\begin{equation*}
    d_{\fp}(\fcc(\fp),\mfa') < 2^{N+2}\, .
\end{equation*}
This implies \eqref{lp'lp''} and completes the proof of the lemma.
\end{proof}

For $\mfa \in \Mf$ and $N\ge 0$ define
\begin{equation*}
    \mathfrak{A}_{\mfa,N}:=\{\fp\in\mathfrak{A}: 2^{N}\le 1+d_{\fp}(\fcc(\fp), \mfa) < 2^{N+1}\} \, .
\end{equation*}

\begin{lemma}[stack density]
\label{stack-density}
\leanok
\lean{Antichain.stack_density}
Let $\mfa \in \Mf$, $N\ge 0$ and
$L\in \mathcal{D}$. Then
\begin{equation}\label{eqanti-1}
    \sum_{\fp\in\mathfrak{A}_{\mfa,N}:\scI(\fp)=L}\mu(E(\fp)\cap G)\le 2^{a(N+5)}\dens_1(\mathfrak{A})\mu(L)\, .
\end{equation}
\end{lemma}
\begin{proof}
\leanok
Let $\mfa,N,L$ be given and set
\begin{equation*}
\mathfrak{A}':=\{\fp\in\mathfrak{A}_{\mfa,N}:\scI(\fp)=L\}\, .
\end{equation*}
Let
$\fp\in\mathfrak{A}'$.
By definition \eqref{definedens1} of $\dens_1$ with $\lambda=2$ and the squeezing property \eqref{eq-freq-comp-ball},
\begin{equation}\label{eqanti-3}
\mu(E(\fp)\cap G)\le \mu(E_2(2, \fp))\le 2^{a}\dens_1(\mathfrak{A}')\mu(L)\, .
\end{equation}
By the covering property \eqref{thirddb}, applied $N+4$ times, there is a collection $\Mf'$ of at most $2^{a(N+4)}$
elements such that
\begin{equation}\label{eqanti-4}
    B_{\fp}(\mfa, 2^{N+1})\subset \bigcup_{\mfa'\in \Mf'}
    B_{\fp}(\mfa', 0.2)\, .
\end{equation}
Note that the metrics $d_{\fp'}$, $\fp' \in \mathfrak{A}'$ are all equal, since they only depend on $\scI(\fp') = L$.
Hence, each $\fcc(\fp')$ with $\fp'\in \mathfrak{A}'$
is contained in the left-hand-side
of \eqref{eqanti-4}. Moreover, the balls $B_{\fp}(\fcc(\fp'), 0.2) \subset \fc(\fp')$ with $\fp' \in \mathfrak{A}'$ are pairwise disjoint by \eqref{eq-dis-freq-cover}. Hence each $B_\fp(\mfa', 0.2)$ contains at most one of the $\fcc(\fp')$ with $\fp' \in \mathfrak{A}'$.
It follows that there are at most $2^{a(N+4)}$ elements in
$\mathfrak{A}'$. Adding \eqref{eqanti-3} over $\mathfrak{A}'$ proves
\eqref{eqanti-1}.
\end{proof}

\begin{lemma}[local antichain density]\label{local-antichain-density}
\leanok
\lean{Antichain.local_antichain_density, Antichain.Ep_inter_G_inter_Ip'_subset_E2}
Let $\mfa\in \Mf$ and {$N$} be
an integer. Let $\fp_{\mfa}$ be a tile with $\mfa\in B_{\fp_{\mfa}}(\fcc(\fp_{\mfa}), 2^{N+1})$.
Then we have
\begin{equation}\label{eqanti-0.5}
    \sum_{\fp\in\mathfrak{A}_{\mfa,N}: \ps(\fp_{\mfa})<\ps(\fp)}\mu(E(\fp)\cap G \cap \scI(\fp_{\mfa}))
    \le \mu (E_2(2^{N+3},\fp_{\mfa}))
 \, .
\end{equation}

\end{lemma}

\begin{proof}
\leanok
On the left hand side only tiles $\fp \in \mathfrak{A}_{\mfa, N}$ with $\scI(\fp_\mfa) \subset \scI(\fp)$ contribute.
Combining the assumption on $\fp_\mfa$ with $\fp \in \mathfrak{A}_{\mfa,N}$ and \Cref{tile-reach}, we conclude
$2^{N+3}\fp_{\mfa} \lesssim 2^{N+3}\fp$ and hence
\begin{equation*}
    E(\fp)\cap G \cap \scI(\fp_{\mfa}) \subset E_2(2^{N+3},\fp_{\mfa})\, .
\end{equation*}
Using disjointness of the various $E(\fp)$ with $\fp\in \mathfrak{A}$, we obtain \eqref{eqanti-0.5}.
\end{proof}
\begin{lemma}[global antichain density]
\label{global-antichain-density}
\leanok
\lean{Antichain.global_antichain_density}
\uses{stack-density,local-antichain-density}
Let $\mfa\in Q(X)$ and let $N\ge 0$ be
an integer. Then we have
\begin{equation*}
    \sum_{\fp\in\mathfrak{A}_{\mfa,N}}\mu(E(\fp)\cap G)
    \le
 2^{101a^3+Na}\dens_1(\mathfrak{A})\mu\left(\cup_{\fp\in\mathfrak{A}}I({\fp})\right)\,.
\end{equation*}
\end{lemma}

\begin{proof}
Fix $\mfa$ and $N$. The contribution of tiles $\fp$ with $s(\fp) = -S$ is taken care of by Lemma \ref{stack-density}. Let
$\mathfrak{A}'$ be the set of remaining tiles $\fp\in\mathfrak{A}_{\mfa,N}$ such that $\scI(\fp)\cap G \ne \emptyset$ and $s(\fp) > -S$.

Let $\mathcal{L}$ be the collection of dyadic cubes $I\in\mathcal{D}$ such that $I\subset \scI(\fp)$ for some $\fp\in\mathfrak{A}'$ and $\scI(\fp)\not \subset I$ for all $\fp\in\mathfrak{A}'$. Let $\mathcal{L}^*$ be the set of maximal elements in $\mathcal{L}$ with respect to set inclusion. By the grid properties, the elements in $\mathcal{L}^*$ are pairwise disjoint and we have
\begin{equation}\label{eqdecAprime}
    \bigcup\mathcal{L}^*=\bigcup_{\fp \in \mathfrak{A}'}\scI(\fp)\, .
\end{equation}
Using the partition \eqref{eqdecAprime}, it suffices to show that for each $L\in \mathcal{L}^*$
\begin{equation}\label{eqanti0}
    \sum_{\fp\in\mathfrak{A}'}\mu(E(\fp)\cap G \cap L)
    \le
    2^{101a^3+aN}
    \dens_1(\mathfrak{A})\mu(L)\,.
\end{equation}
Fix $L\in \mathcal{L}^*$.
By definition of $L$, there exists an element $\fp'\in \mathfrak{A}'$ of minimal scale such that $L\subset \scI(\fp')$.
Let $L'\in \mathcal{D}$ be the unique cube with $s(L')=s(L)+1$ and $L \subset L'$.

We split the left-hand side of \eqref{eqanti0} as
\begin{equation}\label{eqanti1}
    \sum_{\fp\in\mathfrak{A}':\scI(\fp)=L'}\mu(E(\fp)\cap G\cap L)
    +
     \sum_{\fp\in\mathfrak{A}':\scI(\fp)\neq L'}\mu(E(\fp)\cap G\cap L)\, ,
\end{equation}
The first term satisfies the required bound by \Cref{stack-density} and the doubling property \eqref{doublingx}.

For the second term, note that by the maximality of $L$ in $\mathcal{L}$, there exists $\fp''\in \mathfrak{A}'$ with
$\scI(\fp'')\subset L'$. If $\scI(\fp'') = L'$, then we set $\fp_\mfa = \fp''$. Otherwise, we use that by the covering property \eqref{eq-dis-freq-cover}, there exists a unique $\fp_{\mfa}$ with $\scI(\fp_{\mfa})=L'$ such that $\mfa\in \fc(\fp_{\mfa})$, and we take this as the definition of $\fp_\mfa$.
Using that $\fp'' \in \mathfrak{A}_{\mfa,N}$ and \Cref{tile-reach}, we conclude in both cases that
\begin{equation*}
    2^{N+3}\fp'' \lesssim 2^{N+3}\fp_{\mfa} \, .
\end{equation*}
As $\fp''\in \mathfrak{A}'$, we have by the definition \eqref{definedens1} of $\dens_1$ that
\begin{equation}\label{pmfadens}
   \mu(E_2(2^{N+3}, \fp_{\mfa}))\le 2^{Na+3a}\dens_1(\mathfrak{A}) {\mu(L')}\, .
\end{equation}
Now let $\fp$ be any tile in the second sum in \eqref{eqanti1}. It follows by the dyadic property \eqref{dyadicproperty}
and the definition of $L$ that
$L\subset \scI(\fp)$ and $L\neq \scI(\fp)$ and in fact $L'\subset \scI(\fp)$ and $L'\neq \scI(\fp)$, so we conclude $s(L')<\ps(\fp)$.
By \Cref{local-antichain-density}, we thus estimate the second term in \eqref{eqanti1} by
\begin{equation*}
    \sum_{\fp\in\mathfrak{A}':\scI(\fp)\neq L'}\mu(E(\fp)\cap G\cap L')
    \le \mu (E_2(2^{N+3},\fp_{\mfa}))\, .
\end{equation*}
With
\eqref{pmfadens} and the doubling property \eqref{doublingx}, this proves \eqref{eqanti0}.
\end{proof}

We turn to the proof of \Cref{antichain-tile-count}.

\begin{proof}[Proof of \Cref{antichain-tile-count}]
\proves{antichain-tile-count}

Using that $\mathfrak{A}$ is the disjoint union of the
$\mathfrak{A}_{\mfa,N}$ with $N\ge 0$ and that the sets $E(\fp)$ with $\fp \in \mathfrak{A}$ are pairwise disjoint
we estimate the $p$-th power of \eqref{eq-antichain-Lp} by
\begin{equation*}
\sum_{N\ge 0}2^{-pN/(2a^2+a^3)} \sum_{\fp\in\mathfrak{A}_{\mfa,N}}\mu(E(\fp)\cap G)\, .
\end{equation*}
Using \Cref{global-antichain-density}, we estimate the last display by
\begin{equation}\label{eqanti21}
    \le \sum_{N \ge 0} 2^{-pN/(2a^2+a^3)+101a^3+Na}\dens_1(\mathfrak{A})\mu\left(\cup_{\fp\in\mathfrak{A}}\scI(\fp)\right).
\end{equation}
Recalling $p = 4a^4$ and
using  $a\ge 4$, we conclude
\begin{equation*}
    pN/(2a^2+a^3)\ge
    4a^4N/(3a^3) \ge Na+N\, .
\end{equation*}
Hence we have for \eqref{eqanti21} the upper bound
\begin{equation*}
\le 2^{101a^3} \sum_{N \ge 0} 2^{-N}\dens_1(\mathfrak{A})\mu\left(\cup_{\fp\in\mathfrak{A}}\scI(\fp)\right).
\end{equation*}
Summing over $N\ge 0$ and taking the $p$-th root proves the lemma.
\end{proof}

\section{Proof of the Forest Operator Proposition}

\label{treesection}

 After proving a series of auxiliary lemmas, we assemble the proof of \Cref{forest-operator} in Subsection \ref{subsec-forest}.
Fix a forest $(\fU, \fT)$.

\subsection{The pointwise tree estimate}
 The main result of this subsection is the pointwise estimate for operators associated to sets $\fT(\fu)$
 stated in \Cref{pointwise-tree-estimate}.
For $\fu \in \fU$ and $x\in X$, we define
$$
    \sigma (\fu, x):=\{\ps(\fp):\fp\in \fT(\fu), x\in E(\fp)\},
$$
$$
    \overline{\sigma} (\fu, x) := \max \sigma(\fT(\fu), x) \qquad \text{and} \qquad \underline{\sigma} (\fu, x) := \min\sigma(\fT(\fu), x)\,.
$$
By the convexity property \eqref{forest2}, we have for each $\fu \in \fU$
\begin{equation}
    \label{convex-scales}
    \sigma(\fu, x) = \mathbb{Z} \cap [\underline{\sigma} (\fu, x), \overline{\sigma} (\fu, x)]\,.
\end{equation}
For a nonempty collection of tiles $\mathfrak{S} \subset \fP$, we define $\mathcal{J}_0(\mathfrak{S})$
to be the collection of all dyadic cubes $J \in \mathcal{D}$ such that $s(J) = -S$ or
$$
    \scI(\fp) \not\subset B(c(J), 100D^{s(J) + 1})
$$
for all $\fp \in \mathfrak{S}$.  We further define
$\mathcal{L}_0(\mathfrak{S})$ to be the collection of dyadic cubes $L \in \mathcal{D}$ such that $s(L) = -S$, or there exists $\fp \in \mathfrak{S}$ with $L \subset \scI(\fp)$ and there exists no $\fp \in \mathfrak{S}$ with $\scI(\fp) \subset L$.
Let
\begin{equation}
\label{eq-def-jl}
    \mathcal{J}(\mathfrak{S}),\ \mathcal{L}(\mathfrak{S})
\end{equation}
  be the collection of inclusion maximal cubes in $\mathcal{J}_0(\mathfrak{S})$ and  $\mathcal{L}_0(\mathfrak{S})$, respectively. Both collections partition the union of all grid cubes:
\begin{equation*}
    \bigcup_{I \in \mathcal{D}} I = \dot{\bigcup_{J \in \mathcal{J}(\mathfrak{S})}} J = \dot{\bigcup_{L \in \mathcal{L}(\mathfrak{S})}} L\,.
\end{equation*}
For a finite set of pairwise disjoint cubes $\mathcal{C}$, define the projection operator
$$
    P_{\mathcal{C}}f(x) :=\sum_{J\in\mathcal{C}}\mathbf{1}_J(x) \frac{1}{\mu(J)}\int_J f(y) \, \mathrm{d}\mu(y)\,.
$$
We denote by $I_s(x)$ the unique grid cube of scale $s$ containing $x$. Define for $\mfa \in \Mf$ the nontangential maximal operator
\begin{equation*}
    T_{\mathcal{N}}^\mfa f(x) := \sup_{-S \le s_1} \sup_{x' \in I_{s_1}(x)} \sup_{\substack{s_1 \le s_2 \le S\\ D^{s_2-1} \le R_Q(\mfa, x')}} \left| \sum_{s = s_1}^{s_2} \int K_s(x',y) f(y) \, \mathrm{d}\mu(y) \right|\,.
\end{equation*}
Up to boundary terms that are controlled by the Hardy-Littlewood maximal function, this operator is controlled by the maximal operator $T_Q^\mfa$ defined in \eqref{def-lin-star-op}. This implies the estimate
\begin{equation}
    \label{nontangential-operator-bound}
    \|T_{\mathcal{N}}^{\mfa} f\|_2 \le 2^{102a^3} \|f\|_2\,.
\end{equation}
We define also for each $\fu \in \fU$ the auxiliary operator
\begin{equation*}
    S_{1,\fu}f(x) :=\sum_{I\in\mathcal{D}} \mathbf{1}_{I}(x) \sum_{\substack{J\in \mathcal{J}(\fT(\fu))\\
    J\subset B(c(I), 16 D^{s(I)})\\ s(J) \le s(I)}} \frac{D^{(s(J) - s(I))/a}}{\mu(B(c(I), 16D^{s(I)}))}\int_J |f(y)| \, \mathrm{d}\mu(y)\,.
\end{equation*}

\begin{lemma}[pointwise tree estimate]
    \label{pointwise-tree-estimate}
    \leanok
    \lean{TileStructure.Forest.pointwise_tree_estimate}
    Let $\fu \in \fU$ and $L \in \mathcal{L}(\fT(\fu))$. Let $x, x' \in L$.
    Then for all bounded functions $f$ with bounded support
    $$
        \left| T_{\fT(\fu)}[ e(-\fcc(\fu))f](x)\right|
    $$
    \begin{equation}
        \label{eq-LJ-ptwise}
        \leq 2^{129a^3}(M+S_{1,\fu})P_{\mathcal{J}(\fT(\fu))}|f|(x')+|T_{\mathcal{N}}^{\fcc(\fu)} P_{\mathcal{J}(\fT(\fu))}f(x')|\, .
    \end{equation}
\end{lemma}

\begin{proof}
    The left hand side of \eqref{eq-LJ-ptwise} equals
    \begin{equation*}
        \Bigg| \sum_{s \in \sigma(\fu, x)} \int
        e(\fu,x,y)
    K_s(x,y)f(y) \, \mathrm{d}\mu(y) \Bigg|\,.
    \end{equation*}
    with
    \[e(\fu,x,y):=e(-\fcc(\fu)(y) + \tQ(x)(y) + \fcc(\fu)(x) -\tQ(x)(x))\, .\]
    Using the triangle inequality, we bound this by the sum of three terms:
    \begin{equation}
        \label{eq-term-A}
        \le \Bigg| \sum_{s \in \sigma(\fu, x)} \int (
        e(\fu,x,y)
    -1)       K_s(x,y)f(y) \, \mathrm{d}\mu(y) \Bigg|
    \end{equation}
    \begin{equation}
        \label{eq-term-B}
        + \Bigg| \sum_{s \in \sigma(\fu, x)} \int K_s(x,y) P_{\mathcal{J}(\fT(\fu))} f(y) \, \mathrm{d}\mu(y) \Bigg|
    \end{equation}
    \begin{equation}
        \label{eq-term-C}
        + \Bigg| \sum_{s \in \sigma(\fu, x)} \int K_s(x,y) (f(y) - P_{\mathcal{J}(\fT(\fu))} f(y)) \, \mathrm{d}\mu(y) \Bigg|\,.
    \end{equation}
    Unpacking of the definitions shows that \eqref{eq-term-B} is bounded by
    \[
        T_{\mathcal{N}}^{\fcc(\fu)} P_{\mathcal{J}(\fT(\fu))} f(x').
    \]
    The proof is completed using the bounds for the other two terms proven in \Cref{first-tree-pointwise} and \Cref{third-tree-pointwise}.
\end{proof}

\begin{lemma}[first tree pointwise]
    \label{first-tree-pointwise}
    \leanok
    \lean{TileStructure.Forest.first_tree_pointwise}
    \uses{convex-scales}
    For all $\fu \in \fU$, all $L \in \mathcal{L}(\fT(\fu))$, all $x, x' \in L$ and all bounded $f$ with bounded support, we have
    $$
        \eqref{eq-term-A} \le 10 \cdot 2^{104a^3} M P_{\mathcal{J}(\fT(\fu))}|f|(x')\,.
    $$
\end{lemma}

\begin{proof}
    \leanok
    Let $\fp \in \fT(\fu)$ with $s = \ps(\fp) \in \sigma(\fu,x)$.
    If $x, y \in X$ with $K_s(x,y)\neq 0$, then by the support assumption \eqref{supp-Ks} we have $\rho(x,y)\leq 1/2 D^s$.
    Hence, by the oscillation control \eqref{osccontrol},
\begin{equation*}
    |e(\fu,x,y)
        -1|
        \leq d_{B(x, 1/2 D^{s})}(\fcc(\fu), \tQ(x))\,.
  \end{equation*}
    Let $\fp'$ be a tile with $\ps(\fp') = \overline{\sigma}(\fu, x)$ and $x \in E(\fp')$.
    Using the doubling property \eqref{firstdb} repeatedly, we bound the previous display by
    $$
        d_{B(x, 4 D^{s})}(\fcc(\fu), \tQ(x)) \leq 2^{4a} d_{\fp}(\fcc(\fu), \tQ(x)) \le 2^{4a} 2^{s - \overline{\sigma}(\fu, x)} d_{\fp'}(\fcc(\fu), \tQ(x))\,.
    $$
    Since $\fcc(\fu) \in B_{\fp'}(\fcc(\fp'), 4)$ by the tree property \eqref{forest1} and $\tQ(x) \in \Omega(\fp') \subset B_{\fp'}(\fcc(\fp'), 1)$ by the squeezing property \eqref{eq-freq-comp-ball}, the last display is
    $$
        \le 5 \cdot 2^{4a} 2^{s - \overline{\sigma}(\fu, x)} \,.
    $$
    Using the pointwise kernel bound \eqref{eq-Ks-size}, it follows that
    $$
        \eqref{eq-term-A} \le 5\cdot 2^{103a^3} \sum_{s\in\sigma(x)}2^{s - \overline{\sigma}(\fu, x)} \frac{1}{\mu(B(x,D^s))}\int_{B(x,0.5D^{s})}|f(y)|\,\mathrm{d}\mu(y)\,.
    $$
    $$
         \le 5\cdot 2^{103a^3} \sum_{s\in\sigma(x)}2^{s - \overline{\sigma}(\fu, x)} \frac{1}{\mu(B(x,D^s))}\sum_{\substack{J \in \mathcal{J}(\fT(\fu))\\J \cap B(x, 0.5D^s) \ne \emptyset} }\int_{J}|f(y)|\,\mathrm{d}\mu(y)\,.
    $$
    This expression does not change if we replace $|f|$ by $P_{\mathcal{J}(\fT(\fu))}|f|$. Further, if $J \in \mathcal{J}(\fT(\fu))$ with $B(x, 0.5 D^s) \cap J \ne \emptyset$ then by the triangle inequality and the definition of $\mathcal{J}$ we obtain $J \subset B(\pc(\fp_s), 16 D^s)$. Hence the last display is
    $$
        \le 5\cdot 2^{103a^3} \sum_{s\in\sigma(x)}2^{s - \overline{\sigma}(\fu, x)} \frac{1}{\mu(B(x,D^s))}\int_{B(\pc(\fp_s),16 D^s)}P_{\mathcal{J}(\fT(\fu))}|f(y)|\,\mathrm{d}\mu(y)\,.
    $$
    Combined with the doubling property \eqref{doublingx}, this completes the estimate for the term \eqref{eq-term-A} and thus the lemma.
\end{proof}

\begin{lemma}[third tree pointwise]
    \label{third-tree-pointwise}
    \leanok
    \lean{TileStructure.Forest.third_tree_pointwise}
    For all $\fu \in \fU$, all $L \in \mathcal{L}(\fT(\fu))$, all $x, x' \in L$ and all bounded $f$ with bounded support, we have
    \begin{equation*}
          \eqref{eq-term-C} \le 2^{128a^3} S_{1,\fu} P_{\mathcal{J}(\fT(\fu))}|f|(x')\,.
    \end{equation*}
\end{lemma}

\begin{proof}
    \leanok
    We have for $J \in \mathcal{J}(\fT(\fu))$:
    $$
        \int_J K_{s}(x,y)(1 - P_{\mathcal{J}(\fT(\fu))})f(y) \, \mathrm{d}\mu(y)
    $$
    \begin{equation*}
        = \int_J \frac{1}{\mu(J)} \int_J K_s(x,y) - K_s(x,z) \, \mathrm{d}\mu(z) \,f(y) \, \mathrm{d}\mu(y)\,.
    \end{equation*}
    By the kernel regularity \eqref{eq-Ks-smooth} and the squeezing property \eqref{eq-vol-sp-cube}, we have for $y, z \in J$
    $$
        |K_s(x,y) - K_s(x,z)| \le \frac{2^{127a^3}}{\mu(B(x, D^s))} \left(\frac{8 D^{s(J)}}{D^s}\right)^{1/a}\,.
    $$
    Suppose that $s \in \sigma(\fu, x)$.
    As in the proof of Lemma \ref{first-tree-pointwise}, if $K_s(x,y) \ne 0$ for some $y \in J \in \mathcal{J}(\fT(\fu))$ then $J \subset B(x, 16 D^s)$ and $s(J) \le \ps(\fp)$. Thus, we can estimate \eqref{eq-term-C} by
    $$
        2^{127a^3 + 3/a}\sum_{\fp\in \mathfrak{T}}\frac{\mathbf{1}_{E(\fp)}(x)}{\mu(B(x,D^{\ps(\fp)}))}\sum_{\substack{J\in \mathcal{J}(\fT(\fu))\\J\subset B(x, 16D^{\ps(\fp)})\\ s(J) \le \ps(\fp)}} D^{(s(J) - \ps(\fp))/a} \int_J |f|\,.
    $$
    By \eqref{eq-dis-freq-cover} and definition \eqref{defineep}, the sets $E(\fp)$ for tiles $\fp$ with $\scI(\fp) = I$ are pairwise disjoint. It follows from the definition of $\mathcal{L}(\fT(\fu))$ that $x \in \scI(\fp)$ if and only if $x' \in \scI(\fp)$, thus we can estimate the sum over such $\mathbf{1}_{E(\fp)}(x)$ by $\mathbf{1}_{I}(x')$. Using also the doubling property \eqref{doublingx}, we estimate the last display by
    $$
        \le 2^{128a^3}\sum_{I \in \mathcal{D}} \frac{\mathbf{1}_{I}(x')}{\mu(B(c(I), 16D^{s(I)}))}\sum_{\substack{J\in \mathcal{J}(\fT(\fu))\\J\subset B(x, 16 D^{s(I)})\\ s(J) \le s(I)}} D^{(s(J) - s(I))/a} \int_J |f|
    $$
    $$
         = 2^{128a^3} S_{1,\fu} P_{\mathcal{J}(\fT(\fu))}|f|(x')\,.
    $$
    This completes the proof of the lemma.
\end{proof}

\subsection{An auxiliary \texorpdfstring{$L^2$}{L2} tree estimate}

The main result of this subsection is the following estimate on $L^2$ for operators associated to trees.

\begin{lemma}[tree projection estimate]
    \label{tree-projection-estimate}
    \leanok
    \lean{TileStructure.Forest.tree_projection_estimate}
    Let $\fu \in \fU$.
    Then we have for all $f, g$ bounded with bounded support
    \begin{equation*}
         \Big|\int_X  \bar g T_{\fT(\fu)}f \, \mathrm{d}\mu \Big| \le 2^{130a^3}\|P_{\mathcal{J}(\fT(\fu))}|f|\|_{2}\|P_{\mathcal{L}(\fT(\fu))}|g|\|_{2}.
    \end{equation*}
\end{lemma}

\begin{proof}
    \proves{tree-projection-estimate}\leanok
    Let $L \in \mathcal{L}(\fT(\fu))$.
    Let $b(x')$ denote the right-hand side of \eqref{eq-LJ-ptwise} in
\Cref{pointwise-tree-estimate}. Applying this lemma to $e(\fcc(\fu)) f$, we obtain for all $y, x' \in L$
    $$
        | T_{\fT(\fu)} f(y) | \le b(x').
    $$
    Hence, choosing a fixed $x'$ in each $L$,
    $$
        \Big| \int \bar g(y) T_{\fT(\fu)}f(y) \, \mathrm{d}\mu(y) \Big| \le   \int_X \left[P_{\mathcal{L}(\fT(\fu))}|g|(y) \right] b(y) \, \mathrm{d}\mu(y) \,.
    $$
    With Cauchy-Schwarz, this is bounded by $\|P_{\mathcal{L}(\fT(\fu))}|g|\|_2  \|b\|_2$. The bounds for $b$ following from \Cref{boundary-operator-bound} below and \eqref{nontangential-operator-bound} then complete the proof.
\end{proof}
Denote $B(I) := B(c(I), 16 D^{s(I)})$.
\begin{lemma}[boundary operator bound]
    \label{boundary-operator-bound}
    \leanok
    \lean{TileStructure.Forest.boundary_operator_bound}
    For all $\fu \in \fU$ and all bounded functions $f$ with bounded support
    \begin{equation*}
        \|S_{1,\fu}f\|_2 \le 2^{12a} \|f\|_2\,.
    \end{equation*}
\end{lemma}

\begin{proof}
    \leanok
    \proves{boundary-operator-bound}
    Let $g$ be a function with $\|g\|_2 = 1$. Then
    $$
        \Bigg|\int \bar g(y) S_{1,\fu}f(y) \, \mathrm{d}\mu(y)\Bigg|
    $$
    $$
        \le \sum_{I\in\mathcal{D}} \frac{1}{\mu(B(I))} \int_{B(I)} | g(y)| \, \mathrm{d}\mu(y) \times \sum_{\substack{J\in \mathcal{J}(\fT(\fu))\\J\subseteq B(I)\\ s(J) \le s(I)}} D^{(s(J)-s(I))/a}\int_J |f(y)| \,\mathrm{d}\mu(y)\,.
    $$
    Changing the order of summation and using $J \subset B(I)$ to bound the first average integral by $M|g|(y)$ for any $y \in J$, we obtain
    \begin{align*}
        \le \sum_{J\in\mathcal{J}(\fT(\fu))}\int_J|f(y)| M|g|(y) \, \mathrm{d}\mu(y) \sum_{\substack{I \in \mathcal{D} \, : \, J\subset B(I)\\ s(I) \ge s(J)}} D^{(s(J)-s(I))/a}.
    \end{align*}
    Using \Cref{boundary-overlap} below and summing a geometric series, the last display is
    bounded by \begin{equation*}
        2^{9a+1} \int_X|f(y)| M|g|(y) \, \mathrm{d}\mu(y)\,.
    \end{equation*}
    This completes the proof using boundedness of $M$ and duality.
\end{proof}

We used the following simple consequence of the doubling property \eqref{doublingx},
which we do not explicitly prove.

\begin{lemma}[boundary overlap]
    \label{boundary-overlap}
    \leanok
    \lean{TileStructure.Forest.boundary_overlap}
    For every cube $I \in \mathcal{D}$, there exist at most $2^{9a}$ cubes $J \in \mathcal{D}$ with $s(J) = s(I)$ and $B(I) \cap B(J) \ne \emptyset$.
\end{lemma}

\subsection{The quantitative \texorpdfstring{$L^2$}{L2} tree estimate}

This section proves the following bound for tree operators  with control by the densities.

\begin{lemma}[densities tree bound]
    \label{densities-tree-bound}
    \leanok
    \lean{TileStructure.Forest.density_tree_bound1, TileStructure.Forest.density_tree_bound2}
    \uses{tree-projection-estimate,local-dens1-tree-bound,local-dens2-tree-bound}
    Let $\fu \in \fU$. Then for all bounded $f$ with bounded support and $g$ with $|g| \le \mathbf{1}_G$
    we have
    \begin{equation*}
        \left|\int_X \bar g T_{\fT(\fu)}f \, \mathrm{d}\mu \right| \le 2^{181a^3} \dens_1(\fT(\fu))^{1/2} \|f\|_2\|g\|_2\,.
    \end{equation*}
    If additionally $|f| \le \mathbf{1}_F$, then we have
    \begin{equation*}
        \left| \int_X \bar g T_{\fT(\fu)}f\, \mathrm{d}\mu \right| \le 2^{282a^3} \dens_1(\fT(\fu))^{1/2} \dens_2(\fT(\fu))^{1/2} \|f\|_2\|g\|_2\,.
    \end{equation*}
\end{lemma}

Recall that $T_\fp f$ is supported in $E(\fp)$.  \Cref{densities-tree-bound} follows immediately from the estimate of \Cref{tree-projection-estimate}, Cauchy-Schwarz and Lemmas
\ref{local-dens1-tree-bound} and
\ref{local-dens1-tree-bound} below, controlling the size of the support of the tree operator and its adjoint.

\begin{lemma}[local dens1 tree bound]
    \label{local-dens1-tree-bound}
    \leanok
    \lean{TileStructure.Forest.local_dens1_tree_bound}
    \uses{monotone-cube-metrics}
    Let $\fu \in \fU$ and $L \in \mathcal{L}(\fT(\fu))$. Then
    \begin{equation}
    \label{eq-1density-estimate-tree}
        \mu(L \cap G \cap \bigcup_{\fp \in \fT(\fu)} E(\fp)) \le 2^{101a^3} \dens_1(\fT(\fu)) \mu(L)\,.
    \end{equation}
\end{lemma}
\begin{proof}
    \leanok
    \proves{local-dens1-tree-bound}
    We  assume there exists $\fp \in \fT(\fu)$ with $L \cap \scI(\fp) \ne \emptyset$, for otherwise
 \eqref{eq-1density-estimate-tree} is void.
 Suppose first that there exists such $\fp$ with $\ps(\fp) \le s(L)$. Then $\ps(\fp) = -S$ and $L = \scI(\fp)$ by the definition of $\mathcal{L}$. Let $\fq \in \fT(\fu)$ be another tile with $E(\fq) \cap L \ne \emptyset$. By the grid property we must have $\scI(\fp) \subset \scI(\fq)$. Using monotonicity of $d$ and the tree property \eqref{forest1}, we conclude
    \begin{align*}
        d_{\fp}(\fcc(\fp), \fcc(\fq)) &\le d_{\fp}(\fcc(\fp), \fcc(\fu)) + d_{\fp}(\fcc(\fq), \fcc(\fu))\\
        &\le d_{\fp}(\fcc(\fp), \fcc(\fu)) + d_{\fq}(\fcc(\fq), \fcc(\fu)) \le 8\,.
    \end{align*}
    By definition of $E(\fq)$ and the triangle inequality, $L \cap G \cap E(\fq) \subset E_2(9, \fp)$. We obtain
    $$
        \mu(L \cap G \cap \bigcup_{\fq \in \fT(\fu)} E(\fq)) \le \mu(E_2(9, \fp))\,.
    $$
    By the definition of $\dens_1$, this is bounded by
    $$
        9^a \dens_1(\fT(\fu)) \mu(\scI(\fp)) =9^a \dens_1(\fT(\fu)) \mu(L)\,.
    $$
    Since $a \ge 4$, \eqref{eq-1density-estimate-tree} follows in the given case.

    Now assume the opposite case that for each $\fp \in \fT(\fu)$ with $L \cap E(\fp) \ne \emptyset$, we have $\ps(\fp) > s(L)$. Let $L'$ be the parent cube of $L$ and let  $\fp'' \in \fT(\fu)$ with $\scI(\fp'') \subset L'$.
    It suffices to show that there exists a tile $\fp' \in \fP(\fT(\fu))$ with $\scI(\fp') = L'$, $d_{\fp'}(\fcc(\fp'), \fcc(\fu)) < 4$ and $9\fp'' \lesssim 9\fp'$.
    For then, let $\fq \in \fT(\fu)$ with $L \cap E(\fq) \ne \emptyset$. Then $\ps(\fq) \ge s(L')$ and $L' \subset \scI(\fq)$.
    By the same computation as in the first case we deduce $L \cap G \cap E(\fq) \subset E_2(9, \fp')$ and
    $$
        \mu(L \cap G \cap \bigcup_{\fq \in \fT(\fu)} E(\fq)) \le \mu(E_2(9, \fp')) \le 9^a \dens_1(\fT(\fu)) \mu(L')\,.
    $$
    This proves \eqref{eq-1density-estimate-tree} using the doubling property \eqref{doublingx}.

    It remains to show existence of $\fp'$ with the required properties. If $\scI(\fp'') = L'$ we can take $\fp' = \fp''$.
    Otherwise, let $\fp'$ be the unique tile such that $\scI(\fp') = L'$ and such that $\Omega(\fu) \cap \Omega(\fp') \ne \emptyset$.
    Since $\scI(\fp') \subset \scI(\fp)$ and $\fp \in \fT(\fu)$, we have $\fp' \in \fP(\fT(\fu))$.
    By the tree property \eqref{forest1}, we have $\ps(\fp') = s(L') \le \ps(\fp) < \ps(\fu)$. By \eqref{dyadicproperty} and \eqref{eq-freq-dyadic}, we conclude  $\Omega(\fu) \subset \Omega(\fp')$, and hence the distance property required of $\fp'$.
    The property $9\fp'' \lesssim 9\fp'$ follows by the triangle inequality, \eqref{forest1}, \Cref{monotone-cube-metrics} and \eqref{eq-freq-comp-ball}.
    This completes the proof.
\end{proof}

\begin{lemma}[local dens2 tree bound]
    \label{local-dens2-tree-bound}
    \leanok
    \lean{TileStructure.Forest.local_dens2_tree_bound}
    Let $\fu \in \fU$ and $J \in \mathcal{J}(\fT(\fu))$. Then
    \begin{equation}\label{eq-dens2-lemma}
        \mu(F \cap J) \le 2^{201a^3} \dens_2(\fT(\fu)) \mu(J)\,.
    \end{equation}
\end{lemma}

\begin{proof}
    \proves{local-dens2-tree-bound}
    \leanok
    Suppose first that $s(J) = S \ge 1$. Then $J$ is the maximal cube $I_0$ as in \eqref{subsetmaxcube} and the definition of
    $\mathcal{J}(\fT(\fu))$ quickly shows $s(J) = -S $, a contradiction.
    It remains to consider the case $s(J) < S$.

    We show the existence of a tile $\fp \in \fT(\fu)$
    and an $r \geq 4 D ^ {\ps(\fp)}$ such that
    \begin{equation}\label{eq-pr-condition}
        J \subset B(\pc(\fp), r),\qquad \mu(B(\pc(\fp), r)) \le 2 ^ {200a^3 + 14a} \mu(J)\, .
    \end{equation}
    This will imply with the definition
    \eqref{definedens2} of $\dens_2$  the desired
    \eqref{eq-dens2-lemma} as follows
    $$
        \mu(F \cap J) \le \mu(F \cap B(\pc(\fp), r))
    $$
    $$
        \le \dens_2(\fT(\fu)) \mu(B(\pc(\fp), r)) \le 2^{200a^3 + 14a} \dens_2(\fT(\fu))\mu(J)\,.
    $$
    %
    %
    The grid properties give a $J' \in \mathcal{D}$ with $s(J') = s(J) + 1$ and $J \subset J'$ and
    $$
        B(c(J'), 204D^{s(J')+1}) \subset B(c(J), 204D^{s(J') + 1} + 4D^{s(J')})
                                 \subset B(c(J), 2^8 D^{s(J) + 2})\,.
    $$
    The doubling property \eqref{doublingx}, squeezing property  \eqref{eq-vol-sp-cube}, and $D=2^{100a^2}$ give
    \begin{equation}
        \label{measure-comparison}
        \mu(B(c(J'), 204D^{s(J') + 1})) \leq 2 ^ {200a^3 + 10a} \mu(J)\,.
    \end{equation}
    By definition of $\mathcal{J}(\fT(\fu))$, there exists $\fp \in \fT(\fu)$
    with $$\scI (\fp)\subset B(c(J'), 100 D^{s(J') + 1})\, .$$
    If $J \subset B(\pc(\fp), 4 D^{\ps(\fp)})$, then \eqref{measure-comparison} gives \eqref{eq-pr-condition} with $r = 4D^{s(\fp)}$.
So assume $J \not \subset B(\pc(\fp), 4 D^{\ps(\fp)})$. By the triangle inequality,
    $$
        J \subset J' \subset B(c(J'), 4D^{s(J')}) \subset B(\pc(\fp), 104 D^{s(J') + 1})\,,
    $$
    so we must have $104 D^{s(J') + 1} > 4D^{\ps(\fp)}$. By the triangle inequality again,
   $$
        B(\pc(\fp), 104 D^{s(J') + 1}) \subset B(c(J), 204 D^{s(J') + 1})\,,
    $$
    so \eqref{measure-comparison} proves that  $\fp$ satisfies \eqref{eq-pr-condition} with $r=104 D^{s(J') + 1}$.
\end{proof}

\subsection{Almost orthogonality of separated trees}

The main result of this subsection is the almost orthogonality estimate for operators associated to distinct trees in a forest in \Cref{correlation-separated-trees} below. We will deduce it from Lemmas \ref{correlation-distant-tree-parts} and \ref{correlation-near-tree-parts}, which are proven in Subsections \ref{subsec-big-tiles} and \ref{subsec-rest-tiles}, respectively.

The adjoint of the operator $T_{\fp}$ defined in \eqref{definetp} is given by
\begin{equation}
    \label{definetp*}
    T_{\fp}^* g(x) = \int_{E(\fp)} \overline{K_{\ps(\fp)}(y,x)} e(-\tQ(y)(x)+ \tQ(y)(y)) g(y) \, \mathrm{d}\mu(y)\,.
\end{equation}
For each $\fp \in \fP$, we have
\begin{equation}\label{adjoint-tile-support}
    T_{\fp}^* g = \mathbf{1}_{B(\pc(\fp), 5D^{\ps(\fp)})} T_{\fp}^* \mathbf{1}_{\scI(\fp)} g \, .
\end{equation}
With the tree localization \eqref{forest6}, we conclude for $\fu \in \fU$ and $\fp \in \fT(\fu)$
\begin{equation}\label{eq-tree-uu}
      T_{\fp}^* g = \mathbf{1}_{\scI(\fu)} T_{\fp}^* \mathbf{1}_{\scI(\fu)} g\,.
\end{equation}
\Cref{densities-tree-bound} implies that for all $\fu \in \fU$ and $g$ with $|g| \le \mathbf{1}_G$ we have
\begin{equation}    \label{adjoint-tree-control}
    \|S_{2, \fu} g\|_2 \le 2^{182a^3} \|g\|_2\, ,
\end{equation}
$$
    S_{2,\fu}g := \left|T_{\fT(\fu)}^*g \right| + M|g| + |g|\, .
$$

\begin{lemma}[correlation separated trees]
    \label{correlation-separated-trees}
    \leanok
    \lean{TileStructure.Forest.correlation_separated_trees}
    \uses{correlation-distant-tree-parts,correlation-near-tree-parts}
    For any $\fu_1 \ne \fu_2 \in \fU$ and all bounded $g_1, g_2$ with bounded support, we have
    \begin{equation}
        \label{eq-lhs-sep-tree}
        \left| \int_X T^*_{\fT(\fu_1)}g_1 \overline{T^*_{\fT(\fu_2)}g_2 }\,\mathrm{d}\mu \right|
        \le 2^{512a^3+1-4n} \prod_{j =1}^2 \| S_{2, \fu_j} g_j\|_{L^2(\scI(\fu_1) \cap \scI(\fu_2))}\,.
    \end{equation}
\end{lemma}

By \eqref{eq-tree-uu} and the dyadic property \eqref{dyadicproperty}, the left hand side of \eqref{eq-lhs-sep-tree} vanishes unless $\scI(\fu_1) \subset \scI(\fu_2)$ or $\scI(\fu_2) \subset \scI(\fu_1)$. Without loss of generality we assume $\scI(\fu_1) \subset \scI(\fu_2)$.
Defining
\begin{equation*}
     \mathfrak{S} := \{\fp \in \fT(\fu_1) \cup \fT(\fu_2) \ : \ d_{\fp}(\fcc(\fu_1), \fcc(\fu_2)) \ge 2^{Zn/2}\}\, ,
\end{equation*}
\Cref{correlation-separated-trees} then follows by combining the definition \eqref{defineZ} of $Z$ with the following two lemmas.

\begin{lemma}[correlation distant tree parts]
        \label{correlation-distant-tree-parts}
        \leanok
        \lean{TileStructure.Forest.correlation_distant_tree_parts}
        \uses{Holder-van-der-Corput,Lipschitz-partition-unity,Holder-correlation-tree,lower-oscillation-bound}
        We have for all $\fu_1 \ne \fu_2 \in \fU$ with $\scI(\fu_1) \subset \scI(\fu_2)$ and all bounded $g_1, g_2$ with bounded support
      \begin{equation}
            \label{eq-lhs-big-sep-tree}
            \Big| \int_X  T^*_{\fT(\fu_1)}g_1 \overline{T^*_{\fT(\fu_2) \cap \mathfrak{S}}g_2 }\,\mathrm{d}\mu \Big|
            \le 2^{511a^3} 2^{-Zn/(4a^2 + 2a^3)} \prod_{j =1}^2 \| S_{2, \fu_j} g_j\|_{L^2(\scI(\fu_1))}\,.
        \end{equation}
    \end{lemma}
    \begin{lemma}[correlation near tree parts]
        \label{correlation-near-tree-parts}
        \leanok
        \lean{TileStructure.Forest.correlation_near_tree_parts}
        We have for all $\fu_1 \ne \fu_2 \in \fU$ with $\scI(\fu_1) \subset \scI(\fu_2)$ and all bounded $g_1, g_2$ with bounded support
      \begin{equation}
            \label{eq-lhs-small-sep-tree}
            \Big| \int_X  T^*_{\fT(\fu_1)}g_1 \overline{T^*_{\fT(\fu_2) \setminus \mathfrak{S}}g_2 }\,\mathrm{d}\mu \Big|
            \le 2^{232a^3+21a+5} 2^{-\frac{25}{101a}Zn \kappa} \prod_{j =1}^2 \| S_{2, \fu_j} g_j\|_{L^2(\scI(\fu_1))}\,.
        \end{equation}
    \end{lemma}

In the proofs of both lemmas, we will need the following observation.

\begin{lemma}[overlap implies distance]
    \label{overlap-implies-distance}
    \leanok
    \lean{TileStructure.Forest.��_subset_��₀, TileStructure.Forest.overlap_implies_distance}
    Let $\fu_1 \ne \fu_2 \in \fU$ with $\scI(\fu_1) \subset \scI(\fu_2)$. If $\fp \in \fT(\fu_1) \cup \fT(\fu_2)$ with $\scI(\fp) \cap \scI(\fu_1) \ne \emptyset$, then $\fp \in \mathfrak{S}$. In particular, we have $\fT(\fu_1) \subset \mathfrak{S}$.
\end{lemma}

\begin{proof}
    \leanok
    Suppose first that $\fp \in \fT(\fu_1)$. Then $\scI(\fp) \subset \scI(\fu_1) \subset \scI(\fu_2)$, by \eqref{forest1}. We conclude $\fp \in \mathfrak{S}$ as follows, where we use  the separation condition \eqref{forest5}, the squeezing property \eqref{eq-freq-comp-ball}, \eqref{forest1}, and $Z= 2^{12a}\ge 4$:
    \begin{multline*}
        d_{\fp}(\fcc(\fu_1), \fcc(\fu_2)) \ge d_{\fp}(\fcc(\fp), \fcc(\fu_2)) - d_{\fp}(\fcc(\fp), \fcc(\fu_1))\\
        \ge 2^{Z(n+1)} - 4\ge 2^{Zn/2}\,.
    \end{multline*}
    Suppose now that $\fp \in \fT(\fu_2)$. If $\scI(\fp) \subset \scI(\fu_1)$, then the same argument as above with $\fu_1$ and $\fu_2$ swapped shows $\fp \in \mathfrak{S}$. If $\scI(\fp) \not \subset \scI(\fu_1)$ then, by the dyadic property \eqref{dyadicproperty}, $\scI(\fu_1) \subset \scI(\fp)$. Pick $\fp' \in \fT(\fu_1)$, then we have $\scI(\fp') \subset \scI(\fu_1) \subset \scI(\fp)$. We conclude $\fp \in \mathfrak{S}$ as follows, using the monotonicity \Cref{monotone-cube-metrics} and a similar computation as the previous display,
    \begin{equation*}
        d_{\fp}(\fcc(\fu_1), \fcc(\fu_2)) \ge d_{\fp'}(\fcc(\fu_1), \fcc(\fu_2)) \ge 2^{Zn/2}\,.\qedhere
    \end{equation*}
\end{proof}

\subsection{Tiles with large separation}
    \label{subsec-big-tiles}

\Cref{correlation-distant-tree-parts} follows from an application of the van der Corput  \Cref{Holder-van-der-Corput} that we will elaborate  in \Cref{subsubsec-van-der-corput}. To prepare this application, we construct in \Cref{subsubsec-pao} a suitable partition of unity, and show in \Cref{subsubsec-holder-estimates} the H\"older estimates needed to apply \Cref{Holder-van-der-Corput}.

\subsubsection{A partition of unity}
\label{subsubsec-pao}
    Define
    $$
        \mathcal{J}' = \{J \in \mathcal{J}(\mathfrak{S}) \ : \ J \subset \scI(\fu_1)\}\,.
    $$
    This is a partition of $\scI(\fu_1)$. The definition of $\mathcal{J}$ implies that if two balls
    \begin{equation*}
        B(J) := B(c(J), 8D^{s(J)})
    \end{equation*}
    and $B(J')$ intersect, then $|s(J) - s(J')| \le 1$. Therefore, by standard arguments, we obtain the partition of unity of the following lemma.

    \begin{lemma}[Lipschitz partition unity]
        \label{Lipschitz-partition-unity}
        \leanok
        \lean{TileStructure.Forest.sum_χ, TileStructure.Forest.χ_le_indicator,
          TileStructure.Forest.dist_χ_le}
        There exists a family of functions $\chi_J$, $J \in \mathcal{J}'$ such that \begin{equation*}
            \mathbf{1}_{\scI(\fu_1)} = \sum_{J \in \mathcal{J}'} \chi_J\,,
        \end{equation*}
        and for all $J \in \mathcal{J}'$ and all $y,y' \in \scI(\fu_1)$
      \begin{equation*}
            0 \leq \chi_J(y) \leq \mathbf{1}_{B(J)}(y)\,,
        \end{equation*}
      \begin{equation*}
            |\chi_J(y) - \chi_J(y')| \le 2^{227a^3} \frac{\rho(y,y')}{D^{s(J)}}\,.
        \end{equation*}
    \end{lemma}

\subsubsection{H\"older estimates for adjoint tree operators}
\label{subsubsec-holder-estimates}
    Let $g_1, g_2:X \to \mathbb{C}$ be bounded with bounded support.
    Define for $J \in \mathcal{J}'$
    \begin{equation*}
        h_J(y) := \chi_J(y)\cdot(e(\fcc(\fu_1)(y)) T_{\fT(\fu_1)}^* g_1(y)) \cdot \overline{(e(\fcc(\fu_2)(y)) T_{\fT(\fu_2) \cap \mathfrak{S}}^* g_2(y))}\,.
    \end{equation*}
    The  following main $\tau$-H\"older estimate for $h_J$ holds with $\tau = 1/a$.
    We use the notation
    \begin{equation*}
        B(J) := B(c(J), 8D^{s(J)}) \qquad \text{and} \qquad
     B^\circ{}(J) := B(c(J), \frac{1}{8}D^{s(J)})\, .
    \end{equation*}

    \begin{lemma}[Holder correlation tree]
        \label{Holder-correlation-tree}
        \leanok
        \lean{TileStructure.Forest.holder_correlation_tree}
        \uses{global-tree-control-2}
        We have for all $J \in \mathcal{J}'$ that
      \begin{equation*}
            \|h_J\|_{C^{\tau}(2B(J))} \le 2^{485a^3} \prod_{j = 1,2} (\inf_{B^\circ{}(J)} |T_{\fT(\fu_j)}^* g_j| + \inf_J M |g_j|)\,.
        \end{equation*}
    \end{lemma}

 \begin{proof}
 This lemma follows by combining the upper bounds and Hölder bounds from \Cref{Lipschitz-partition-unity} for the first factor of $h_J$ and from \Cref{global-tree-control-1} and \Cref{global-tree-control-2} below for the last two factors.
\end{proof}

    \begin{lemma}[Holder correlation tile]
        \label{Holder-correlation-tile}
        \leanok
        \lean{TileStructure.Forest.holder_correlation_tile}
        \uses{adjoint-tile-support}
        Let $\fu \in \fU$ and $\fp \in \fT(\fu)$. Then for all $y, y' \in X$ and all bounded $g$ with bounded support, we have
        \begin{equation}
        \label{eq-tp-minus-tp}
            |e(\fcc(\fu)(y)) T_{\fp}^* g(y) - e(\fcc(\fu)(y')) T_{\fp}^* g(y')|
        \end{equation}
        \begin{equation*}
            \le \frac{2^{128a^3}}{\mu(B(\pc(\fp), 4D^{\ps(\fp)}))} \left(\frac{\rho(y, y')}{D^{\ps(\fp)}}\right)^{1/a} \int_{E(\fp)} |g(x)| \, \mathrm{d}\mu(x)\,.
        \end{equation*}
    \end{lemma}

    \begin{proof}
        \leanok
       We will assume $y, y' \in B(\pc(\fp), 5D^{\ps(\fp)})$. Else at least one summand on the left of \eqref{eq-tp-minus-tp} vanishes, and the proof is easy. Then we have $\rho(y,y') \le 10D^{\ps(\fp)}$.
        We estimate the left hand side of \eqref{eq-tp-minus-tp} by
        \begin{equation*}
            \int_{E(\fp)} |g(x)\overline{K_{\ps(\fp)}(x, y)}|
            |e( \fcc(\fu)(y) - \fcc(\fu)(y')-\tQ(x)(y) + \tQ(x)(y')) - 1|
        \, \mathrm{d}\mu(x) 
        \end{equation*}
      \begin{equation}
            + \int_{E(\fp)} |g(x)| |\overline{K_{\ps(\fp)}(x, y)} - \overline{K_{\ps(\fp)}(x, y')} |\, \mathrm{d}\mu(x)\,.\label{T*Holder1}
        \end{equation}
         Let $k \in \mathbb{Z}$ be such that $2^{ak}\rho(y,y') \le 10D^{\ps(\fp)}$ but $2^{a(k+1)} \rho(y,y') > 10D^{\ps(\fp)}$.
        In particular, $k \ge 0$.
        By the oscillation estimate \eqref{osccontrol},
         followed by the doubling properties
         \eqref{seconddb} and \eqref{firstdb},
        we have
        \begin{equation*}
            | \fcc(\fu)(y) - \fcc(\fu)(y')-\tQ(x)(y) + \tQ(x)(y') |
            \le d_{B(y, 1.6\rho(y,y'))}(\tQ(x), \fcc(\fu))
        \end{equation*}
        $$
            \le 2^{6a - k} d_{\fp}(\tQ(x), \fcc(\fu))
            \le 5 \cdot 2^{6a - k} \le 10 \cdot 2^{6a} \left(\frac{\rho(y,y')}{10 D^{\ps(\fp)}}\right)^{1/a}\,.
        $$
        Together with the kernel bound \eqref{eq-Ks-size} and the doubling property \eqref{doublingx} this gives the needed estimate for the first summand in \eqref{T*Holder1}. The second summand  is estimated similarly using the kernel regularity \eqref{eq-Ks-smooth} and \eqref{doublingx}.
    \end{proof}

    \begin{lemma}[limited scale impact]
        \label{limited-scale-impact}
        \leanok
        \lean{TileStructure.Forest.limited_scale_impact}
        \uses{overlap-implies-distance}
        Let $\fp \in \fT(\fu_2) \setminus \mathfrak{S}$, $J \in \mathcal{J}'$ and suppose that
        \begin{equation}
         \label{eq-b-bo}
         B(\scI(\fp)) \cap B^\circ(J) \ne \emptyset\,.
        \end{equation}
        Then
        \begin{equation}
            \label{eq-limited-scales}
         s(J) \le \ps(\fp) \le s(J) +3\,.
         \end{equation}
    \end{lemma}

    \begin{proof}
        \leanok
        For the first inequality in \eqref{eq-limited-scales}, assume to get a contradiction that $\ps(\fp) <  s(J)$
        . Since $\fp \notin \mathfrak{S}$, we have by \Cref{overlap-implies-distance} that $\scI(\fp) \cap \scI(\fu_1) = \emptyset$.
        Since $B(c(J), \frac{1}{4} D^{s(J)}) \subset \scI(J) \subset \scI(\fu_1)$, this implies
        $$
            \rho(c(J), \pc(\fp)) \ge \frac{1}{4}D^{s(J)}\,.
        $$
        On the other hand by our assumption
        $$
            \rho(c(J), \pc(\fp)) \le \frac{1}{8} D^{s(J)} + 8 D^{\ps(\fp)}\,.
        $$
        Thus $D^{\ps(\fp)} \ge 64^{-1} D^{s(J)}$, contradicting the definition \eqref{defineD} of $D$ and $a \ge 4$.

        For the second inequality in \eqref{eq-limited-scales}, assume to get a contradiction that $\ps(\fp) > s(J) + 3$. Let $J' \in \mathcal{D}$ with $J \subset J'$ and $s(J') = s(J) + 1$, and $\fp' \in \mathfrak{S}$ such that $\scI(\fp') \subset B(c(J'), 100 D^{s(J) + 2})$. By \eqref{eq-b-bo} and the triangle inequality,
        \begin{equation}
            \label{eq-ball-incl-617}
            B(c(J'), 100 D^{s(J) + 3}) \subset B(\pc(\fp), 10 D^{\ps(\fp)})\,.
        \end{equation}
        Using the definition of $\mathfrak{S}$, we have
        $$
            2^{Zn/2} \le d_{\fp'}(\fcc(\fu_1), \fcc(\fu_2)) \le d_{B(c(J'), 100 D^{s(J) + 2})}(\fcc(\fu_1), \fcc(\fu_2))\,.
        $$
        By the doubling property \eqref{seconddb} and \eqref{eq-ball-incl-617} and the definition of $\mathfrak{S}$, this is
        $$
            \le 2^{-100a} d_{B(\pc(\fp), 10 D^{\ps(\fp)})}(\fcc(\fu_1), \fcc(\fu_2))
            \le 2^{-94a} d_{\fp}(\fcc(\fu_1), \fcc(\fu_2)) \le 2^{Zn/2-94a}.
        $$
        This is a contradiction, hence the second inequality in \eqref{eq-limited-scales} follows.
    \end{proof}

    \begin{lemma}[local tree control]
        \label{local-tree-control}
        \leanok
        \lean{TileStructure.Forest.local_tree_control}
        \uses{limited-scale-impact}
        For all $J \in \mathcal{J}'$ and all bounded $g$ with bounded support,
        \begin{equation}\label{eq-local-tree-control}
            \sup_{B^\circ{}(J)} |T_{\mathfrak{T}(\mathfrak{u}_2)\setminus\mathfrak{S}}^* g| \le 2^{104a^3} \inf_J M|g|\, .
        \end{equation}
    \end{lemma}

    \begin{proof}
        \leanok
        Since $T_{\fp}^*$ is supported on $ B(\pc(\fp), 5D^{\ps(\fp)})$, the triangle inequality and \Cref{limited-scale-impact} bound the left hand side of \eqref{eq-local-tree-control} by
        \begin{equation}
            \label{eq-sep-tree-aux-3}
            \sup_{B^\circ{}(J)} \sum_{\substack{\fp \in \fT(\fu_2) \setminus \mathfrak{S}\\ B(\scI(\fp)) \cap B^\circ(J) \ne \emptyset}} |T_{\fp}^*g| \le \sum_{s = s(J)}^{s(J) + 3} \sum_{\substack{\fp \in \fP, \ps(\fp) = s\\ B(\scI(\fp)) \cap B^\circ(J) \ne \emptyset}} \sup_{B^\circ{}(J)} |T_{\fp}^* g|\,.
        \end{equation}
        If $x \in E(\fp)$ and $B(\scI(\fp)) \cap B^\circ(J) \ne \emptyset$, then
        $$
            B(c(J), 16D^{\ps(\fp)}) \subset B(x, 32 D^{\ps(\fp)})\,.
        $$
        Together with the doubling property \eqref{doublingx} and the kernel bounds \eqref{eq-Ks-size} we bound \eqref{eq-sep-tree-aux-3} by
        $$
            2^{103a^3}\sum_{s = s(J)}^{s(J) + 3} \sum_{\substack{\fp \in \fP, \ps(\fp) = s\\B(\scI(\fp)) \cap B^\circ(J) \ne \emptyset}} \frac{1}{\mu(B(c(J), 16 D^s)} \int_{E(\fp)} |g| \, \mathrm{d}\mu\,.
        $$
        Since the $E(\fp)$ in the inner sum are pairwise disjoint and contained in $B(c(J), 16 D^{\ps(\fp)})$, the last display is bounded by
        $$
            2^{103a^3}\sum_{s = s(J)}^{s(J) + 3} \frac{1}{\mu(B(c(J), 16 D^s))} \int_{B(c(J), 16 D^s)} |g| \, \mathrm{d}\mu
            \le \inf_{x' \in J} 2^{103a^3 +2} M|g|\,.
        $$
    \end{proof}

    \begin{lemma}[global tree control 1]
        \label{global-tree-control-1}
        \leanok
        \lean{TileStructure.Forest.global_tree_control1_edist_left, TileStructure.Forest.global_tree_control1_edist_right, TileStructure.Forest.global_tree_control1_supbound}
        Let $\fC_1 = \fT(\fu_1)$ and $\fC_2 = \fT(\fu_2) \cap \mathfrak{S}$. Then for $i = 1,2$ and  $J \in \mathcal{J}'$ and  bounded $g$ with bounded support, we have
        \begin{align}
            \label{TreeUB}
            \sup_{2B(J)} |T_{\fC_i}^*g| \leq \inf_{B^\circ{}(J)} |T^*_{\fC_i} g| + 2^{128a^3+4a+3} \inf_{J}  M|g|
        \end{align}
        and for all $y,y' \in 2B(J)$
        \begin{equation}
            \label{TreeHolder}
             |e(\fcc(\fu_i)(y)) T_{\fC_i}^* g(y) - e(\fcc(\fu_i)(y')) T_{\fC_i}^* g(y')|\end{equation}
             \begin{equation*}\le 2^{128a^3+4a+1} \Big(\frac{\rho(y,y')}{D^{s(J)}}\Big)^{\frac{1}{a}} \inf_J M|g|\,.
        \end{equation*}
    \end{lemma}

    \begin{proof}
        \leanok
        Note that \eqref{TreeUB} follows from \eqref{TreeHolder}. By the triangle inequality, \Cref{adjoint-tile-support} and \Cref{Holder-correlation-tile}, we have for all $y, y' \in 2B(J)$
        \begin{equation*}
            |e(\fcc(\fu_i)(y)) T_{\fC_i}^* g(y) - e(\fcc(\fu_i)(y')) T_{\fC_i}^* g(y')|
        \end{equation*}
        \begin{equation*}
            \leq \sum_{\substack{\fp \in \fC_i\\ B(\scI(\fp)) \cap 2B(J) \neq \emptyset}} |e(\fcc(\fu_i)(y)) T_{\fp}^* g(y) - e(\fcc(\fu_i)(y')) T_{\fp}^* g(y')|
        \end{equation*}
        \begin{equation*}
            \le 2^{128a^3}\rho(y,y')^{1/a} \sum_{\substack{\fp \in \fC_i\\ B(\scI(\fp)) \cap 2B(J) \neq \emptyset}} \frac{D^{- \ps(\fp)/a}}{\mu(B(\pc(\fp), 4D^{\ps(\fp)}))} \int_{E(\fp)} |g| \, \mathrm{d}\mu\,.
        \end{equation*}
        For tiles $\fp \in \fC_i$ with $B(\scI(\fp)) \cap 2B(J) \neq \emptyset$ and $\ps(\fp) < s(J)$, we have $\scI(\fp) \subset B(c(J), 100 D^{s(J) + 1})$. Since $\fp \in \fC_i \subset \mathfrak{S}$, it follows from the definition of $\mathcal{J}'$ that $s(J) = -S$, which contradicts $\ps(\fp) < s(J)$.
          Further, for each $s \ge s(J)$, the sets $E(\fp)$ for $\fp \in \fP$ with $\ps(\fp) = s$ are pairwise disjoint and contained in $B(c(J), 32D^{s})$. With \eqref{doublingx}, we then estimate the previous display by
        $$
            \le 2^{128a^3}\rho(y,y')^{1/a} \sum_{S \ge s \ge s(J)} D^{-s/a} \frac{2^{4a}}{\mu(B(c(J), 32D^{s}))} \int_{B(c(J), 32D^{s})} |g| \, \mathrm{d}\mu
        $$
        \[
            \le 2^{128a^3+4a + 1} \Big(\frac{\rho(y,y')}{D^{s(J)}}\Big)^{1/a} \inf_J M|g|. \qedhere
        \]
    \end{proof}

    Combining \Cref{global-tree-control-1} and \Cref{local-tree-control} also proves the following lemma.
    \begin{lemma}[global tree control 2]
        \label{global-tree-control-2}
        \leanok
        \lean{TileStructure.Forest.global_tree_control2}
        \uses{global-tree-control-1, local-tree-control}
        We have for all $J \in \mathcal{J}'$ and all bounded $g$ with bounded support
        $$
            \sup_{2B(J)} |T^*_{\fT(\fu_2) \cap \mathfrak{S}} g| \le \inf_{B^\circ{}(J)} |T^*_{\fT(\fu_2)} g| + 2^{129a^3+4a+4} \inf_{J} M|g|\,.
        $$
    \end{lemma}


\subsubsection{The van der Corput estimate}
\label{subsubsec-van-der-corput}
    \begin{lemma}[lower oscillation bound]
        \label{lower-oscillation-bound}
        \leanok
        \lean{TileStructure.Forest.lower_oscillation_bound}
        \uses{overlap-implies-distance}
        For all $J \in \mathcal{J}'$, we have that
        $$
            d_{B(J)}(\fcc(\fu_1), \fcc(\fu_2)) \ge 2^{-201a^3} 2^{Zn/2}\,.
        $$
    \end{lemma}

    \begin{proof}
    \leanok
    Let $J'$ be the parent cube of $J$. By definition of $\mathcal{J'}$ and the triangle inequality, there exists $\fp \in \mathfrak{S}$ such that
    $$
        \scI(\fp) \subset B(c(J'), 100 D^{s(J') + 1}) \subset B(c(J), 128 D^{s(J)+2})\,.
    $$
    Thus, by definition of $\mathfrak{S}$:
    \begin{align*}
        2^{Zn/2} \le d_{\fp}(\fcc(\fu_1), \fcc(\fu_2)) \le d_{B(c(J), 128 D^{s(J)+2})}(\fcc(\fu_1), \fcc(\fu_2))\,.
    \end{align*}
    The lemma follows using the doubling property \eqref{firstdb} and $a \ge 4$.
    \end{proof}

    \begin{proof}[Proof of \Cref{correlation-distant-tree-parts}]
    \proves{correlation-distant-tree-parts}
    By the triangle inequality, the left hand side of \eqref{eq-lhs-big-sep-tree} is at most
    $$
        \le \sum_{J \in \mathcal{J}'} \Big|\int_{B(J)} e(\fcc(\fu_2)(y) - \fcc(\fu_1)(y)) h_J(y) \, \mathrm{d}\mu(y) \Big|\,.
    $$
    The van der Corput \Cref{Holder-van-der-Corput}  estimates this by
    $$
        \le 2^{7a} \sum_{J \in \mathcal{J}'} \mu(2B(J)) \|h_J\|_{C^{\tau}(B(J))} (1 + d_{B(J)}(\fcc(\fu_1), \fcc(\fu_2)))^{-1/(2a^2+a^3)}\,.
    $$
    With \Cref{Holder-correlation-tree}, \Cref{lower-oscillation-bound} and $a \ge 4$,  the last display is estimated by
    \begin{equation*}
        \le 2^{485a^3+201} 2^{-Zn/(4a^2 + 2a^3)} \sum_{J \in \mathcal{J}'} \mu(B(J))
        \prod_{j=1}^2 (\inf_{B^\circ{}(J)} |T_{\fT(\fu_j)}^* g_j| + \inf_J M |g_j|)\,.
    \end{equation*}
    Using the doubling property \eqref{doublingx}, a summand with fixed $J$ is controlled by
    $$
        2^{6a} \int_J \prod_{j=1}^2 ( |T_{\fT(\fu_j)}^* g_j|(x) + M |g_j|(x)) \, \mathrm{d}\mu(x)\,.
    $$
    The lemma follows by summing over $J \in \mathcal{J}'$ and using Cauchy-Schwarz.
    \end{proof}

\subsection{The remaining tiles}
    \label{subsec-rest-tiles}
    Differently from the previous subsection, define
    $$
        \mathcal{J}' := \{J \in \mathcal{J}(\fT(\fu_1)) \, : \, J \subset \scI(\fu_1)\}\,.
    $$
    The collection $\mathcal{J}'$  is a partition of $\scI(\fu_1)$.

In this section, we prove
\Cref{correlation-near-tree-parts}.
    By \Cref{tree-projection-estimate} and the support property \eqref{adjoint-tile-support}, we estimate the left side of \eqref{eq-lhs-small-sep-tree} by
    \begin{align*}
        \le  2^{130a^3} \|g_1\mathbf{1}_{\scI(\fu_1)}\|_2 \|P_{\mathcal{J}'}|T_{\fT(\fu_2) \setminus \mathfrak{S}}^* g_2|\|_2.
    \end{align*}
    This reduces \Cref{correlation-near-tree-parts} to the following lemma.

    \begin{lemma}[bound for tree projection]
        \label{bound-for-tree-projection}
        \leanok
        \lean{TileStructure.Forest.bound_for_tree_projection}
        We have
        \begin{equation}
        \label{eq-bound-tree-proj}
              \|P_{\mathcal{J}'}|T_{\fT(\fu_2) \setminus \mathfrak{S}}^* g_2|\|_2
            \le 2^{102a^3+21a+5} 2^{-\frac{25}{101a}Zn\kappa} \|\mathbf{1}_{\scI(\fu_1)} M |g_2|\|_2.
    \end{equation}
    \end{lemma}

    \begin{proof}
    \proves{bound-for-tree-projection}
    Expanding the definition of $P_{\mathcal{J}'}$, we have for the left side of \eqref{eq-bound-tree-proj}
    \begin{equation*}
        \Big(\sum_{J \in \mathcal{J}'} \frac{1}{\mu(J)} \Big|\int_J \sum_{\fp \in \fT(\fu_2) \setminus \mathfrak{S}} T_{\fp}^* g_2 \, \mathrm{d}\mu(y) \Big|^2 \Big)^{1/2}
    \end{equation*}
    \begin{equation}
    \label{eq-sep-tree-small-1}
        \le \sum_{s \ge s_1} \Bigg( \sum_{J \in \mathcal{J}'} \frac{1}{\mu(J)} \Bigg|\int_J \sum_{\substack{\fp \in \fT(\fu_2) \setminus \mathfrak{S}\\ \ps(\fp) = s(J) - s\\
        J \cap B(\scI(\fp)) \ne \emptyset}} T_{\fp}^* g_2 \, \mathrm{d}\mu(y) \Bigg|^2\Bigg)^{1/2}\,.
    \end{equation}
    Here we have restricted the summation set using
    \Cref{thin-scale-impact} below with $s_1 := \frac{Zn}{202a^3}$ and used
    that by the support property \eqref{adjoint-tile-support}, the integral over $J$ vanishes if $J \cap B(\scI(\fp)) = \emptyset$. We also used Minkowski's inequality.

    Since for each $I \in \mathcal{D}$ the sets $E(\fp)$ with $\fp \in \fP(I)$ are disjoint, it follows from the doubling property \eqref{doublingx} and the kernel upper bound \eqref{eq-Ks-size} that
    $$
        \bigg| \int_J \sum_{\substack{\fp \in \fT(\fu_2) \setminus \mathfrak{S}\\ \scI(\fp) = I\\
        J \cap B(\scI(\fp)) \ne \emptyset}} T_{\fp}^* g_2 \, \mathrm{d}\mu \bigg|
        \le 2^{103a^3} \int_J \mathbf{1}_{B(I)} M |g_2|  \, \mathrm{d}\mu\,.
    $$
    By \Cref{overlap-implies-distance}, we have $\scI(\fp) \cap \scI(\fu_1) = \emptyset$ for all $\fp \in \fT(\fu_2) \setminus \mathfrak{S}$.
    Thus we can estimate \eqref{eq-sep-tree-small-1} by
    $$
        2^{103a^3} \sum_{s \ge s_1} \Bigg( \sum_{J \in \mathcal{J}'} \frac{1}{\mu(J)} \Bigg|\int_J \sum_{\substack{I \in \mathcal{D}, s(I) = s(J) - s\\ I \cap \scI(\fu_1) = \emptyset\\
        J \cap B(I) \ne \emptyset}} M |g_2| \mathbf{1}_{B(I)} \, \mathrm{d}\mu \Bigg|^2\Bigg)^{\frac 1 2}\,,
    $$
    which by Cauchy-Schwarz and \Cref{square-function-count} below is bounded by
    $$
        \le 2^{103a^3} \sum_{s \ge s_1} \Big(\sum_{J \in \mathcal{J}'} \int_J (M |g_2|)^2 2^{14a+1} (8 D^{-s})^\kappa\Big)^{\frac 1 2}.
    $$
    Summing a geometric series using
    $s_1 = \frac{Zn}{202a^3}$ and using that $\mathcal{J}'$ is a partition of $\scI(\fu_1)$ completes the proof.
\end{proof}

    \begin{lemma}[thin scale impact]
        \label{thin-scale-impact}
        \leanok
        \lean{TileStructure.Forest.thin_scale_impact}
        If $\fp \in \fT(\fu_2) \setminus \mathfrak{S}$ and $J \in \mathcal{J'}$ with $B(\scI(\fp)) \cap B(J) \ne \emptyset$, then
        $$
            \ps(\fp) \le s(J) - \frac{Zn}{202a^3}\,.
        $$
    \end{lemma}

    \begin{proof}
        \leanok
        Assume to the contrary that $\ps(\fp) > s(J)  -s_1$ with $s_1:=\frac{Zn}{202a^3}$. Then
        $$
            \rho(\pc(\fp), c(J)) \le 8D^{s(J)}+8D^{\ps(\fp)} \le 16 D^{\ps(\fp) + s_1}\,.
        $$
        Let $J'$ be the parent of $J$. By definition \eqref{eq-def-jl} of $\mathcal{J}$, there is $\fp' \in \fT(\fu_1)$ with
        \begin{equation}
        \label{eq-ip-bcp}
            \scI(\fp') \subset B(c(J'), 100 D^{s(J') + 1}) \subset B(\pc(\fp), 128 D^{\ps(\fp) + s_1 + 1})\,.
        \end{equation}
        Since
        $\scI(\fu_1) \subset \scI(\fu_2)$, we have by the forest properties \eqref{forest5} and \eqref{forest1}
        $$
            d_{\fp'}(\fcc(\fu_1), \fcc(\fu_2)) > 2^{Z(n+1)} - 4 \ge 2^{Z(n+1) - 1}\,.
        $$
        It follows by \eqref{eq-ip-bcp} and the monotonicity property \eqref{monotonedb} that
        $$
            2^{Z(n+1)-1}
            \le d_{B(\pc(\fp), 128 D^{\ps(\fp) + s_1+ 1})}(\fcc(\fu_1), \fcc(\fu_2))\,.
        $$
        Using the doubling property \eqref{firstdb} and $\fp \notin \mathfrak{S}$, we estimate this further by
        $$
            \le 2^{9a + 100a^3 (s_1+2)} d_{\fp}(\fcc(\fu_1), \fcc(\fu_2))
            \le 2^{9a + 100a^3 (s_1+2)} 2^{Zn/2}\,.
        $$
    The last two displays give the following contradiction to the definition of $s_1$:
        \begin{equation*}
            Z n/2 + Z - 1 \le 9a + 100a^3(s_1 + 2)\,.\qedhere
        \end{equation*}
    \end{proof}

    \begin{lemma}[square function count]
        \label{square-function-count}
        \leanok
        \lean{TileStructure.Forest.square_function_count}
        For $J \in \mathcal{J}'$ and $s \ge 0$, we have
        \begin{equation*}
            \frac{1}{\mu(J)} \int_J \Bigg(\sum_{\substack{I \in \mathcal{D}, s(I) = s(J) - s\\ I \cap \scI(\fu_1) = \emptyset\\
        J \cap B(I) \ne \emptyset}} \mathbf{1}_{B(I)}\bigg)^2 \, \mathrm{d}\mu \le 2^{14a+1} (8 D^{-s})^\kappa\,.
        \end{equation*}
    \end{lemma}

    \begin{proof}
        \leanok
        Since $J \in \mathcal{J}'$ we have $J \subset \scI(\fu_1)$. Thus, if $B(I) \cap J \ne \emptyset$, then
    \begin{equation*}
        B(I) \cap J \subset \{x \in J \ : \ \rho(x, X \setminus J) \le 8D^{-s}D^{s(J)}\}\,.
    \end{equation*}
    Furthermore, for each $s$, the balls $B(I)$ with $s(I) = s$ have overlap at most $2^{7a}$ by the doubling property \eqref{doublingx}. Combining this with the small boundary property \eqref{eq-small-boundary}, the lemma follows.
\end{proof}

\subsection{Forests}
\label{subsec-forest}
In this subsection, we complete the proof of \Cref{forest-operator} from the results of the previous subsections.

Define an $n$-row to be an $n$-forest $(\fU, \fT)$, such that in addition the sets $\scI(\fu), \fu \in \fU$ are pairwise disjoint. By iteratively selecting the trees with inclusion maximal top tiles $\fu$, we can decompose the forest $(\fU, \fT)$ into a disjoint union of at most $2^n$ many $n$-rows
\begin{equation*}
(\fU_j, \fT_j) := (\fU_j, \fT|_{\fU_j}).
\end{equation*}
We set $\mathfrak{R}_j = \cup_{\fu \in \fU_j} \fT(\fu)$.
The support property \eqref{adjoint-tile-support} implies that  $T_{\fT(\fu)}$ with $\fu \in \fU_j$ maps the corresponding summand of the orthogonal direct sum
\[
    \bigoplus_{\fu \in \fU_j} L^2(\scI(\fu)) \subset L^2(X),
\]
into itself and annihilates all other summands. Hence, the operator norm of the operator associated to a row is the maximum of the norms of the corresponding tree operators.
With Lemma \ref{densities-tree-bound} and the density assumption \eqref{forest4} this gives the following row estimate.

\begin{lemma}[row bound]
    \label{row-bound}
    \leanok
    \lean{TileStructure.Forest.row_bound, TileStructure.Forest.indicator_row_bound}
    \uses{densities-tree-bound,adjoint-tile-support}
    For each $1 \le j \le 2^n$ and each bounded $g$ with bounded support with $|g| \le \mathbf{1}_G$,
    we have
    \begin{equation*}
        \left\|  T_{\mathfrak{R}_j}^* g\right\|_2 \le 2^{182a^3} 2^{-n/2} \|g\|_2
    \end{equation*}
    and
    \begin{equation*}
        \left\|  \mathbf{1}_F T_{\mathfrak{R}_j}^* g\right\|_2 \le 2^{283a^3} 2^{-n/2} \dens_2(\bigcup_{\fu\in \fU}\fT(\fu))^{1/2} \|g\|_2\,.
    \end{equation*}
\end{lemma}

We further have the following $TT^*$ estimate for distinct rows.

\begin{lemma}[row correlation]
    \label{row-correlation}
    \leanok
    \lean{TileStructure.Forest.row_correlation}
    \uses{adjoint-tree-control,correlation-separated-trees}
    For all $1 \le j < j' \le 2^n$ and for all $g_1, g_2$ with $|g_i| \le \mathbf{1}_G$, it holds that
    \begin{equation}
        \label{eq-row-corr}\left| \int  T^*_{\mathfrak{R}_j} g_1 \overline{T^*_{\mathfrak{R}_{j'}} g_2} \, \mathrm{d}\mu \right| \le
        2^{876a^3+1-4n}\|g_1\|_2 \|g_2\|_2\,.
    \end{equation}
\end{lemma}

\begin{proof}
    Using the support property \eqref{adjoint-tile-support} and \Cref{correlation-separated-trees} first and then Cauchy Schwarz, we bound the left hand side of \eqref{eq-row-corr} by
    \begin{equation*}
        \sum_{\fu \in \fU_j} \sum_{\fu' \in \fU_{j'}} \left| \int T^*_{\fT_j(\fu)} (\mathbf{1}_{\scI(\fu)} g_1) \overline{T^*_{\fT_{j'}(\fu')} (\mathbf{1}_{\scI(\fu')} g_2)} \, \mathrm{d}\mu \right|
    \end{equation*}
    \begin{equation*}
         \le 2^{512a^3+1-4n} \sum_{\fu \in \fU_j} \sum_{\fu' \in \fU_{j'}} \|S_{2,\fu} (\mathbf{1}_{\scI(\fu)}g_1)\|_{L^2(\scI(\fu'))} \|S_{2, \fu'} (\mathbf{1}_{\scI(\fu')}g_2)\|_{L^2(\scI(\fu))}
    \end{equation*}
    $$
        \le 2^{512a^3+1-4n} \Big(\!\sum_{\substack{\fu \in \fU_j\\\fu' \in \fU_{j'}}} \|S_{2,\fu} (\mathbf{1}_{\scI(\fu)}g_1)\|_{L^2(\scI(\fu'))}^2 \sum_{\substack{\fu \in \fU_j\\\fu' \in \fU_{j'}}} \|S_{2,\fu'} (\mathbf{1}_{\scI(\fu')}g_2)\|_{L^2(\scI(\fu))}^2 \Big)^{\frac{1}{2}}.
    $$
    By the estimate  \eqref{adjoint-tree-control} on $S_{2,\fu}$  and pairwise disjointness of the sets $\scI(\fu)$ for $\fu \in \fU_j$ and of the sets $\scI(\fu')$ for $\fu' \in \fU_{j'}$, the last display is controlled by the right hand side of \eqref{eq-row-corr}, completing the proof of the lemma.
\end{proof}

Define for $1 \le j \le 2^n$
$$
    E_j := \bigcup_{\fu \in \fU_j} \bigcup_{\fp \in \fT(\fu)} E_1(\fp)\,.
$$
The separation condition \eqref{forest5} for trees in a forest implies that the sets $E_j$ are pairwise disjoint.

\begin{proof}[Proof of \Cref{forest-operator}]
    \proves{forest-operator}
    Recalling the expression  \eqref{definetp*} for $T^*_{\fp}$, we have for each $j$
    $$
        T_{\mathfrak{R}_j}^*g = \sum_{\fu \in \fU_j} \sum_{\fp \in \fT(\fu)} T_{\fp}^* g = \sum_{\fu \in \fU_j} \sum_{\fp \in \fT(\fu)} T_{\fp}^* \mathbf{1}_{E_j} g = T_{\mathfrak{R}_j}^* \mathbf{1}_{E_j} g\,.
    $$
    Using this we can write
    $$
         \Big\|\sum_{j = 1}^{2^n} T^*_{\mathfrak{R}_{j}} g\Big\|_2^2
        = \sum_{j=1}^{2^n} \int_X |T_{\mathfrak{R}_j}^* \mathbf{1}_{E_j} g|^2 + \sum_{j =1}^{2^n} \sum_{\substack{j' = 1,j' \ne j}}^{2^n} \int_X \overline{ T_{\mathfrak{R}_j}^* \mathbf{1}_{E_j} g} T_{\mathfrak{R}_{j'}}^* \mathbf{1}_{E_{j'}} g \, \mathrm{d}\mu\,.
    $$
    We use \Cref{row-bound} to estimate each term in the first sum, and \Cref{row-correlation} to bound each term in the second sum.
    $$
        \le 2^{566a^3-n} \sum_{j = 1}^{2^n} \|\mathbf{1}_{E_j} g\|_2^2 + 2^{876a^3+1-4 n}\sum_{j=1}^{2^n}\sum_{j' = 1}^{2^n} \|\mathbf{1}_{E_j} g\|_2 \|\mathbf{1}_{E_{j'}}g\|_2\,.
    $$
    By Cauchy-Schwarz in the second two sums and disjointness of the sets $E_j$, this is at most
    $$
        2^{876a^3+1} (2^{-n} + 2^{n}2^{-4 n}) \sum_{j = 1}^n \|\mathbf{1}_{E_j} g\|_2^2 \le
        2^{876a^3+2 - n} \|g\|_2^2\,.
    $$
    Taking adjoints and square roots, it follows that for all $f$
    \begin{equation}
        \label{eq-forest-bound-1}
        \Big\|\sum_{\fu \in \fU} \sum_{\fp \in \fT(\fu)} T_{\fp} f\Big\|_2 \le 2^{439a^3-\frac{n}{2}} \|f\|_2\,.
    \end{equation}
    On the other hand, we have by disjointness of the sets $E_j$
    $$
        \Big\|\sum_{\fu \in \fU} \sum_{\fp \in \fT(\fu)} T_{\fp} f\Big\|_2 = \Big\|\sum_{j=1}^{2^n} \mathbf{1}_{E_j} T_{\mathfrak{R}_{j}} f\Big\|_2 = \Big(\sum_{j = 1}^{2^n} \|\mathbf{1}_{E_j} T_{\mathfrak{R}_{j}} f\|_2^2\Big)^{1/2}\,.
    $$
    If $|f| \le \mathbf{1}_F$ then we obtain from \Cref{row-bound}
    \begin{equation*}
        \le 2^{283a^3} \dens_2\Big(\bigcup_{\fu\in \fU}\fT(\fu)\Big)^{\frac{1}{2}} 2^{-\frac{n}{2}} \Big(\sum_{j = 1}^{2^n} \|f\|_2^2\Big)^{\frac{1}{2}} =
        \end{equation*}\begin{equation}
\label{eq-forest-bound-2}=2^{283a^3} \dens_2\Big(\bigcup_{\fu\in \fU}\fT(\fu)\Big)^{\frac{1}{2}} \|f\|_2\, .
    \end{equation}
    \Cref{forest-operator} follows by taking the product of the $(2 - \frac{2}{q})$-th power of \eqref{eq-forest-bound-1} and the $(\frac{2}{q} - 1)$-st power of \eqref{eq-forest-bound-2}.
\end{proof}

\section{Proof of the H\"older cancellative condition}
\label{liphoel}

We use the following standard approximation lemma, which we will not prove here.
Recall that $\tau = 1/a$.

\begin{lemma}[Lipschitz Holder approximation]
    \label{Lipschitz-Holder-approximation}
    \leanok
    \lean{support_holderApprox_subset, dist_holderApprox_le, iLipENorm_holderApprox}
    Let $z\in X$ and $R>0$. Let $\varphi: X \to \mathbb{C}$ be a function supported in the ball
    $B:=B(z,R)$ with finite norm $\|\varphi\|_{C^\tau(2B)}$. Let $0<t \leq 1$. There exists a function $\tilde \varphi : X \to \mathbb{C}$, supported in $B(z,2R)$, such that for every $x\in X$
    \begin{equation}\label{eq-firstt}
        |\varphi(x) - \tilde \varphi(x)| \leq t^{\tau} \|\varphi\|_{C^\tau(2B)}
    \end{equation}and
   \begin{equation}\label{eq-secondt}
       \|\tilde \varphi\|_{\Lip(B(z,2R))} \leq 2^{4a}t^{-1-a} \|\varphi\|_{C^{\tau}(2B)}\, .
   \end{equation}
\end{lemma}

We turn to the proof of \Cref{Holder-van-der-Corput}.
\begin{proof}[Proof of \Cref{Holder-van-der-Corput}]
    \proves{Holder-van-der-Corput}
Let $z\in X$ and $R>0$ and set $B=B(z,R)$. Let $\varphi$
be given as in \Cref{Holder-van-der-Corput}.
Set
\begin{equation*}
    t:=(1+d_B(\mfa,\mfb))^{-\frac{\tau}{2+a}}
\end{equation*}
and define $\tilde{\varphi}$ as in \Cref{Lipschitz-Holder-approximation}. Let $\mfa$ and $\mfb$ be in $\Mf$.
Then
\begin{equation*}
   \left|\int e(\mfa(x)-{\mfb(x)}) \varphi (x)\, \mathrm{d}\mu(x)\right|
\end{equation*}
\begin{equation}\label{eql61}
\le \left|\int e(\mfa(x)-{\mfb(x)}) \tilde{\varphi} (x)\, \mathrm{d}\mu(x)\right|
 + \left|\int e(\mfa(x)-{\mfb(x)}) (\varphi (x)-\tilde{\varphi}(x))\, \mathrm{d}\mu(x)\right|
\end{equation}
Using the cancellative condition \eqref{eq-vdc-cond} of $\Mf$ on the ball $B(z,2R)$, the first term in \eqref{eql61} is bounded above by
\begin{equation*}
    2^a \mu(B(z,2R)) \|\tilde{\varphi}\|_{\Lip(B(z,2R))} (1 + d_{B(z,2R)}(\mfa,\mfb))^{-\tau} \, .
\end{equation*}
With the  doubling condition \eqref{doublingx},
the inequality \eqref{eq-secondt}, and the monotonicity
$d_B\le d_{B(z,2R)}$, the last display is bounded above by
\begin{equation*}
     2^{6a}t^{-1-a} \mu(B) \|{\varphi}\|_{C^\tau(B)}
    (1 + d_{B}(\mfa,\mfb))^{-\tau} \, .
 \end{equation*}
The second term in \eqref{eql61} is by
\eqref{eq-firstt} bounded by
\begin{equation*}
  \mu(B(z,2R)) t^{\tau} \|\varphi\|_{C^\tau(B)}
  \le 2^a \mu(B) t^{\tau} \|\varphi\|_{C^\tau(B)}
  \,.
\end{equation*}
The proposition now follows by adding these two estimates.
\end{proof}

\printbibliography

\end{document}